\documentclass[11pt,reqno,a4paper]{amsart}

\usepackage[T1]{fontenc}
\usepackage[utf8]{inputenc}
\usepackage{lmodern}
\usepackage{microtype}
\usepackage{geometry}
\usepackage{amsmath,amssymb,amsthm,mathtools,mathrsfs,bm}
\usepackage{enumitem}
\usepackage{booktabs,array,tabularx}
\usepackage{aliascnt}
\usepackage[numbers,sort&compress]{natbib}
\usepackage[colorlinks=true,linkcolor=blue,citecolor=blue,urlcolor=blue]{hyperref}
\usepackage[nameinlink,noabbrev]{cleveref}

\allowdisplaybreaks
\setlist[itemize]{topsep=4pt,itemsep=2pt,parsep=1pt}
\setlist[enumerate]{topsep=4pt,itemsep=2pt,parsep=1pt}

\numberwithin{equation}{section}
\newtheorem{theorem}{Theorem}[section]
\newaliascnt{proposition}{theorem}
\newtheorem{proposition}[proposition]{Proposition}
\aliascntresetthe{proposition}
\newaliascnt{lemma}{theorem}
\newtheorem{lemma}[lemma]{Lemma}
\aliascntresetthe{lemma}
\newaliascnt{corollary}{theorem}
\newtheorem{corollary}[corollary]{Corollary}
\aliascntresetthe{corollary}
\newaliascnt{assumption}{theorem}
\newtheorem{assumption}[assumption]{Assumption}
\aliascntresetthe{assumption}
\theoremstyle{definition}
\newaliascnt{definition}{theorem}
\newtheorem{definition}[definition]{Definition}
\aliascntresetthe{definition}
\theoremstyle{remark}
\newaliascnt{remark}{theorem}
\newtheorem{remark}[remark]{Remark}
\aliascntresetthe{remark}

\crefname{theorem}{theorem}{theorems}
\Crefname{theorem}{Theorem}{Theorems}
\crefname{proposition}{proposition}{propositions}
\Crefname{proposition}{Proposition}{Propositions}
\crefname{lemma}{lemma}{lemmas}
\Crefname{lemma}{Lemma}{Lemmas}
\crefname{corollary}{corollary}{corollaries}
\Crefname{corollary}{Corollary}{Corollaries}
\crefname{definition}{definition}{definitions}
\Crefname{definition}{Definition}{Definitions}
\crefname{remark}{remark}{remarks}
\Crefname{remark}{Remark}{Remarks}
\crefname{assumption}{assumption}{assumptions}
\Crefname{assumption}{Assumption}{Assumptions}

\newcommand{\R}{\mathbb R}
\newcommand{\N}{\mathbb N}
\newcommand{\Z}{\mathbb Z}
\newcommand{\E}{\mathbb E}
\newcommand{\Prob}{\mathbb P}
\newcommand{\Hh}{\mathcal H}
\newcommand{\Hs}[1]{\mathcal H^{#1}}
\newcommand{\Es}[1]{\mathcal E_{#1}}

\newcommand{\Vv}{\mathcal V}

\newcommand{\cA}{\mathcal A}

\newcommand{\cE}{\mathcal E}

\newcommand{\cK}{\mathcal K}
\newcommand{\cL}{\mathcal L}

\newcommand{\cS}{\mathcal S}
\newcommand{\dd}{\mathrm d}
\newcommand{\e}{\mathrm e}
\newcommand{\Id}{\operatorname{Id}}
\newcommand{\Dom}{\operatorname{Dom}}
\newcommand{\Ran}{\operatorname{Ran}}

\newcommand{\Tr}{\operatorname{Tr}}
\newcommand{\Ent}{\operatorname{Ent}}
\newcommand{\Law}{\operatorname{Law}}
\newcommand{\supp}{\operatorname{supp}}

\newcommand{\spanop}{\operatorname{span}}
\newcommand{\one}{\mathbf 1}
\newcommand{\Wass}{\mathsf W}
\newcommand{\norm}[1]{\left\lVert #1\right\rVert}
\newcommand{\abs}[1]{\left\lvert #1\right\rvert}
\newcommand{\ip}[2]{\left\langle #1,#2\right\rangle}
\newcommand{\rank}{\operatorname{rank}}

\title[Spectral gap for 3D damped cubic NLW]{Spectral gap for the three-dimensional damped cubic wave equation with degenerate noise}

\author{Rongchang Liu}
\address{School of Mathematics, Sichuan University, Chengdu 610064, China}
\email{rcliu@scu.edu.cn}

\author{Kening Lu}
\address{School of Mathematics, Sichuan University, Chengdu 610064, China}
\email{keninglu@scu.edu.cn}

\date{}

\hypersetup{
 pdftitle={Unique Ergodicity and a Weighted Wasserstein Spectral Gap for the Three-Dimensional Damped Cubic Wave Equation with Degenerate Noise},
 pdfauthor={Rongchang Liu and Kening Lu},
 pdfsubject={Unique ergodicity, sharp saturation geometry, and a weighted Wasserstein spectral gap for a stochastic cubic wave equation},
 pdfkeywords={stochastic wave equation, finite rank noise, saturation, Fourier lattice, Malliavin calculus, unique ergodicity, Wasserstein spectral gap}
}

\thanks{This work was supported by the Fundamental Research Funds for
the Central Universities and the National Natural Science Foundation of China
(Grant Nos.~12090010 and 12090013).}

\begin{document}

\begin{abstract}
We establish a weighted Wasserstein spectral gap for the three-dimensional  
damped cubic wave equation with genuinely finite rank Brownian forcing.
Under a saturation condition, the gap holds with respect to the negative
phase topology $\mathcal E_s=H^{-s}\times H^{-1-s}$ for every
$0<s<1/2$, from which we deduce unique ergodicity and exponential mixing
in the energy topology.  Sharp geometric characterizations of the
saturation condition are also obtained.

The method we develop is a stable--compact asymptotic coupling mechanism
for hypoelliptic dissipative SPDEs beyond the parabolic setting.  Instead
of relying on positive time smoothing or asymptotic gradient estimates, it reduces the infinite-dimensional obstruction to contraction to
a compact defect in a weaker coupling topology.  Dense Malliavin range
then permits this defect to be compensated by a finite dimensional
Cameron--Martin shift, producing a finite distance contraction on
bounded Lyapunov cores.  Together with a separate high energy contraction
from dissipation, this yields a global weighted Wasserstein spectral gap.
\end{abstract}

\maketitle
\tableofcontents
\section{Introduction}

The statistical theory of fully developed turbulence initiated by
Kolmogorov's K41 theory
\cite{Kolmogorov1941a,Kolmogorov1941b,Frisch1995} and its weak wave
counterpart developed by Zakharov and collaborators
\cite{ZakharovLvovFalkovich1992,NewellRumpf2011} seek universal
statistical laws in stationary cascade regimes.  In the standard direct cascade picture,
statistically stationary cascades are sustained by a separation between
injection and dissipation scales, with energy injected at low wave numbers,
transferred through an inertial range, and eventually dissipated at high
wave numbers.  In the
weakly nonlinear regime, the wave spectrum is described at kinetic
scales by the wave kinetic equation, which has been rigorously
justified for the cubic nonlinear
Schr\"odinger equation in the works of Deng and Hani
\cite{DengHani2023,DengHani2026}. In experimental and numerical studies, damped--driven models are widely
used, with forcing and dissipation sustaining a nonequilibrium steady
state. Grande and Hani \cite{GrandeHani2026} recently established the corresponding
kinetic description for a stochastically forced and viscously dissipated
nonlinear Schr\"odinger model. However, an analogous rigorous kinetic derivation remains an active
problem for continuum nonlinear wave equations.

At the level of the underlying nonlinear dynamics, a complementary
statistical viewpoint, emphasized by Kolmogorov and recalled by Sinai
\cite{Sinai1989}, is to describe stationary turbulence regimes through invariant
measures with sufficiently strong mixing properties.  Ruelle subsequently developed this viewpoint within the dynamical systems
approach to dissipative turbulence, where turbulent statistical states are
expected to be described by SRB (Sinai-Ruelle-Bowen) measures
\cite{Ruelle1978,EckmannRuelle1985}.  This leads to the basic question
whether, in a damped--driven setting, low dimensional random forcing can select a unique invariant measure with
quantitative mixing.

We address this question for the damped cubic nonlinear wave equation.  On
$\mathbb T^3=(\R/2\pi\Z)^3$ we consider
\begin{align*}
 \dd u=v\,\dd t,\qquad \dd v=(\Delta u-a(x)v-u^3)\,\dd t+B\,\dd W_t,
\end{align*}
where $a$ is smooth and strictly positive, $W$ is a finite dimensional
Brownian motion, and the range of $B$ is generated by smooth spatial
profiles.  The precise phase space and the cubic saturation condition are
given in \Cref{sec:setting}.

Our first result is a weighted Wasserstein spectral gap under the saturation
condition.  For every $0<s<1/2$, the contraction is measured by a cost
induced from the negative phase topology
\begin{align*}
 \Es{s}=H^{-s}(\mathbb T^3)\times H^{-1-s}(\mathbb T^3).
\end{align*}
We then deduce the existence of a unique invariant probability measure and
exponential mixing in the energy topology.  We also obtain sharp geometric characterizations of the saturation
condition.  In particular, four smooth forcing profiles suffice in the
general geometric construction, while in the Fourier-mode class the
minimal saturating systems are determined exactly.

Furthermore, when the forcing generates a proper Fourier sector and the
damping preserves it, the same spectral gap and mixing conclusions hold
for the restricted dynamics.  This
in particular yields the corresponding results for the cubic stochastic wave
equation in dimensions one and two. The precise statements are given
in \Cref{thm:weighted-gap,thm:main,thm:sector-and-lower-dimensional}.

A main contribution of the paper is a new abstract
stable--compact spectral gap criterion for dissipative Markov systems with
Gaussian driver.  The Markov phase space $H$ is continuously and densely
embedded in a weaker coupling space $\mathbb H$ where contraction is
measured, while the nonlinear dynamics need not be wellposed on $\mathbb H$.
This gives a finite distance route to Wasserstein spectral gaps for
hypoelliptic systems distinct from the parabolic frameworks of
\cite{EckmannHairer2001,HairerMattingly2006,HairerMattingly2008,HairerMattingly2011,GerasimovicsHairer2019}.
For the parabolic applications in this circle of ideas, positive time
smoothing or a Foia\c{s}--Prodi mechanism, together with Malliavin
integration by parts, leads to asymptotic gradient estimates.  Our starting
point is instead a finite distance structure for the exact two point
dynamics that places parabolic smoothing and hyperbolic stability
within a common framework.

On one time block, the exact solution difference is represented by a secant
propagator
\begin{align*}
 \Phi(\mathsf y,\omega)-\Phi(\mathsf x,\omega)
 =\mathcal J_{\mathsf x,\mathsf y,\omega}(\mathsf y-\mathsf x),
 \qquad \mathcal J=S+\mathcal K,
\end{align*}
where $S$ is strictly contractive on $\mathbb H$ and $\mathcal K$ is
compact.  On compact Lyapunov driver cores, dense range of the Malliavin
derivative allows the compact defect to be compensated approximately by finitely many Cameron--Martin directions.
This compact--dense approximation was used  in
\cite{LianLiuLu2026} for the projective Navier--Stokes dynamics, where
compactness is supplied by parabolic smoothing.  A suitable finite
dimensional shift then contracts the endpoint, while a Gaussian change of variables and coupling recover the correct Markov
marginal.  Since only a fixed
contraction margin is needed, no asymptotically exact inversion of the
Malliavin covariance is required.  This yields a finite distance
Wasserstein contraction on bounded Lyapunov cores without a semigroup
gradient estimate.

The compact core contraction is complemented by a separate high energy
mechanism.  A Lyapunov structure produces a high--low \emph{energy}
decomposition in which the sufficiently high energy region is contractive
under synchronous coupling.  This is distinct from a high--low \emph{mode}
decomposition based on a Foia\c{s}--Prodi structure.  Uniform accessibility then 
turns a bounded Lyapunov set into a small set for the resulting weak
distance.  The low energy stable--compact contraction, the high energy
synchronous contraction, and the Lyapunov drift then close a weighted weak
Harris argument \cite{HairerMattinglyScheutzow2011}.
The resulting sampling time spectral gap is upgraded to continuous time by
a finite time bound for the same weighted cost.

For the wave equation, the abstract inputs arise from the following specific
structures.  First, the exact difference of two solutions satisfies a
linear wave equation with a lower order random potential.  Relative to the
exponentially stable damped wave semigroup, this produces a compact defect
on $\Es{s}$, despite the absence of positive
time smoothing. Second, the weaker coupling topology places the dual
terminal data in a regularity class for which an exact backward
zero-propagation argument closes along every energy trajectory. Brownian
separation then propagates the resulting annihilation relations through
the cubic saturation algebra, yielding dense Malliavin range.  Third, the same
saturation algebra is generated deterministically by fast controls, in the
spirit of
\cite{AgrachevSarychev2005,AgrachevSarychev2006,GlattHoltzHerzogMattingly2018},
which gives fixed time full support and, after a compactness argument,
uniform accessibility.  The remaining high energy
contraction follows directly from the dissipative energy structure and an
exponential occupation estimate. 

We next compare the result with the stochastic wave literature.  For
white in time Gaussian forcing, Martirosyan proved exponential mixing for
three dimensional damped nonlinear wave equations under spatially
nondegenerate noise \cite{Martirosyan2014}; the uniqueness and mixing
argument there assumes the nonlinearity $|f'(u)|\lesssim 1+|u|^\rho$ with $\rho<2$, and therefore excludes
cubic growth.  Tolomeo established unique ergodicity for a class of
stochastic hyperbolic equations driven by space--time white noise,
including the one dimensional wave equation \cite{Tolomeo2020}, while
Forlano--Tolomeo treated a two dimensional cubic wave equation with
additive noise slightly smoother than space--time white noise
\cite{ForlanoTolomeo2024}.

A different theory treats highly degenerate bounded randomness through
controllability.  Kuksin, Nersesyan and Shirikyan developed a general mixing theory for bounded decomposable random forces,
including forces acting on only finitely many Fourier modes
\cite{KuksinNersesyanShirikyan2020a,KuksinNersesyanShirikyan2020b}.  More
recently, \cite{LiuWeiXiangZhangZhao2024} obtained exponential mixing for
cubic nonlinear wave equations with weak dissipation and physically
localized bounded noise by combining asymptotic compactness, stability and
localized control.  The finite rank Brownian forcing considered here lies
outside this bounded noise setting: its law on each time block is Gaussian
and unbounded.  Accordingly, the finite distance correction is implemented
through Cameron--Martin quasi-invariance and Malliavin range, rather than
through transformations of bounded noise coordinates.  To our knowledge, this regime has not previously
been treated for the three dimensional cubic wave equation in the energy
phase space.

\section{Settings and main results}\label{sec:setting}

In this section we give the necessary functional setting and state the principal results of the paper in detail.  Let $\mathbb T^3=(\R/2\pi\Z)^3$ carry normalized Lebesgue measure.  We consider
\begin{align}\label{eq:main-spde}
 \dd u=v\,\dd t,\qquad \dd v=(\Delta u-a(x)v-u^3)\,\dd t+B\,\dd W_t,
\end{align}
where $W=(W^1,\ldots,W^m)$ is a standard $\R^m$-valued Brownian motion over the probability space $(\Omega,\mathcal F, (\mathcal F_t),\mathbb P)$, and $Be_j=b_j,\,1\leq j\leq m$ so that  
\begin{align*}
B\,\dd W_t=\sum_{j=1}^m b_j\,\dd W_t^j.
\end{align*}
The noise amplitudes are incorporated into the smooth profiles $b_j$.  The state space is
\begin{align*}
\Hh=H^1(\mathbb T^3)\times L^2(\mathbb T^3),\qquad \norm{(u,v)}_{\Hh}^2=\norm u_{H^1}^2+\norm v_{L^2}^2.
\end{align*}
We write $P_t$ for the Markov semigroup and let probability measures act on the right.

Set $\Lambda=(I-\Delta)^{1/2}$
and, for $s\in\R$, define the Hilbert scale
\begin{align*}
\Hs{s}=H^{1+s}(\mathbb T^3)\times H^s(\mathbb T^3),\qquad \norm{(u,v)}_{\Hs{s}}^2=\norm{\Lambda^{1+s}u}_{L^2}^2+\norm{\Lambda^sv}_{L^2}^2.
\end{align*}
Thus $\Hs{0}=\Hh$.  For every $s\in\R$, we shall repeatedly use the unitary phase identification
\begin{align}\label{eq:phase-identification}
 \mathfrak I:\Hs{s}\longrightarrow H^s(\mathbb T^3)\times H^s(\mathbb T^3),
 \qquad
 \mathfrak I(u,v)=(\Lambda u,v).
\end{align}
Throughout the paper,
\begin{align}\label{eq:damping-assumption}
 a\in C^\infty(\mathbb T^3),\qquad 0<a_0\le a(x)\le a_1<\infty,
\end{align}
where $a_0,a_1$ are constants and we will denote by $X_t=(u_t,v_t)$ the solution to \eqref{eq:main-spde}. We denote initial data in $\Hh$ by $\mathsf x$ to distinguish them from the spatial variable $x\in\mathbb T^3$.

The natural energy and modified energy are
\begin{align}\label{eq:E-energy}
\cE(u,v)&=\frac12\norm{\nabla u}_{L^2}^2+\frac12\norm v_{L^2}^2
            +\frac14\norm u_{L^4}^4,\\\label{eq:V-energy}
 \Vv(u,v)&=1+\cE(u,v)+\eta\ip uv_{L^2}
 +\frac\eta2\int_{\mathbb T^3}a(x)u(x)^2\,\dd x,
\end{align}
where $\eta>0$ is fixed sufficiently small chosen as in
\Cref{lem:periodic-energy-coercivity}.  The modified energy is equivalent to
$1+\norm u_{H^1}^2+\norm v_{L^2}^2+\norm u_{L^4}^4$ and satisfies a
Foster drift up to an additive constant; see
\Cref{lem:periodic-energy-coercivity} and \Cref{prop:lyapunov}.

\subsection{Cubic saturation}
We state the saturation condition here, which is the only remaining
assumption in the main theorems besides the regularity of the forcing
profiles and the damping condition \eqref{eq:damping-assumption}.

\begin{definition}[Cubic saturation]\label{def:saturation}
Assume
\begin{align}\label{eq:smooth-profiles}
 b_j\in C^\infty(\mathbb T^3;\R),\qquad 1\le j\le m.
\end{align}
Set $\cS_0=\{b_1,\ldots,b_m\}$
and define recursively
\begin{align}\label{eq:saturation-closure}
 \cS_{n+1}
 =\cS_n\cup
 \{\phi b_jb_k:\phi\in\cS_n,\ 1\le j,k\le m\}.
\end{align}
Finally, set
\begin{align}\label{eq:saturation-space}
 \cS_\infty
 =\spanop_{\R}\biggl(\bigcup_{n\ge0}\cS_n\biggr).
\end{align}
The forcing profiles are called \emph{cubically saturating} if
\begin{align}\label{eq:saturation}
 \overline{\cS_\infty}^{\,H^1(\mathbb T^3)}
 =H^1(\mathbb T^3).
\end{align}
\end{definition}
\begin{remark}
Sharp geometric characterizations of the saturation condition are given
in \Cref{section:geometric-saturation}.  For general smooth forcing, four
profiles suffice for saturation, and this number is optimal for the
general geometric construction developed there; see
\eqref{eq:four-profile-embedding}.  In the Fourier-mode class, the
smallest saturating system containing the constant profile consists of
seven real Fourier profiles; an example is
\begin{align*}
 B\,\dd W_t
 =\dd W_t^0+\sum_{i=1}^3
 \bigl(\cos x_i\,\dd W_t^{i,c}+\sin x_i\,\dd W_t^{i,s}\bigr).
\end{align*}
Without the constant profile, eight real profiles are necessary and
sufficient; see \Cref{cor:minimal-pure-mode-systems}.  In dimensions one
and two, the corresponding minima among Fourier-mode systems containing
the constant profile are three and five real profiles, respectively; see
\Cref{cor:lower-dimensional-unique-ergodicity}.
\end{remark}

\subsection{Main results}
To keep the present section focused on the concrete equation,
we defer the precise formulation and proof of the abstract stable--compact spectral
gap criterion to
\Cref{sec:spectral-gap-criterion}.

Fix $0<s<1/2$ and set
\begin{align*}
 \Es{s}=H^{-s}(\mathbb T^3)\times H^{-1-s}(\mathbb T^3)=\Hs{-1-s}.
\end{align*}
Define the following distance-like function on $\mathcal H$,
\begin{align}\label{gap:eq:basic-cost}
 d_0(\mathsf x,\mathsf y)=1\wedge
 \left[
  \frac{\norm{\mathsf x-\mathsf y}_{\Es{s}}}{r_0}
  \exp\{\chi(\Vv(\mathsf x)+\Vv(\mathsf y))\}
 \right]^\alpha, \quad \mathsf x,\mathsf y\in \mathcal H, 
\end{align}
where $\Vv\ge1$ is the modified wave energy defined in \eqref{eq:V-energy} and $r_0,\alpha,\chi>0$ are constants that will be chosen in the proof.  

For any nonnegative lower semicontinuous distance-like function $\delta$ and probability measures $\mu_1,\mu_2$ on $\Hh$, we write
\begin{align}\label{gap:eq:Wasserstein-definition}
 \Wass_\delta(\mu_1,\mu_2)
 =\inf_{\Gamma\in\mathsf C(\mu_1,\mu_2)}
   \int_{\Hh\times\Hh}\delta(\mathsf x,\mathsf y)\,\Gamma(\dd \mathsf x,\dd \mathsf y),
\end{align}
where $\mathsf C(\mu_1,\mu_2)$ denotes the set of couplings.  

Our first main result is a weighted Wasserstein spectral gap for a cost
built from the negative phase topology $\mathcal E_s$.

\begin{theorem}\label{thm:weighted-gap}
Assume \eqref{eq:damping-assumption}, \eqref{eq:smooth-profiles}, and the cubic saturation condition \eqref{eq:saturation}.  For every $0<s<1/2$ there exist  constants $\alpha,r_0,\chi,\beta,\lambda>0$, and $C,     \varrho>0$ such that, with
\begin{align}\label{gap:eq:weighted-cost}
 d(\mathsf x,\mathsf y)
 =\sqrt{d_0(\mathsf x,\mathsf y)
 \left(1+\beta\e^{\lambda\Vv(\mathsf x)}
          +\beta\e^{\lambda\Vv(\mathsf y)}\right)},
\end{align}
one has
\begin{align}\label{gap:eq:skeleton-gap}
 \Wass_d(\mu_1P_{t},\mu_2P_{t})
 \le C\e^{-\varrho t}\Wass_d(\mu_1,\mu_2),
 \qquad t\ge0,
\end{align}
for all probability measures with finite $\e^{\lambda\Vv}$ moment.
\end{theorem}

The contraction \eqref{gap:eq:skeleton-gap} is the weighted Wasserstein spectral gap.  It is a spectral gap for a cost built from the negative phase $\Es{s}$. 

Our second main result below shows that the system has a unique invariant measure that is exponentially mixing on the energy space $\mathcal H$, which is obtained from the negative phase spectral gap by interpolation and stationary Sobolev regularity gain.  

\begin{theorem}\label{thm:main}
Assume \eqref{eq:damping-assumption}, \eqref{eq:smooth-profiles}, and the cubic saturation condition \eqref{eq:saturation}.  Then the Markov semigroup $P_t$ generated by \eqref{eq:main-spde} has a unique invariant probability measure $\mu$ with $\supp\mu=\Hh$, and there is $\lambda_*>0$ for which
\begin{align}\label{eq:main-exp-moment}
 \int_{\Hh}\exp\{\lambda_*\Vv(\mathsf x)\}\,\mu(\dd \mathsf x)<\infty.
\end{align}
Furthermore, for every $\frac12<s_*<1$, every invariant probability measure $\nu$ satisfies
\begin{align}\label{eq:main-strong-support}
 \nu(\Hs{s_*})=1,
 \qquad
 \int_{\Hs{s_*}}\norm{\mathsf x}_{\Hs{s_*}}^p\,\nu(\dd \mathsf x)<\infty,
 \qquad 1\le p<\infty.
\end{align}
There are $0<\eta<1$ and $C,\varrho,\lambda>0$ such that, for
\begin{align*}
 d_{\Hh,\eta}(\mathsf x,\mathsf y)=1\wedge\norm{\mathsf x-\mathsf y}_{\Hh}^{\eta},
\end{align*}
one has
\begin{align}\label{gap:eq:energy-topology-gap}
 \Wass_{d_{\Hh,\eta}}(P_t(\mathsf x,\cdot),\mu)
 \le C\e^{-\varrho t}\bigl(1+\e^{\lambda\Vv(\mathsf x)}\bigr),
 \qquad \mathsf x\in\Hh,\quad t\ge0.
\end{align}
\end{theorem}

Our third main result  shows that, in the Fourier-mode forcing case, the same conclusions hold on the exact invariant Fourier sector generated by the forcing.  Let $\cK=\{k_1,\ldots,k_N\}\subset\mathbb Z^3\setminus\{0\}$, $N\ge1$, satisfy $\cK\cap(-\cK)=\varnothing$, and assume that
\begin{align}\label{eq:Fourier-mode-forcing-main}
 \cS_0
 =
 \{1:\varepsilon=1\}\cup
 \{\cos(k\cdot x),\sin(k\cdot x):k\in\cK\},
 \qquad \varepsilon\in\{0,1\},
\end{align}
where the first set is empty when $\varepsilon=0$.  Set
\begin{align*}
 \cA=\{\pm k:k\in\cK\}\cup\{0:\varepsilon=1\},
 \qquad
 \Gamma_{\mathrm{pair}}
 =\spanop_{\mathbb Z}\{k+k':k,k'\in\cA\}.
\end{align*}
For $e^\circ\in\cA$, put $\mathscr C=e^\circ+\Gamma_{\mathrm{pair}}$, which is independent of the choice of $e^\circ$.  For $r\in\mathbb R$, define
\begin{align*}
 H_{\mathscr C}^r
 =\{f\in H^r(\mathbb T^3;\mathbb R):
       \widehat f(\ell)=0\text{ for }\ell\notin\mathscr C\},
\end{align*}
where $\widehat f$ is  the Fourier transform of $f$ and put
\begin{align*}
 \mathcal H_{\mathscr C}
 =H_{\mathscr C}^1\times H_{\mathscr C}^0,
 \qquad
 \mathcal H_{\mathscr C}^{r}
 =H_{\mathscr C}^{1+r}\times H_{\mathscr C}^{r},
 \qquad
 \mathcal E_{s,\mathscr C}
 =H_{\mathscr C}^{-s}\times H_{\mathscr C}^{-1-s}.
\end{align*}
By \Cref{thm:sharp-Fourier-criterion},
\begin{align*}
 \cS_\infty^{\mathbb C}
 =\spanop_{\mathbb C}
 \{\e^{\mathrm i\ell\cdot x}:\ell\in\mathscr C\},
\end{align*}
and hence $\overline{\cS_\infty}^{\,H^r}=H_{\mathscr C}^r$
for every finite $r$.

\begin{theorem}\label{thm:sector-and-lower-dimensional}
Assume \eqref{eq:damping-assumption} and \eqref{eq:Fourier-mode-forcing-main}.  If $\supp\widehat a\subset\Gamma_{\mathrm{pair}}$, then $\mathcal H_{\mathscr C}$ is invariant under \eqref{eq:main-spde}.  For every $0<s<1/2$, the restricted Markov semigroup satisfies the weighted Wasserstein spectral gap conclusion of \Cref{thm:weighted-gap}, with $\Hh$ and $\Es{s}$ replaced by $\mathcal H_{\mathscr C}$ and $\mathcal E_{s,\mathscr C}$, respectively.  It also satisfies all conclusions of \Cref{thm:main} in the relative energy phase space; in particular, its invariant probability measure is unique, has full support in $\mathcal H_{\mathscr C}$, and the restricted dynamics is exponentially mixing in the relative energy topology.

For $d\in\{1,2\}$, consider the corresponding cubic stochastic wave equation on $\mathbb T^d$ with smooth strictly positive damping and Fourier-mode forcing of the form \eqref{eq:Fourier-mode-forcing-main}, with the pair sum lattice $\Gamma_{\mathrm{pair}}^{(d)}\subset\mathbb Z^d$ defined analogously.  If $\Gamma_{\mathrm{pair}}^{(d)}=\mathbb Z^d$, then the conclusions of \Cref{thm:weighted-gap,thm:main} hold on $H^1(\mathbb T^d)\times L^2(\mathbb T^d)$.
\end{theorem}

The Fourier-sector assertions are proved in \Cref{prop:invariant-Fourier-sector}, and the lower-dimensional reduction is given in \Cref{cor:lower-dimensional-unique-ergodicity}.

\section{A stable--compact criterion for spectral gap}\label{sec:spectral-gap-criterion}

In this section we formulate and prove the abstract stable--compact
spectral gap criterion used in the sequel.  Its assumptions separate the
Lyapunov and stable--compact structure from the two controllability inputs:
dense Malliavin range and uniform accessibility.

We first formulate the setting and assumptions.  The sampling time is not
fixed in advance: a sufficiently large block will be selected below, and
once this block has been chosen, all endpoint and coupling objects are
frozen and the proof reduces to an ordinary one block argument.

Let $H$ and $\mathbb H$ be separable Hilbert spaces with a continuous
dense embedding $H\hookrightarrow\mathbb H$, and let $(P_t)_{t\ge0}$ be
a Feller Markov semigroup on $H$.  We use the Wasserstein notation
$\Wass_\delta$ from \eqref{gap:eq:Wasserstein-definition}.  For every
$T>0$, let $(E_T,H_{W,T},\gamma_T)$ be an abstract Wiener space with
Cameron--Martin embedding $\iota_T:H_{W,T}\hookrightarrow E_T$, and let
$\Phi_T:H\times E_T\to H$ be a Borel endpoint representation of $P_T$, which in applications is the time $T$ solution map of the SPDE.

The first assumption below provides the Lyapunov structure.

\begin{assumption}
\label{ass:QS-Lyapunov}
There are $\lambda_0>0$, $C<\infty$, and a nonnegative function
$\vartheta(t)=o(t^{-1})$ as $t\to\infty$ such that
$\Vv:H\to[1,\infty)$ is continuous with compact sublevels in
$\mathbb H$ and
\begin{align}\label{eq:QS-Lyapunov}
 P_t\e^{s\Vv}(\mathsf x)
 \le C\exp\{s\vartheta(t)\Vv(\mathsf x)\},
 \qquad 0<s\le\lambda_0,\quad t\ge0.
\end{align}
\end{assumption}

For $R<\infty$, set
\begin{align}\label{eq:QS-energy-core}
 \mathbb V_R=\{\mathsf x\in H:\Vv(\mathsf x)\le R\}.
\end{align}
For the next two assumptions, fix a sufficiently large block length $T$
and suppress the block subscript from
$E_T,H_{W,T},\gamma_T,\iota_T,\Phi_T$ and all endpoint derivatives.

The assumption below gives the stable-compact structure: the solution difference is
stable modulo a compact defect, and the endpoint depends regularly on the
Cameron--Martin driver on bounded cores.
\begin{assumption}
\label{ass:QS-stable-compact}
We write $\mathsf z=(\mathsf x,\mathsf y)\in H^2$ for convenience.  
\begin{enumerate}[label=\textup{(\roman*)}]
\item There are strongly measurable $S_{\mathsf z,\omega},\mathcal K_{\mathsf z,\omega}\in\cL(\mathbb H)$ such that, writing $\mathcal J_{\mathsf z,\omega}=S_{\mathsf z,\omega}+\mathcal K_{\mathsf z,\omega}$,
\begin{align}\label{eq:QS-secant}
 \Phi(\mathsf y,\omega)-\Phi(\mathsf x,\omega)=\mathcal J_{\mathsf z,\omega}(\mathsf y-\mathsf x),\qquad
 \norm{S_{\mathsf z,\omega}}_{\cL(\mathbb H)}\le\rho,\qquad
 \mathcal K_{\mathsf z,\omega}\in\cK(\mathbb H),
\end{align}
where the deterministic $\rho=\rho(T)$ satisfies $\rho(T)\to0$ as $T\to\infty$. 
\item The Malliavin derivative $\mathcal D\Phi(\mathsf x,\omega)\in\cL(H_W,\mathbb H)$ is defined on $H\times E$ and strongly measurable, and for every $R<\infty$ and compact $G\subset E$, the maps
\begin{align*}
 (\mathsf z,\omega)\mapsto\mathcal K_{\mathsf z,\omega}\in\cL(\mathbb H),\quad 
 (\mathsf x,\omega)\mapsto\mathcal D\Phi(\mathsf x,\omega)\in\cL(H_W,\mathbb H)
\end{align*}
are continuous on $\mathbb V_R^2\times G$ and $\mathbb V_R\times G$, respectively, where $\mathbb V_R$ carries the $\mathbb H$-topology.

\item There is $p_0>0$ independent of $T$ such that, for every $R<\infty$, there are measurable $C_R:H\times E\to[1,\infty)$ and $r_R:H\times E\to(0,1]$ satisfying
\begin{align}\label{eq:QS-driver-Taylor}
 \norm{\Phi(\mathsf x,\omega+\iota h)-\Phi(\mathsf x,\omega)-\mathcal D\Phi(\mathsf x,\omega)h}_{\mathbb H}
 \le C_R(\mathsf x,\omega)\norm h_{H_W}^2
\end{align}
whenever $\mathsf x\in\mathbb V_R$ and $\norm h_{H_W}\le r_R(\mathsf x,\omega)$, and
\begin{align}\label{eq:QS-core-moments}
 \sup_{\mathsf z\in\mathbb V_R^2}\int_E\left[
 C_R(\mathsf x,\omega)^{p_0}+r_R(\mathsf x,\omega)^{-p_0}
 +\norm{\mathcal D\Phi(\mathsf x,\omega)}^{p_0}
 +\norm{\mathcal J_{\mathsf z,\omega}}^{p_0}
 \right]\gamma(\dd\omega)<\infty.
\end{align}
For every finite dimensional $F\subset\iota^*(E^*)$ and compact $G\subset E$, \eqref{eq:QS-driver-Taylor} holds on $\mathbb V_R\times G$ with a deterministic constant and radius depending on $T,R,G,F$.
\end{enumerate}
\end{assumption}

The next assumption provides the Malliavin controllability needed to
compensate the compact defect by finitely many Cameron--Martin directions.
\begin{assumption}
\label{ass:QS-Malliavin}
There is a Borel set $\Omega_*\subset E$ of full $\gamma$-measure such that
\begin{align}\label{eq:QS-dense-range}
 \overline{\Ran\mathcal D\Phi(\mathsf x,\omega)}^{\,\mathbb H}=\mathbb H
\end{align}
for every $\mathsf x\in H$ and $\omega\in\Omega_*$.  This is required for every sufficiently large block under the suppressed notation above.
\end{assumption}

For $\chi>0$, define the premetric
\begin{align}\label{eq:QS-premetric}
 \mathsf D_{\chi}(\mathsf x,\mathsf y)
 =\e^{\chi[\Vv(\mathsf x)+\Vv(\mathsf y)]}\norm{\mathsf x-\mathsf y}_{\mathbb H}.
\end{align}
The next assumption is the dissipative estimate from which the high energy region is selected. For concrete SPDEs, this estimate is typically obtained by combining the
dissipative structure with suitable a priori control relative to the
Lyapunov function.  For the wave equation, see
\Cref{rem:highcontraction} and \Cref{gap:prop:high}
\begin{assumption}
\label{ass:QS-synchronous}
There are $\alpha_*\in(0,1]$ and $\chi,c_*,C_*>0$ such that, for every
sufficiently large $T$, there is $C_T\ge1$ satisfying
\begin{align*}
 \log C_T\le C_*(1+T)
\end{align*}
and for every $\mathsf x\ne\mathsf y$,
\begin{align}\label{eq:QS-weighted-synchronous}
 \int_E\mathsf D_{\chi}
 \bigl(\Phi_T(\mathsf x,\omega),\Phi_T(\mathsf y,\omega)\bigr)^{\alpha_*}\gamma_T(\dd\omega)
 \le C_T\e^{-c_*[\Vv(\mathsf x)+\Vv(\mathsf y)]}
 \mathsf D_{\chi}(\mathsf x,\mathsf y)^{\alpha_*}.
\end{align}
\end{assumption}

The final assumption provides the uniform accessibility that allows bounded
Lyapunov sets to be turned into small sets in the Harris argument.
\begin{assumption}
\label{ass:QS-accessibility}
There is $\mathsf x_*\in H$ such that, for every $R<\infty$ and $\varepsilon>0$, there is $t_{R,\varepsilon}<\infty$ for which
\begin{align}\label{eq:QS-accessibility}
 \inf_{\mathsf x\in\mathbb V_R}P_t\bigl(\mathsf x,B_H(\mathsf x_*,\varepsilon)\bigr)>0,\qquad t\ge t_{R,\varepsilon}.
\end{align}
\end{assumption}

We are now ready to state the main result of this section. 
\begin{theorem}[Stable--compact spectral gap criterion]
\label{thm:stable-compact-spectral-gap}
Suppose that \Cref{ass:QS-Lyapunov,ass:QS-stable-compact,ass:QS-Malliavin,ass:QS-synchronous,ass:QS-accessibility} hold.  Then there are $\alpha\in(0,\alpha_*]$, a sampling time $\tau>0$, $r_0,\beta>0$, $0<\lambda\le\lambda_0$, and $\varrho\in(0,1)$ such that, with $\mathsf D_{\chi}$ from \eqref{eq:QS-premetric} and
\begin{align}\label{eq:QS-costs}
 \begin{aligned}
 d_0(\mathsf x,\mathsf y)&=1\wedge\left[\frac{\mathsf D_{\chi}(\mathsf x,\mathsf y)}{r_0}\right]^\alpha,\\
 d(\mathsf x,\mathsf y)&=\sqrt{d_0(\mathsf x,\mathsf y)\left(1+\beta\e^{\lambda\Vv(\mathsf x)}+\beta\e^{\lambda\Vv(\mathsf y)}\right)},
 \end{aligned}
\end{align}
one has
\begin{align}\label{eq:QS-weighted-gap}
 \Wass_d(\mu_1P_{n\tau},\mu_2P_{n\tau})\le\varrho^n\Wass_d(\mu_1,\mu_2),\qquad n\ge0,
\end{align}
for all probability measures $\mu_1,\mu_2$ with finite $\e^{\lambda\Vv}$ moment.
\end{theorem}

The proof of the theorem will be given at the end of this section.  We
first select a sufficiently long coupling block.

\begin{lemma}[Selection of a coupling block]
\label{lem:QS-block-selection}
Under \Cref{ass:QS-Lyapunov,ass:QS-stable-compact,ass:QS-synchronous},
there exist $\alpha\in(0,\alpha_*]$, a sufficiently large block length
$T_*$, and constants $R_*,M_{\rm en}<\infty$ such that, after fixing this
block and suppressing its subscript,
\begin{align}\label{eq:QS-selected-high-energy}
 \int_E\mathsf D_{\chi}\bigl(\Phi(\mathsf x,\omega),\Phi(\mathsf y,\omega)\bigr)^\alpha\,\gamma(\dd\omega)
 \le\frac14\mathsf D_{\chi}(\mathsf x,\mathsf y)^\alpha
\end{align}
whenever $\Vv(\mathsf x)+\Vv(\mathsf y)\ge R_*$, and
\begin{align}\label{eq:QS-selected-core-margin}
 M_{\rm en}:=\sup_{\mathsf x\in\mathbb V_{R_*}}P\e^{2\alpha\chi\Vv}(\mathsf x)<\infty,
 \qquad \rho^\alpha M_{\rm en}<1.
\end{align}
Moreover, for every prescribed $p>1$, the preceding objects may be chosen so that
\begin{align}\label{eq:alpha-p}
 2\alpha p<p_0,\qquad 4\alpha\chi p<\lambda_0.
\end{align}
\end{lemma}

\begin{proof}
For $0<\alpha\le\alpha_*$, set $\theta_\alpha=\alpha/\alpha_*$.  Since
$\mathsf D_{\chi}^{\alpha}=(\mathsf D_{\chi}^{\alpha_*})^{\theta_\alpha}$,
Jensen's inequality shows that \eqref{eq:QS-weighted-synchronous} remains
valid with exponent $\alpha$, with $C_T$ replaced by $C_T^{\theta_\alpha}$ and $c_*$ by
$\theta_\alpha c_*$.  Fix any $p>1$ and choose $0<\alpha\le\alpha_*$ sufficiently small so that
\begin{align*}
 2\alpha p<p_0,\qquad 4\alpha\chi p<\lambda_0.
\end{align*}
With no change of notation, write $C_T$ and $c$ for the resulting
constants and set
\begin{align*}
 R_T=1\vee c^{-1}\log(4C_T),
\end{align*}
which is $O(T)$ since $\log C_T=O(T)$.  By
\eqref{eq:QS-Lyapunov} and $2\alpha\chi<\lambda_0$,
\begin{align*}
 \sup_{\mathsf x\in\mathbb V_{R_T}}P_T\e^{2\alpha\chi\Vv}(\mathsf x)
 \le C\exp\{2\alpha\chi\vartheta(T)R_T\}.
\end{align*}
Since $\vartheta(T)=o(T^{-1})$ and $R_T=O(T)$, the right-hand side is
uniformly bounded for all sufficiently large $T$.  As $\rho(T)\to0$, we
may therefore choose $T_*$ sufficiently large and set $ R_*=R_{T_*}.$
Then
\begin{align*}
 \rho(T_*)^\alpha
 \sup_{\mathsf x\in\mathbb V_{R_*}}P_{T_*}\e^{2\alpha\chi\Vv}(\mathsf x)<1.
\end{align*}
The definition of $R_T$ gives \eqref{eq:QS-selected-high-energy}, while
the last display gives \eqref{eq:QS-selected-core-margin}.
\end{proof}

Fix once and for all some $p>2$, and let $\alpha,T_*,R_*,M_{\rm en}$ be supplied by \Cref{lem:QS-block-selection} for this choice of $p$.  Write
\begin{align*}
 P=P_{T_*},\qquad \Phi=\Phi_{T_*},\qquad E=E_{T_*},\qquad H_W=H_{W,T_*},\qquad
 \gamma=\gamma_{T_*},\qquad \iota=\iota_{T_*}.
\end{align*}
We suppress the block subscript from $S,\mathcal K,\mathcal J,\mathcal D\Phi,C_R,r_R,\Omega_*$ and write $\rho=\rho(T_*)$.  All constants below may depend on this fixed block.

\subsection{Uniform compensated contraction on compact cores}
The next proposition provides the finite dimensional compensation needed
for the low energy coupling.  On a compact energy driver core, the compact
part of the solution difference can be uniformly absorbed by finitely many
Cameron--Martin directions, leading to a strict contraction of the endpoint.

In what follows, if $F$ is a finite dimensional subspace of the Cameron-Martin space $H_W$, a
smooth cylindrical map $U:E\to F$ means a map of the form
\begin{align*}
 U(\omega)=\varphi\bigl(\ell_1(\omega),\ldots,\ell_N(\omega)\bigr),
\end{align*}
for some $N\ge1$, $\ell_1,\ldots,\ell_N\in E^*$, and
$\varphi\in C^\infty(\mathbb R^N;F)$.

\begin{proposition}
\label{prop:QS-core-compensated-contraction}
Suppose that \Cref{ass:QS-stable-compact,ass:QS-Malliavin} hold.
Fix $R<\infty$, a compact set $G\subset\Omega_*$, and $q\in(\rho,1)$.
Then there are a finite dimensional space
$F\subset\iota^*(E^*)$ and $M<\infty$ such that, for every $\delta>0$,
there are an open neighborhood $\mathcal O\supset G$,
a constant $C_\delta<\infty$, and $0<r_\delta\leq \frac{1}{2(M+C_\delta)}$ satisfying
\begin{align*}
 \gamma(\mathcal O\setminus G)<\delta,
\end{align*}
with the following property.

For every $\mathsf z=(\mathsf x,\mathsf y)\in\mathbb V_R^2$, with
$\mathbb V_R$ as in \eqref{eq:QS-energy-core}, and
$r=\norm{\mathsf y-\mathsf x}_{\mathbb H}\le r_\delta$, there is a
smooth cylindrical map $U_{\mathsf z}:E\to F$, vanishing on
$\mathcal O^c$, such that
\begin{align}\label{eq:QS-shift-bounds}
 \sup_\omega\norm{U_{\mathsf z}(\omega)}_{H_W}
 &\le Mr,\qquad
 \sup_\omega\norm{\mathcal D U_{\mathsf z}(\omega)}_{\rm HS}
 \le C_\delta r,\\
 \sup_\omega\norm{\mathcal D U_{\mathsf z}(\omega)}_{\rm op}
 &\le\frac12,
 \nonumber
\end{align}
and
\begin{align}\label{eq:QS-good-core-contraction}
 \norm{\Phi(\mathsf y,\omega+\iota U_{\mathsf z}(\omega))
 -\Phi(\mathsf x,\omega)}_{\mathbb H}
 \le qr,\qquad \omega\in G.
\end{align}
Note that the space $F$ and the constant $M$ are independent of $\delta$.
\end{proposition}

\begin{proof}
Put $\eta=(q-\rho)/4$.  We divide the proof into three steps.

\medskip
\noindent\emph{Step 1. Finite-factor compensation.}
For $a=(\mathsf z,\omega)\in\mathbb V_R^2\times G$, compactness of
$\mathcal K_a$ gives a finite rank operator
\begin{align*}
 K^{(0)}h=\sum_{j=1}^m\ip h{e_j}_{\mathbb H}v_j
\end{align*}
such that $\norm{\mathcal K_a-K^{(0)}}<\eta/4$.  Since the Cameron--Martin
embedding $\iota:H_W\hookrightarrow E$ is injective, $\iota^*(E^*)$ is
dense in $H_W$.  Hence \eqref{eq:QS-dense-range} implies
\begin{align*}
 \overline{\mathcal D\Phi(\mathsf y,\omega)
 \bigl(\iota^*(E^*)\bigr)}^{\,\mathbb H}
 =\mathbb H.
\end{align*}
We may therefore choose $g_1,\ldots,g_m\in\iota^*(E^*)$ sufficiently
close to preimages of $v_1,\ldots,v_m$ so that
\begin{align*}
 R_a h=\sum_{j=1}^m\ip h{e_j}_{\mathbb H}g_j
\end{align*}
satisfies
\begin{align*}
 \norm{\mathcal K_a-\mathcal D\Phi(\mathsf y,\omega)R_a}<\frac\eta2.
\end{align*}
In particular, $R_a$ has finite rank and
$\Ran R_a\subset\iota^*(E^*)$.

Fix such an $R_a$.  By the operator norm continuity in
\Cref{ass:QS-stable-compact}\textup{(ii)}, there is a neighborhood
$U_a$ of $a$ in $\mathbb V_R^2\times G$ such that the same fixed operator
$R_a$ satisfies
\begin{align*}
 \norm{\mathcal K_{\mathsf z',\omega'}
 -\mathcal D\Phi(\mathsf y',\omega')R_a}<\eta,
 \qquad
 (\mathsf z',\omega')\in U_a.
\end{align*}
Compactness of $\mathbb V_R^2\times G$ gives a finite subcover
$U_{a_1},\ldots,U_{a_N}$.  Set $R_i=R_{a_i}$ and choose a continuous
partition of unity $(\varphi_i)_{i=1}^N$ subordinate to this cover.  Then
\begin{align*}
 R^0(\mathsf z,\omega)
 =\sum_{i=1}^N\varphi_i(\mathsf z,\omega)R_i
\end{align*}
satisfies
\begin{align}\label{eq:QS-R0-compensation}
 \sup_{\mathbb V_R^2\times G}
 \norm{\mathcal K_{\mathsf z,\omega}
 -\mathcal D\Phi(\mathsf y,\omega)R^0(\mathsf z,\omega)}
 <\eta.
\end{align}
Set
\begin{align}\label{eq:QS-factor-FM}
 F=\spanop\bigcup_{i=1}^N\Ran R_i,
 \qquad
 M=2\sum_{i=1}^N\norm{R_i}.
\end{align}
Then $F\subset\iota^*(E^*)$ is finite dimensional.  The space $F$ and
the constant $M$ remain fixed throughout the rest of the proof.

\medskip
\noindent\emph{Step 2. Cylindrical localization.}
We first replace the continuous field $R^0$ by a smooth cylindrical one.
Since $E^*$ separates points of $G$, the algebra on
$\mathbb V_R^2\times G$ generated by continuous functions of $\mathsf z$
and smooth cylindrical functions of $\omega$ separates points and contains
the constants.  By the Stone--Weierstrass theorem, for every
$\varepsilon>0$ and $1\le i\le N$ there is a function
\begin{align*}
 \psi_i(\mathsf z,\omega)
 =\sum_{k=1}^{L_i}f_{ik}(\mathsf z)
 \theta_{ik}\bigl(\ell_1(\omega),\ldots,\ell_d(\omega)\bigr),
\end{align*}
where $f_{ik}\in C(\mathbb V_R^2)$,
$\theta_{ik}\in C_c^\infty(\mathbb R^d)$, and
$\ell_1,\ldots,\ell_d\in E^*$, such that
\begin{align*}
 \sup_{\mathbb V_R^2\times G}\abs{\psi_i-\varphi_i}<\varepsilon\, \text{ and } \quad  \sup_{\mathbb V_R^2\times E}\abs{\psi_i}\le2,
 \quad
 \sup_{\mathbb V_R^2\times E}
 \norm{\mathcal D\psi_i}_{H_W}<\infty
\end{align*}
upon composing with a fixed smooth truncation. 
Define
\begin{align*}
 R^1(\mathsf z,\omega)
 =\sum_{i=1}^N\psi_i(\mathsf z,\omega)R_i.
\end{align*}
Then $R^1(\mathsf z,\omega)\in\cL(\mathbb H,F)$ and, by
\eqref{eq:QS-factor-FM},
\begin{align}\label{eq:QS-R1-bound}
 \sup_{\mathbb V_R^2\times E}\norm{R^1}\le M.
\end{align}
Moreover, \Cref{ass:QS-stable-compact}\textup{(ii)} and compactness give
\begin{align*}
 B_{R,G}:=
 \sup_{\mathbb V_R\times G}
 \norm{\mathcal D\Phi(\mathsf y,\omega)}_{\cL(H_W,\mathbb H)}
 <\infty,
\end{align*}
which combined with \eqref{eq:QS-R0-compensation} and the construction of $R^1$ gives 
\begin{align}\label{eq:QS-R1-compensation}
 \sup_{\mathbb V_R^2\times G}
 \norm{\mathcal K_{\mathsf z,\omega}
 -\mathcal D\Phi(\mathsf y,\omega)R^1(\mathsf z,\omega)}
 <2\eta.
\end{align}
For $\norm h_{\mathbb H}\le1$,
\begin{align*}
 \mathcal D\bigl[R^1(\mathsf z,\cdot)h\bigr]
 =\sum_{i=1}^N
 (R_i h)\otimes\mathcal D\psi_i(\mathsf z,\cdot),
\end{align*}
and hence the preceding cylindrical bounds give
\begin{align}\label{eq:QS-R1-derivative}
 \sup_{\mathbb V_R^2\times E}
 \sup_{\norm h_{\mathbb H}\le1}
 \norm{\mathcal D\bigl[R^1(\mathsf z,\cdot)h\bigr]}_{\rm HS}
 <\infty.
\end{align}

We now localize $R^1$ near $G$.  Choose a countable separating family
$(\lambda_j)_{j\ge1}\subset E^*$ and set
\begin{align*}
 \mathbb L_n(\omega)
 =(\lambda_1(\omega),\ldots,\lambda_n(\omega)),
 \qquad
 C_n=\mathbb L_n^{-1}\bigl(\mathbb L_n(G)\bigr).
\end{align*}
Then $C_n\downarrow G$ and therefore for some $n$,
\begin{align*}
 \gamma(C_n\setminus G)<\frac\delta2.
\end{align*}
Since $\mathbb L_n(G)$ is compact and $(\mathbb L_n)_\#\gamma$ is a
Radon probability on $\mathbb R^n$, there is an open set
$V\supset\mathbb L_n(G)$ such that
\begin{align*}
 \gamma\bigl(\mathbb L_n^{-1}(V)\setminus C_n\bigr)
 <\frac\delta2.
\end{align*}
Choose $\chi\in C_c^\infty(\mathbb R^n;[0,1])$ with
$\chi=1$ on $\mathbb L_n(G)$ and $\supp\chi\subset V$, and put
\begin{align*}
 \mathcal O=\mathbb L_n^{-1}(V),\qquad
 \mathscr R(\mathsf z,\omega)
 =\chi(\mathbb L_n(\omega))R^1(\mathsf z,\omega).
\end{align*}
Then $\mathcal O$ is an open neighborhood of $G$,
$\gamma(\mathcal O\setminus G)<\delta$, and
$\mathscr R=0$ on $\mathcal O^c$.  Since $\chi=1$ on $G$,
\eqref{eq:QS-R1-compensation} gives
\begin{align}\label{eq:QS-core-compensation}
 \sup_{\mathbb V_R^2\times G}
 \norm{\mathcal K_{\mathsf z,\omega}
 -\mathcal D\Phi(\mathsf y,\omega)\mathscr R(\mathsf z,\omega)}
 <2\eta.
\end{align}
Moreover, \eqref{eq:QS-R1-bound} gives
\begin{align*}
 \sup_{\mathbb V_R^2\times E}\norm{\mathscr R}\le M.
\end{align*}
For $\norm h_{\mathbb H}\le1$, the product rule gives
\begin{align*}
 \mathcal D\bigl[\mathscr R(\mathsf z,\cdot)h\bigr]
 &=
 \mathcal D[\chi\circ\mathbb L_n]\otimes
 R^1(\mathsf z,\cdot)h\\
 &\quad
 +(\chi\circ\mathbb L_n)
 \mathcal D\bigl[R^1(\mathsf z,\cdot)h\bigr].
\end{align*}
Since $\chi$ is smooth with compact support, $\mathbb L_n$ has finite
rank, and \eqref{eq:QS-R1-bound}--\eqref{eq:QS-R1-derivative} hold, there
is $C_\delta<\infty$ such that
\begin{align}\label{eq:QS-selector-derivative}
 \sup_{\mathbb V_R^2\times E}
 \sup_{\norm h_{\mathbb H}\le1}
 \norm{\mathcal D\bigl[\mathscr R(\mathsf z,\cdot)h\bigr]}_{\rm HS}
 \le C_\delta.
\end{align}

For $h=\mathsf y-\mathsf x$ and
$r=\norm h_{\mathbb H}$, define
\begin{align*}
 U_{\mathsf z}(\omega)
 =-\mathscr R(\mathsf z,\omega)h.
\end{align*}
By \eqref{eq:QS-factor-FM},
\eqref{eq:QS-R1-bound}, and
\eqref{eq:QS-selector-derivative}, the map $U_{\mathsf z}$ takes values
in $F$ and satisfies
\begin{align*}
 \sup_\omega\norm{U_{\mathsf z}(\omega)}_{H_W}\le Mr,
 \qquad
 \sup_\omega\norm{\mathcal DU_{\mathsf z}(\omega)}_{\rm HS}
 \le C_\delta r.
\end{align*}
After decreasing $r_\delta$ so that
$(M+C_\delta) r_\delta\le1/2$, we also have
\begin{align*}
 \sup_\omega\norm{\mathcal DU_{\mathsf z}(\omega)}_{\rm op}
 \le\frac12.
\end{align*}
Thus \eqref{eq:QS-shift-bounds} holds.

\medskip
\noindent\emph{Step 3. Endpoint contraction.}
By \Cref{ass:QS-stable-compact}\textup{(iii)}, applied to the fixed
finite dimensional space $F$ in \eqref{eq:QS-factor-FM} and the compact
driver core $G$, there are $c_*,C_* >0$ such that
\begin{align*}
 \norm{\Phi(\mathsf y,\omega+\iota u)-\Phi(\mathsf y,\omega)
 -\mathcal D\Phi(\mathsf y,\omega)u}_{\mathbb H}
 \le C_*\norm u_{H_W}^2
\end{align*}
whenever $\mathsf y\in\mathbb V_R$, $\omega\in G$, $u\in F$, and
$\norm u_{H_W}\le c_*$.  Decrease $r_\delta$ further so that
$Mr_\delta\le c_*$. Since $\rho+2\eta=\frac{\rho+q}{2}<q$, for $\omega\in G$, the identity
\eqref{eq:QS-secant}, \eqref{eq:QS-core-compensation}, and
$U_{\mathsf z}=-\mathscr Rh$ give
\begin{align*}
 \norm{\Phi(\mathsf y,\omega+\iota U_{\mathsf z}(\omega))
 -\Phi(\mathsf x,\omega)}_{\mathbb H}
 &\le
 \norm{S_{\mathsf z,\omega}h+
 \bigl[\mathcal K_{\mathsf z,\omega}
 -\mathcal D\Phi(\mathsf y,\omega)\mathscr R(\mathsf z,\omega)\bigr]h}_{\mathbb H} + C_*M^2r^2\\
 &\le (\rho+2\eta)r+C_*M^2r^2 \le qr,
\end{align*}
which is \eqref{eq:QS-good-core-contraction}.
\end{proof}

\subsection{Coupling at low energy}
The preceding proposition contracts the shifted endpoint on a compact Wiener core, but the shifted endpoint does not yet have the exact second Markov marginal.  The next proposition uses the finite factor Gaussian transformation and a maximal repair to obtain a true coupling of the two transition probabilities.

Recall the distance
\begin{align}\label{eq:QS-basic-cost-internal}
 d_0(\mathsf x,\mathsf y)=1\wedge\left[\frac{\mathsf D_{\chi}(\mathsf x,\mathsf y)}{r_0}\right]^\alpha
 =1\wedge\left[\frac{\norm{\mathsf x-\mathsf y}_{\mathbb H}}{r_0}
 \e^{\chi(\Vv(\mathsf x)+\Vv(\mathsf y))}\right]^\alpha,
\end{align}
as in \eqref{eq:QS-costs} and \eqref{eq:QS-premetric}, and the radius $R_*$, exponent $p$ and constant $M_{\rm en}$ from \Cref{lem:QS-block-selection}.  
\begin{proposition}
\label{prop:QS-low-energy-contraction}
There are $r_0>0$ and $\theta_{\rm lo}\in(0,1)$ such that
\begin{align}\label{eq:QS-low-contraction}
 \Wass_{d_0}(P(\mathsf x,\cdot),P(\mathsf y,\cdot))
 \le\theta_{\rm lo}d_0(\mathsf x,\mathsf y)
\end{align}
whenever $\Vv(\mathsf x)+\Vv(\mathsf y)\le R_*$ and
$d_0(\mathsf x,\mathsf y)<1$.
\end{proposition}
\begin{proof}
By \eqref{eq:QS-selected-core-margin}, choose $q\in(\rho,1)$ and $\varepsilon_0>0$ so that
\begin{align}\label{eq:QS-q-choice}
 q^\alpha M_{\rm en}+4\varepsilon_0<1.
\end{align}

\medskip
\noindent\emph{Step 1. Compact-core shift and exceptional sets.}
Recall that $\mathcal J_{\mathsf z,\omega}=S_{\mathsf z,\omega}+\mathcal K_{\mathsf z,\omega}$ and set
\begin{align*}
 \mathscr E_{\mathsf z}(\omega)=\norm{\mathcal J_{\mathsf z,\omega}}^\alpha
 \e^{\alpha\chi[\Vv(\Phi(\mathsf x,\omega))+\Vv(\Phi(\mathsf y,\omega))]}.
\end{align*}
Cauchy--Schwarz, \eqref{eq:QS-Lyapunov} at $t=T_*$, \eqref{eq:QS-core-moments}, and \eqref{eq:alpha-p} give
\begin{align*}
 \sup_{\mathsf z\in\mathbb V_{R_*}^2}\int_E\mathscr E_{\mathsf z}(\omega)^p\,\gamma(\dd\omega)<\infty.
\end{align*}
Since $\gamma(\Omega_*)=1$ and $\gamma$ is Radon, choose a compact $G\subset\Omega_*$ such that
\begin{align}\label{eq:QS-good-driver-tail}
 \sup_{\mathsf z\in\mathbb V_{R_*}^2}\int_{G^c}\mathscr E_{\mathsf z}\,\dd\gamma
 \le\left(\sup_{\mathsf z\in\mathbb V_{R_*}^2}\int_E\mathscr E_{\mathsf z}^p\,\dd\gamma\right)^{1/p}
 \gamma(G^c)^{1-1/p}<\varepsilon_0.
\end{align}

Apply \Cref{prop:QS-core-compensated-contraction} with $R=R_*$ and the above $G,q$.  This fixes a finite dimensional space $F\subset\iota^*(E^*)$ and a constant $M<\infty$, both independent of $\delta$.  Fix $1<p_1<p/2$.  By \eqref{eq:alpha-p},
\begin{align}\label{eq:QS-p1-margin}
 4\alpha p_1<p_0,\qquad 4\alpha\chi p_1<\lambda_0.
\end{align}
For the moment let $\delta\in(0,1)$ be arbitrary, and let $\mathcal O,C_\delta,r_\delta$ and $U_{\mathsf z}$ be supplied by \Cref{prop:QS-core-compensated-contraction}, with
\begin{align}\label{eq:QS-radius-normalization}
 r_\delta\le(M+C_\delta)^{-1}.
\end{align}
For $r=\norm{\mathsf y-\mathsf x}_{\mathbb H}\le r_\delta$, set
\begin{align*}
 T_{\mathsf z}(\omega)=\omega+\iota U_{\mathsf z}(\omega).
\end{align*}
By \eqref{eq:QS-shift-bounds} and \eqref{eq:QS-radius-normalization},
\begin{align*}
 \sup_\omega\left(\norm{U_{\mathsf z}(\omega)}_{H_W}+\norm{\mathcal DU_{\mathsf z}(\omega)}_{\rm HS}\right)
 \le Mr+C_\delta r\le1,\qquad
 \sup_\omega\norm{\mathcal DU_{\mathsf z}(\omega)}_{\rm op}\le\frac12.
\end{align*}
Hence \Cref{lem:Gaussian-entropy} gives
\begin{align}\label{eq:QS-density-moments}
 \sup_{\mathsf z\in\mathbb V_{R_*}^2,\,r\le r_\delta}
 \norm{\frac{\dd(T_{\mathsf z})_\#\gamma}{\dd\gamma}}_{L^2(\gamma)}\le C_F,
\end{align}
where $C_F$ is independent of $\delta$.  Put
\begin{align*}
 \mathcal A_r=\{\omega\in E:\norm{U_{\mathsf z}(\omega)}_{H_W}\le r_{R_*}(\mathsf y,\omega)\}.
\end{align*}
For $\Vv(\mathsf y)\le R_*$, on $\mathcal A_r$, \eqref{eq:QS-driver-Taylor}, \eqref{eq:QS-secant}, and $\norm{U_{\mathsf z}(\omega)}_{H_W}\le Mr$ give
\begin{align*}
 \frac{\norm{\Phi(\mathsf y,T_{\mathsf z}\omega)-\Phi(\mathsf x,\omega)}_{\mathbb H}}r
 \le C_M\mathscr Q_{\mathsf z}(\omega),\qquad
 \mathscr Q_{\mathsf z}(\omega)=1+\norm{\mathcal J_{\mathsf z,\omega}}+\norm{\mathcal D\Phi(\mathsf y,\omega)}+C_{R_*}(\mathsf y,\omega).
\end{align*}
Then
\begin{align*}
 &\int_E\one_{\mathcal A_r}(\omega)
 \left[\frac{\mathsf D_{\chi}\bigl(\Phi(\mathsf x,\omega),\Phi(\mathsf y,T_{\mathsf z}\omega)\bigr)}r\right]^{\alpha p_1}\gamma(\dd\omega)\\
 &\qquad\le C_M\int_E\mathscr Q_{\mathsf z}(\omega)^{\alpha p_1}
 \e^{\alpha\chi p_1\Vv(\Phi(\mathsf x,\omega))}
 \e^{\alpha\chi p_1\Vv(\Phi(\mathsf y,T_{\mathsf z}\omega))}\,\gamma(\dd\omega).
\end{align*}
By Cauchy--Schwarz, the last integral is bounded by
\begin{align*}
 C_M\left(\int_E\mathscr Q_{\mathsf z}^{2\alpha p_1}
 \e^{2\alpha\chi p_1\Vv(\Phi(\mathsf x,\omega))}\,\dd\gamma\right)^{1/2}
 \left(\int_E\e^{2\alpha\chi p_1\Vv(\Phi(\mathsf y,T_{\mathsf z}\omega))}\,\dd\gamma\right)^{1/2}.
\end{align*}
A second application of Cauchy--Schwarz, together with \eqref{eq:QS-core-moments}, \eqref{eq:QS-Lyapunov}, and \eqref{eq:QS-p1-margin}, gives
\begin{align*}
 \int_E\mathscr Q_{\mathsf z}^{2\alpha p_1}
 \e^{2\alpha\chi p_1\Vv(\Phi(\mathsf x,\omega))}\,\dd\gamma
 \le\left(\int_E\mathscr Q_{\mathsf z}^{4\alpha p_1}\,\dd\gamma\right)^{1/2}
 \left(P\e^{4\alpha\chi p_1\Vv}(\mathsf x)\right)^{1/2}\le C.
\end{align*}
For the shifted factor, writing $L_{\mathsf z}=\frac{\dd(T_{\mathsf z})_\#\gamma}{\dd\gamma}$, \eqref{eq:QS-density-moments}, Cauchy--Schwarz, \eqref{eq:QS-Lyapunov}, and \eqref{eq:QS-p1-margin} give
\begin{align*}
 \int_E\e^{2\alpha\chi p_1\Vv(\Phi(\mathsf y,T_{\mathsf z}\omega))}\,\gamma(\dd\omega)
 &=\int_E\e^{2\alpha\chi p_1\Vv(\Phi(\mathsf y,\omega))}L_{\mathsf z}(\omega)\,\gamma(\dd\omega)\\
 &\le C_F\left(P\e^{4\alpha\chi p_1\Vv}(\mathsf y)\right)^{1/2}\le C.
\end{align*}
Thus there is $C_{\rm sh}<\infty$, independent of $\delta$, such that
\begin{align}\label{eq:QS-shifted-envelope}
 \sup_{\mathsf z\in\mathbb V_{R_*}^2,\,r\le r_\delta}
 \int_E\one_{\mathcal A_r}(\omega)
 \left[\frac{\mathsf D_{\chi}\bigl(\Phi(\mathsf x,\omega),\Phi(\mathsf y,T_{\mathsf z}\omega)\bigr)}r\right]^{\alpha p_1}
 \gamma(\dd\omega)\le C_{\rm sh}^{p_1}.
\end{align}

Choose now $\delta>0$ so small that
\begin{align}\label{eq:QS-collar-small}
 C_{\rm sh}\delta^{1-1/p_1}<\varepsilon_0,
\end{align}
and fix the corresponding $\mathcal O,C_\delta,r_\delta$ and $U_{\mathsf z}$.  In particular,
\begin{align*}
 \gamma(\mathcal O\setminus G)<\delta.
\end{align*}
Since $\norm{U_{\mathsf z}(\omega)}_{H_W}\le Mr$, \eqref{eq:QS-core-moments} gives
\begin{align}\label{eq:QS-Taylor-failure}
 \sup_{\mathsf z\in\mathbb V_{R_*}^2}\gamma(\mathcal A_r^c)
 \le M^{p_0}r^{p_0}\sup_{\Vv(\mathsf y)\le R_*}
 \int_Er_{R_*}(\mathsf y,\omega)^{-p_0}\,\gamma(\dd\omega)\le Cr^{p_0}.
\end{align}

On $\mathcal O^c$ one has $U_{\mathsf z}=0$, and hence
\begin{align*}
 d_0\bigl(\Phi(\mathsf x,\omega),\Phi(\mathsf y,\omega)\bigr)
 \le\left(\frac r{r_0}\right)^\alpha\mathscr E_{\mathsf z}(\omega).
\end{align*}
Since $\mathcal O^c\subset G^c$, \eqref{eq:QS-good-driver-tail} yields
\begin{align*}
 \int_{\mathcal O^c}d_0\bigl(\Phi(\mathsf x,\omega),\Phi(\mathsf y,T_{\mathsf z}\omega)\bigr)\,\gamma(\dd\omega)
 <\varepsilon_0\left(\frac r{r_0}\right)^\alpha.
\end{align*}
On $\mathcal O\setminus G$, \eqref{eq:QS-basic-cost-internal}, \eqref{eq:QS-shifted-envelope}, H\"older's inequality, and \eqref{eq:QS-collar-small} give
\begin{align*}
 \int_{\mathcal O\setminus G}\one_{\mathcal A_r}(\omega)
 d_0\bigl(\Phi(\mathsf x,\omega),\Phi(\mathsf y,T_{\mathsf z}\omega)\bigr)\,\gamma(\dd\omega)
 <\varepsilon_0\left(\frac r{r_0}\right)^\alpha.
\end{align*}
Consequently,
\begin{align}\label{eq:QS-bad-driver-cost}
 \int_{G^c}\one_{\mathcal A_r}(\omega)
 d_0\bigl(\Phi(\mathsf x,\omega),\Phi(\mathsf y,T_{\mathsf z}\omega)\bigr)\,\gamma(\dd\omega)
 \le2\varepsilon_0\left(\frac r{r_0}\right)^\alpha.
\end{align}

\medskip
\noindent\emph{Step 2. Exact marginal repair.}
Define
\begin{align*}
 X(\omega)=\Phi(\mathsf x,\omega),\qquad Z(\omega)=\Phi(\mathsf y,T_{\mathsf z}\omega).
\end{align*}
By \eqref{eq:QS-shift-bounds} and \Cref{lem:Gaussian-entropy},
\begin{align*}
 \Ent((T_{\mathsf z})_\#\gamma\mid\gamma)\le C_\delta r^2,
\end{align*}
where from now on $C_\delta$ denotes a finite constant which may change from line to line.  Since
\begin{align*}
 Z_\#\gamma=\Phi(\mathsf y,\cdot)_\#(T_{\mathsf z})_\#\gamma,\qquad
 P(\mathsf y,\cdot)=\Phi(\mathsf y,\cdot)_\#\gamma,
\end{align*}
the monotonicity of relative entropy under measurable pushforwards gives
\begin{align*}
 \Ent(Z_\#\gamma\mid P(\mathsf y,\cdot))\le C_\delta r^2.
\end{align*}
By Pinsker's inequality and maximal coupling, after enlarging the probability space if necessary, we may extend $(X,Z)$ to a triple $(X,Z,Y)$.  We retain the notation $\omega$ for the original driver coordinate, whose law is still $\gamma$, and write $\E$ for expectation on the enlarged space.  Then
\begin{align}\label{eq:QS-repair}
 \Law X=P(\mathsf x,\cdot),\qquad
 \Law Y=P(\mathsf y,\cdot),\qquad
 \E\one_{\{Y\ne Z\}}\le C_\delta r.
\end{align}
Indeed, one first maximally couples $Z_\#\gamma$ and $P(\mathsf y,\cdot)$ and then applies the gluing lemma with the joint law of $(X,Z)$.

\medskip
\noindent\emph{Step 3. Low-energy contraction.}
On $\{\omega\in G,Y=Z\}$, \eqref{eq:QS-good-core-contraction} gives
\begin{align*}
 d_0(X,Y)\le q^\alpha\left(\frac r{r_0}\right)^\alpha
 \e^{\alpha\chi[\Vv(X)+\Vv(Y)]}.
\end{align*}
Since $X$ and $Y$ have the exact marginals and $\mathsf x,\mathsf y\in\mathbb V_{R_*}$,
\begin{align*}
 \E\e^{\alpha\chi[\Vv(X)+\Vv(Y)]}
 \le\left(P\e^{2\alpha\chi\Vv}(\mathsf x)P\e^{2\alpha\chi\Vv}(\mathsf y)\right)^{1/2}
 \le M_{\rm en}.
\end{align*}
Since the $\omega$-marginal remains $\gamma$, \eqref{eq:QS-bad-driver-cost} and \eqref{eq:QS-Taylor-failure} remain valid on the enlarged space.  Using these estimates together with \eqref{eq:QS-repair} and $d_0\le1$, we obtain
\begin{align}\label{eq:QS-low-cost-preliminary}
 \E d_0(X,Y)
 \le\left[q^\alpha M_{\rm en}+2\varepsilon_0\right]\left(\frac r{r_0}\right)^\alpha+C_\delta r+Cr^{p_0}.
\end{align}

It remains to choose $r_0$.  If $d_0(\mathsf x,\mathsf y)<1$, then \eqref{eq:QS-basic-cost-internal} gives
\begin{align*}
 d_0(\mathsf x,\mathsf y)
 =\left(\frac r{r_0}\right)^\alpha\e^{\alpha\chi[\Vv(\mathsf x)+\Vv(\mathsf y)]}
 \ge\left(\frac r{r_0}\right)^\alpha,
\end{align*}
and hence $r<r_0$.  Since $0<\alpha\le1$, $p_0>\alpha$, and $r\le r_0$,
\begin{align*}
 \frac r{d_0(\mathsf x,\mathsf y)}\le r_0,\qquad
 \frac{r^{p_0}}{d_0(\mathsf x,\mathsf y)}\le r_0^{p_0}.
\end{align*}
Choose $r_0\le r_\delta$ sufficiently small that
\begin{align*}
 C_\delta r_0+Cr_0^{p_0}<2\varepsilon_0.
\end{align*}
Then \eqref{eq:QS-low-cost-preliminary} and \eqref{eq:QS-q-choice} give
\begin{align*}
 \E d_0(X,Y)
 \le\left(q^\alpha M_{\rm en}+4\varepsilon_0\right)d_0(\mathsf x,\mathsf y)
 =:\theta_{\rm lo}d_0(\mathsf x,\mathsf y),
\end{align*}
where $\theta_{\rm lo}<1$.  Since $(X,Y)$ is a coupling of $P(\mathsf x,\cdot)$ and $P(\mathsf y,\cdot)$, this proves \eqref{eq:QS-low-contraction}.
\end{proof}

\subsection{Proof of the spectral gap criterion}
\begin{proof}[Proof of \Cref{thm:stable-compact-spectral-gap}]
Let $r_0$ and $d_0$ be supplied by \Cref{prop:QS-low-energy-contraction}.

\medskip
\noindent\emph{Step 1. Global $d_0$-contraction.}
If $\Vv(\mathsf x)+\Vv(\mathsf y)\le R_*$ and $d_0(\mathsf x,\mathsf y)<1$, \Cref{prop:QS-low-energy-contraction} gives
\begin{align*}
 \Wass_{d_0}(P(\mathsf x,\cdot),P(\mathsf y,\cdot))
 \le\theta_{\rm lo}d_0(\mathsf x,\mathsf y).
\end{align*}
If $\Vv(\mathsf x)+\Vv(\mathsf y)\ge R_*$ and $d_0(\mathsf x,\mathsf y)<1$, synchronous coupling and \eqref{eq:QS-selected-high-energy} give
\begin{align*}
 \Wass_{d_0}(P(\mathsf x,\cdot),P(\mathsf y,\cdot))
 \le\frac14d_0(\mathsf x,\mathsf y).
\end{align*}
Thus, with $\theta=\max\{\theta_{\rm lo},1/4\}<1$,
\begin{align*}
 \Wass_{d_0}(P(\mathsf x,\cdot),P(\mathsf y,\cdot))
 \le\theta d_0(\mathsf x,\mathsf y),\qquad d_0(\mathsf x,\mathsf y)<1.
\end{align*}
For $d_0(\mathsf x,\mathsf y)=1$, the bound by one is automatic.  Hence $P$ is globally nonexpanding in $\Wass_{d_0}$, and so is every power $P^n$.

\medskip
\noindent\emph{Step 2. A $d_0$-small Lyapunov set.}
Set $L=\e^{\lambda_0\Vv}$.  Since
$\vartheta(t)\to0$, choose $t_1<\infty$ so that
$\vartheta(t)\le1/2$ for $t\ge t_1$.  Then
\eqref{eq:QS-Lyapunov} and Young's inequality give
\begin{align}\label{eq:QS-selected-Foster}
 P_tL\le C L^{1/2}\le\frac12L+K,
 \qquad t\ge t_1,
\end{align}
where $K=C^2/2$.  Choose
\begin{align*}
 L_*>\max\{\e^{\lambda_0},8K\},\qquad
 \ell_*=\lambda_0^{-1}\log L_*.
\end{align*}
Let $\mathsf x_*$ be the point in \Cref{ass:QS-accessibility}.  By
continuity of $\Vv$ on $H$ and the continuous embedding
$H\hookrightarrow\mathbb H$, choose $\varepsilon>0$ so small that,
with $B=B_H(\mathsf x_*,\varepsilon)$,
\begin{align*}
 \operatorname{diam}_{d_0}B<1.
\end{align*}
By \eqref{eq:QS-accessibility}, choose $N\in\N$ so large that
\begin{align*}
 NT_*\ge\max\{t_1,t_{\ell_*,\varepsilon}\}.
\end{align*}
Then
\begin{align*}
 p_*:=\inf_{L(\mathsf x)\le L_*}P^N(\mathsf x,B)>0.
\end{align*}
Since $P^N=P_{NT_*}$ is $d_0$-nonexpanding, coupling mass $p_*$ inside
$B$ and the remaining mass arbitrarily gives
\begin{align*}
 \sup_{L(\mathsf x),L(\mathsf y)\le L_*}
 \Wass_{d_0}(P^N(\mathsf x,\cdot),P^N(\mathsf y,\cdot))
 \le1-p_*\bigl(1-\operatorname{diam}_{d_0}B\bigr)<1.
\end{align*}
Thus $\{L\le L_*\}$ is $d_0$-small for $Q=P^N$.

\medskip
\noindent\emph{Step 3. Weighted weak Harris.}
By \eqref{eq:QS-selected-Foster} and the choice of $N$,
\begin{align*}
 QL\le\frac12L+K.
\end{align*}
The $d_0$-contraction from Step~1, the preceding Foster estimate, and the
$d_0$-smallness from Step~2 are precisely the hypotheses of the weighted
weak Harris theorem \cite[Theorem~4.8]{HairerMattinglyScheutzow2011}.
Hence there are $\beta>0$, $0<\lambda\le\lambda_0$, and
$\varrho\in(0,1)$ such that the weighted cost $d$ in
\eqref{eq:QS-costs} satisfies
\begin{align*}
 \Wass_d(\mu_1Q^n,\mu_2Q^n)\le\varrho^n\Wass_d(\mu_1,\mu_2),\qquad n\ge0.
\end{align*}
Taking $\tau=NT_*$ proves \eqref{eq:QS-weighted-gap}.
\end{proof}

\section{Spectral gap of the wave equation}\label{sec:wave-application}

This section is devoted to applying \Cref{thm:stable-compact-spectral-gap} to \eqref{eq:main-spde}.  The argument follows the same order as the abstract criterion.  We first keep the existence and invariant-law regularity theory separate from the coupling.  We then verify the stable--compact structure and the dense range of the Malliavin derivative on the weak phase $\Es{s}$, establish approximate controllability and the uniform irreducibility needed for the Harris step, and give the geometric characterizations of saturation.  The final subsection verifies the remaining quantitative hypotheses and proves the main theorems.

\subsection{Existence and invariant-law regularity}
We first establish existence by norm topology asymptotic compactness and deduce the stronger Sobolev regularity enjoyed by invariant laws.  The latter result is not used to obtain the spectral gap or uniqueness; it enters only in the final bootstrap to the energy topology.  

\subsubsection{Asymptotic compactness and existence}
Existence of invariant measures for closely related damped stochastic wave
equations with polynomial nonlinearities is classical; see
\cite{BarbuDaPrato2002,Kim2004,BrzezniakOndrejatSeidler2016} for example.  We include
the argument below because the present periodic setting with variable
damping admits the stronger norm topology asymptotic compactness statement
in \Cref{lem:strong-asymptotic-compactness}.

We set (recall that $\Lambda=(I-\Delta)^{1/2}$)
\begin{align}\label{eq:AFB}
\mathbb A\binom uv=\binom v{-\Lambda^2u-av},
 \qquad
 \mathbb F\binom uv=\binom0{u-u^3},
 \qquad
 \mathbb B h=\binom0{Bh},
\end{align}
and $S_a(t)$ the semigroup generated by $\mathbb A$. 

For a scalar function $f$, write $\mathsf M_f$ for the multiplication operator $g\mapsto fg$.

\begin{lemma}
\label{lem:linear-semigroup}
For every $-2\le s\le2$, the operator $\mathbb A$ generates a strongly
continuous semigroup $S_a(t)$ on $\Hs{s}$.  There are $M_s\ge1$ and
$\varpi_s>0$ such that
\begin{align}\label{eq:semigroup-decay-scale}
 \norm{S_a(t)}_{\cL(\Hs{s})}\le M_s\e^{-\varpi_st},
 \qquad t\ge0.
\end{align}
The constants may be chosen locally uniformly for $s$ in this interval.
\end{lemma}
\begin{proof}
Generation follows from the standard massive wave group and the bounded
perturbation theorem, see
\cite{Pazy1983}.  Since $a\ge a_0>0$, the
energy level damped wave semigroup is exponentially stable, which follows
from the usual modified-energy argument and is also a special case of the
classical stabilization theory on compact manifolds
\cite{Lebeau1996}.

For the positive scale, commute the equation successively with
$\Lambda^k$, $k=1,2$.  The commutator
$[\Lambda^k,\mathsf M_a]$ has order $k-1$.  The energy estimate for the
commuted equation and induction in $k$ therefore give exponential decay
on $\Hs{k}$; the convolution of the lower-order exponentially decaying
commutator source with the energy semigroup contributes at most
$C(1+t)\e^{-ct}$ and can be absorbed by decreasing the decay rate.
Complex interpolation gives the assertion for $0\le s\le2$.

For the negative scale, use the phase identification
\eqref{eq:phase-identification}.  In these coordinates the generator is
\begin{align*}
 \widetilde{\mathbb A}
 =\begin{pmatrix}0&\Lambda\\-\Lambda&-\mathsf M_a\end{pmatrix}
 \quad\text{on }H^s\times H^s.
\end{align*}
Its $L^2\times L^2$ adjoint satisfies
\begin{align*}
 \widetilde{\mathbb A}^{\,*}
 =\mathbb J\widetilde{\mathbb A}\mathbb J,
 \qquad
 \mathbb J(p,q)=(p,-q).
\end{align*}
Thus the already proved positive scale estimate applies equally to the
adjoint semigroup.  Duality between $H^s\times H^s$ and
$H^{-s}\times H^{-s}$ then gives exponential decay on $\Hs{-s}$ for
$0\le s\le2$.  Local uniformity follows from interpolation on compact
subintervals.
\end{proof}

We need the following conditional block moments for later use.  For a real
starting time $r$, put $I=[r,r+1]$ and
\begin{align*}
 U_r=\norm u_{L^5(I;L^{10})},
 \qquad
 H_r=\sup_{t\in I}(1+\Vv(X_t)),
 \qquad
 L_r=1+\Vv(X_r).
\end{align*}
where recall that $X_r=(u_r,v_r)$ and  $\Vv$ is the modified energy \eqref{eq:V-energy}.
\begin{proposition}\label{prop:strichartz}
For every real $p\ge1$ and every starting time $r$ for which the solution
is defined on $I = [r,r+1]$,
\begin{align}\label{eq:block-moment-estimates}
 \E[H_r^p\mid\mathcal F_r]\le C_pL_r^p,
 \qquad
 \E[U_r^p\mid\mathcal F_r]\le C_pL_r^{3p/2},
\end{align}
with constant $C_p$ independent of $r$.
\end{proposition}

\begin{proof}
The first estimate is the conditional form of
\eqref{eq:finite-horizon-energy-moment}.  On $[r,r+1]$, split
$X=Y+Z_r$, where
\begin{align*}
Z_r(t)=\int_r^tS_a(t-s)\mathbb B\,\dd W_s,
 \qquad r\le t\le r+1.
\end{align*}
Write $Z_r=(Z_r^{(1)},Z_r^{(2)})$ and set
$Y=X-Z_r=(y,\partial_t y)$, where
$y=u-Z_r^{(1)}$ and
$\partial_t Z_r^{(1)}=Z_r^{(2)}$, then 
\begin{align*}
 \partial_t^2 y+\Lambda^2y+a\partial_t y=u-u^3,
 \qquad (y,\partial_t y)(r)=X_r,
 \qquad u=y+Z_r^{(1)}.
\end{align*}

Theorem~1 of \cite{CacciafestaDanesiMeng2024} gives, for the
wave-admissible pair $(p,q)=(5,10)$ in dimension three,
\begin{align*}
\norm z_{L^5(I;L^{10})}
 \le C\Bigl(
 \norm{(z,\partial_t z)(r)}_{\Hh}
 +\norm G_{L^1(I;L^2)}
 \Bigr)
\end{align*}
for $\partial_t^2z+\Lambda^2z=G$. 
Taking $G=F-a\partial_tz$  yields
\begin{align}\label{eq:unit-block-Strichartz}
\norm z_{L^5(I;L^{10})}
 \le C\Bigl(
 \norm{(z,\partial_tz)(r)}_{\Hh}
 +\norm F_{L^1(I;L^2)}
 \Bigr),
\end{align}
where $C$ depends on the fixed damping $a$ but is independent of $r$.

Choose $\sigma\in(1/5,1)$.  Since the forcing profiles are smooth, the
finite rank stochastic convolution $Z_r$ has moments of every order in
$C(I;\Hs{\sigma})$. The Sobolev
embedding $H^{1+\sigma}(\mathbb T^3)\hookrightarrow L^{10}(\mathbb T^3)$
gives
\begin{align*}
\norm{Z_r^{(1)}}_{L^5(I;L^{10})}
\le C\norm{Z_r}_{C(I;\Hs{\sigma})}.
\end{align*}
By $H^1(\mathbb T^3)\hookrightarrow L^6(\mathbb T^3)$ and
\Cref{lem:periodic-energy-coercivity},
\begin{align*}
\norm{u-u^3}_{L^2}\le C\Vv(X)^{3/2}.
\end{align*}
It then follows from 
\eqref{eq:unit-block-Strichartz}  that 
\begin{align*}
U_r\le C\bigl(1+L_r^{1/2}+H_r^{3/2}
+\norm{Z_r}_{C(I;\Hs{\sigma})}\bigr).
\end{align*}
Therefore, for every $p\ge1$,
\begin{align*}
\E[U_r^p\mid\mathcal F_r]
&\le C_p\Bigl(
 1+L_r^{p/2}
 +\E[H_r^{3p/2}\mid\mathcal F_r]
 +\E[\norm{Z_r}_{C(I;\Hs{\sigma})}^p\mid\mathcal F_r]
 \Bigr) \le C_pL_r^{3p/2}.
\end{align*}
\end{proof}

The following lemma gives the compactness needed for the existence of
invariant measures.  The argument is a unit block version of the classical
decaying--regular decomposition used in the asymptotic compactness theory
for damped wave equations.

\begin{lemma}
\label{lem:strong-asymptotic-compactness}
Fix $\mathsf x\in\Hh$ and $0<s_0<2/5$.  For every $\varepsilon>0$ there exists a
compact set $K_{\mathsf x,\varepsilon}\subset\Hh$ such that
\begin{align}\label{eq:strong-asymptotic-tightness}
 \inf_{t\ge0}P_t(\mathsf x,K_{\mathsf x,\varepsilon})\ge1-\varepsilon.
\end{align}
Moreover, for every $\mathsf y\in\Hh$, writing
\begin{align*}
 X_t=S_a(t)\mathsf y+Y_t,
\end{align*}
for every $1\le p<\infty$ there are $C_p,N_p<\infty$ such that
\begin{align}\label{eq:strong-remainder-energy-uniform}
 \sup_{t\ge0}\E\norm{Y_t}_{\Hs{s_0}}^p
 \le C_p(1+\Vv(\mathsf y))^{N_p}.
\end{align}
\end{lemma}

\begin{proof}
Let $X_t=(u_t,v_t)$ be the solution issued from $\mathsf x$ and write
\begin{align*}
 X_t=Z_t+Y_t,
 \qquad
 Z_t=S_a(t)\mathsf x,
\end{align*}
where
\begin{align*}
Y_t
={}&
 \int_0^tS_a(t-s)\mathbb F(X_s)\,\dd s
 +\int_0^tS_a(t-s)\mathbb B\,\dd W_s.
\end{align*}
We first show that, for every $1\le p<\infty$,
\begin{align}\label{eq:strong-remainder-uniform-moment}
 \sup_{t\ge0}\E \norm{Y_t}_{\Hs{s_0}}^p<\infty.
\end{align}
By the fractional Leibniz rule \cite{KatoPonce1988} and
\begin{align*}
 H^1(\mathbb T^3)\hookrightarrow W^{s_0,10/3}(\mathbb T^3),
 \qquad s_0<2/5,
\end{align*}
we have
\begin{align*}
 \norm{u^3}_{H^{s_0}}
 \le C\norm u_{L^{10}}^2\norm u_{H^1}.
\end{align*}
Hence, on every unit interval $I=[r,r+1]$,
\begin{align*}
\int_I\norm{\mathbb F(X_s)}_{\Hs{s_0}}\,\dd s
&\le
 C\left(
 H_r^{1/2}+H_r^{1/2}U_r^2
 \right).
\end{align*}
By H\"older's inequality, \eqref{eq:block-moment-estimates}, and
\eqref{eq:poly-Foster}, there are $C_p,N_p<\infty$ such that
\begin{align*}
\sup_{r\ge0}
\E \left(
 \int_r^{r+1}\norm{\mathbb F(X_s)}_{\Hs{s_0}}\,\dd s
\right)^p
\le C_p(1+\Vv(\mathsf x))^{N_p}.
\end{align*}
The constants are independent of the starting time.
Using \eqref{eq:semigroup-decay-scale}, we therefore obtain
\begin{align*}
\norm{
 \int_0^tS_a(t-s)\mathbb F(X_s)\,\dd s
}_{\Hs{s_0}}
&\le
 C\sum_{k\ge0}\e^{-\varpi_{s_0}k}
 \int_{(t-k-1)\vee0}^{(t-k)\vee0}
 \norm{\mathbb F(X_s)}_{\Hs{s_0}}\,\dd s.
\end{align*}
Every nonempty interval on the right has length at most one and is
contained in a unit interval.  Minkowski's inequality and the preceding
uniform block bound give
\begin{align*}
\sup_{t\ge0}
\E \norm{
 \int_0^tS_a(t-s)\mathbb F(X_s)\,\dd s
}_{\Hs{s_0}}^p<\infty.
\end{align*}
This proves \eqref{eq:strong-remainder-uniform-moment} and
\eqref{eq:strong-remainder-energy-uniform}, since the stochastic
convolution has moments of every order uniformly in $t$.  Since
$\Hs{s_0}\Subset\Hh$, the set
\begin{align*}
 K_R^1
 =
 \overline{\{z\in\Hs{s_0}:\norm z_{\Hs{s_0}}\le R\}}^{\,\Hh}
\end{align*}
is compact in $\Hh$.  By
\eqref{eq:strong-remainder-uniform-moment} and Markov's inequality, $R$
may be chosen so that
\begin{align}\label{eq:strong-remainder-compact-probability}
 \inf_{t\ge0}\Prob \{Y_t\in K_R^1\}\ge1-\varepsilon.
\end{align}

It remains to treat the rough homogeneous part.  Set
\begin{align*}
 K_{\mathsf x}^0=\{S_a(t)\mathsf x:t\ge0\}\cup\{0\}.
\end{align*}
This set is compact in $\Hh$.  Indeed, for any sequence
$S_a(t_n)\mathsf x$, either $(t_n)$ has a bounded subsequence, in which case
strong continuity of $S_a$ gives a convergent subsequence, or
$t_n\to\infty$, in which case \eqref{eq:semigroup-decay-scale} gives
\begin{align*}
 \norm{S_a(t_n)\mathsf x}_{\Hh}\longrightarrow0.
\end{align*}
Thus $Z_t=S_a(t)\mathsf x\in K_{\mathsf x}^0$ for every $t\ge0$. Set
\begin{align*}
 K_{\mathsf x,\varepsilon}=K_{\mathsf x}^0+K_R^1,
\end{align*}
which is compact in $\Hh$.  Since
$X_t=Z_t+Y_t$ and $Z_t\in K_{\mathsf x}^0$ deterministically,
\eqref{eq:strong-remainder-compact-probability} gives
\begin{align*}
 P_t(\mathsf x,K_{\mathsf x,\varepsilon})
 \ge\Prob \{Y_t\in K_R^1\}
 \ge1-\varepsilon
\end{align*}
for every $t\ge0$.  This proves
\eqref{eq:strong-asymptotic-tightness}.
\end{proof}

\begin{proposition}\label{prop:existence}
The semigroup $P_t$ of \eqref{eq:main-spde} has at least one invariant probability measure.
\end{proposition}

\begin{proof}
Fix $\mathsf x\in\Hh$ and define the Krylov--Bogoliubov averages
\begin{align*}
 \nu_T=\frac1T\int_0^T P_t(\mathsf x,\cdot)\,\dd t,
 \qquad T>0.
\end{align*}
By \Cref{lem:strong-asymptotic-compactness}, for every
$\varepsilon>0$ there exists a compact set
$K_{\mathsf x,\varepsilon}\subset\Hh$ such that
\begin{align*}
 \inf_{t\ge0}P_t(\mathsf x,K_{\mathsf x,\varepsilon})\ge1-\varepsilon.
\end{align*}
Consequently,
\begin{align*}
 \nu_T(K_{\mathsf x,\varepsilon})
 =\frac1T\int_0^T
 P_t(\mathsf x,K_{\mathsf x,\varepsilon})\,\dd t
 \ge1-\varepsilon
\end{align*}
for every $T>0$.  Hence the family
$\{\nu_T:T>0\}$ is tight in the norm topology of $\Hh$.  Prokhorov's theorem yields a sequence
$T_n\to\infty$ and a probability measure $\nu$ such that
\begin{align*}
 \nu_{T_n}\Longrightarrow\nu.
\end{align*}
Then $\nu$ is an invariant
probability measure for $P_t$.
\end{proof}

\subsubsection{Stationary Sobolev regularization}
The wave semigroup does not smooth a deterministic initial condition at
positive time.  For an invariant law, however, the homogeneous component
can be pushed to arbitrarily large positive times, where exponential
stability makes it disappear, while the forced remainder is uniformly
bounded in a stronger Sobolev space.  A second multiplication bootstrap
then raises the stationary regularity above one half.

\begin{proposition}
\label{prop:stationary-regularization}
For each $s_*\in(1/2,1)$, every invariant
probability $\nu$ satisfies \eqref{eq:main-strong-support}.
\end{proposition}

\begin{proof}
Fix $1/2<s_*<1$ and choose
\begin{align*}
 \frac{s_*}{3}<s_0<\frac25.
\end{align*}
Then $s_0>1/6$ and $s_*<3s_0$.  Let $X_0$ have law $\nu$, be independent of $W$, and let $X$ be the corresponding forward solution.  By invariance, $X$ is stationary.
Write
\begin{align*}
 X_t=S_a(t)X_0+Y_t,
\end{align*}
where
\begin{align*}
Y_t
={}&
 \int_0^tS_a(t-r)\mathbb F(X_r)\,\dd r
 +\int_0^tS_a(t-r)\mathbb B\,\dd W_r.
\end{align*}

\medskip
\noindent\emph{Step 1. Gain below $2/5$.}
By \eqref{eq:strong-remainder-energy-uniform}, conditioning on
$X_0\sim\nu$ and using the polynomial moments of invariant laws from
\Cref{prop:lyapunov}, for every $p\ge1$,
\begin{align}\label{eq:forward-remainder-s0}
 \sup_{t\ge0}\E \norm{Y_t}_{\Hs{s_0}}^p
 &\le C_p\int_{\Hh}(1+\Vv(\mathsf x))^{N_p}\,\nu(\dd\mathsf x)<\infty.
\end{align}
On the other hand, \eqref{eq:semigroup-decay-scale} and
\Cref{prop:lyapunov} give
\begin{align*}
 \E\norm{S_a(t)X_0}_{\Hh}^p
 \le C_p\e^{-p\varpi_0t}\E\norm{X_0}_{\Hh}^p
 \longrightarrow0.
\end{align*}
Since $\Law(X_t)=\nu$ and
$X_t=S_a(t)X_0+Y_t$, it follows that
\begin{align*}
 \Law(Y_t)\Longrightarrow\nu
 \qquad\text{in }\Hh.
\end{align*}
For $s>0$, extend $\mathsf x\mapsto\norm{\mathsf x}_{\Hs{s}}^p$
to $\Hh$ by $+\infty$ outside $\Hs{s}$.  This function is lower
semicontinuous on $\Hh$, so Portmanteau's theorem and
\eqref{eq:forward-remainder-s0} yield
\begin{align}\label{eq:intermediate-stationary-gain}
 \nu(\Hs{s_0})=1,\qquad
 \int_{\Hs{s_0}}\norm{\mathsf x}_{\Hs{s_0}}^p\,\nu(\dd\mathsf x)<\infty,
 \qquad 1\le p<\infty.
\end{align}

\medskip
\noindent\emph{Step 2. Bootstrap above one half.}
Since $s_*<3s_0$, choose $\beta$ such that
\begin{align*}
 s_*+\frac12-s_0<\beta<\frac12+2s_0.
\end{align*}
Standard Sobolev multiplication estimates in dimension three
give
\begin{align*}
 \norm{u^2}_{H^\beta}
 \le C\norm u_{H^{1+s_0}}^2,
 \qquad
 \norm{u^3}_{H^{s_*}}
 \le C\norm{u^2}_{H^\beta}\norm u_{H^{1+s_0}},
\end{align*}
and hence
\begin{align*}
 \norm{u^3}_{H^{s_*}}
 \le C\norm u_{H^{1+s_0}}^3.
\end{align*}
By \eqref{eq:intermediate-stationary-gain}, stationarity, and H\"older's
inequality, for every unit interval $I=[r,r+1]$,
\begin{align*}
 \sup_{r\ge0}\E\left(
 \int_r^{r+1}\norm{u_t^3}_{H^{s_*}}\,\dd t
 \right)^p<\infty.
\end{align*}
Since $s_*<1$, the linear part satisfies $u\in H^1\hookrightarrow H^{s_*}$,
and the stochastic convolution is smooth in space.  Repeating the
exponentially weighted unit block estimate in $\Hs{s_*}$ gives
\begin{align}\label{eq:forward-remainder-sstar}
 \sup_{t\ge0}\E\norm{Y_t}_{\Hs{s_*}}^p<\infty.
\end{align}
Using again $\Law(Y_t)\Longrightarrow\nu$ in $\Hh$ and the lower
semicontinuity of the extended $\Hs{s_*}$-norm, Portmanteau's theorem and
\eqref{eq:forward-remainder-sstar} give \eqref{eq:main-strong-support}.
\end{proof}

\subsection{Stable--compact structure}

Fix $0<s<1/2$ and a block length $T>0$.  We use the canonical Wiener and Cameron--Martin spaces
\begin{align}\label{eq:wiener-cm}
 E_T=C_0([0,T];\R^m),\qquad H_T=\{h\in H^1([0,T];\R^m):h(0)=0\},
\end{align}
with Wiener measure $\gamma_T$, canonical Cameron--Martin embedding $\iota_T:H_T\hookrightarrow E_T$, and norm $\norm h_{H_T}=\norm{\partial_t h}_{L^2}$.  Subtracting the additive path, $\bar v=v-B\omega$, turns \eqref{eq:main-spde} into
\begin{align}\label{eq:pathwise-shifted-wave}
 \partial_tu=\bar v+B\omega,\qquad \partial_t\bar v=-\Lambda^2u+u-a(\bar v+B\omega)-u^3.
\end{align}
For $\mathsf x=(u_0,v_0)\in\Hh$ and $\omega\in E_T$, let $X=(u,v)$ be the corresponding pathwise solution of \eqref{eq:pathwise-shifted-wave}, with $v=\bar v+B\omega$, and set
\begin{align*}
 \Phi_T(\mathsf x,\omega)=X_T.
\end{align*}
Then $\Phi_T:\Hh\times E_T\to\Hh$ is Borel and agrees $\gamma_T$-almost surely with the stochastic endpoint.

Let $J_{r,t}$ denote the homogeneous linearized solution operator along $X$:
\begin{align*}
 \partial_t^2z+\Lambda^2z+a\partial_tz+(3u_t^2-1)z=0.
\end{align*}
The Malliavin derivative is
\begin{align}\label{eq:Malliavin-endpoint}
 \mathcal D\Phi_T(\mathsf x,\omega)h=\cA_Th=\int_0^TJ_{r,T}\binom0{B\partial_rh(r)}\,\dd r,\qquad h\in H_T.
\end{align}

For $\mathsf x,\mathsf y\in\Hh$ and a common driver $\omega$, let $X=(u,v)$ and $\widetilde X=(\widetilde u,\widetilde v)$ be the corresponding solutions and set
\begin{align*}
 \mathscr W=u^2+u\widetilde u+\widetilde u^2.
\end{align*}
The exact difference $w=\widetilde u-u$ solves
\begin{align}\label{gap:eq:exact-difference}
 \ddot w-\Delta w+a\dot w+\mathscr W w=0.
\end{align}
Once
$(\mathsf x,\mathsf y,\omega)$ is fixed, the coefficient $\mathscr W$ is
frozen and \eqref{gap:eq:exact-difference} is a linear wave equation.  For
$h\in\Es{s}$, let $z^h$ denote its solution with initial phase
$(z^h_0,\partial_t z^h_0)=h$, and define the corresponding endpoint propagator by
\begin{align*}
 \mathcal J_T(\mathsf x,\mathsf y,\omega)h
 =(z_T^h,\partial_t z_T^h).
\end{align*}
In particular, taking $h=\mathsf y-\mathsf x$ gives the exact endpoint
difference identity
\begin{align*}
 \Phi_T(\mathsf y,\omega)-\Phi_T(\mathsf x,\omega)
 =\mathcal J_T(\mathsf x,\mathsf y,\omega)(\mathsf y-\mathsf x).
\end{align*}
Relative to the massive damped-wave semigroup,
\begin{align}\label{gap:eq:weak-stable-compact}
 \mathcal J_T(\mathsf x,\mathsf y,\omega)
 =S_a(T)+\mathcal K_T(\mathsf x,\mathsf y,\omega),\qquad
 \mathcal K_T(\mathsf x,\mathsf y,\omega)h
 =\int_0^TS_a(T-t)\binom0{(1-\mathscr W_t)z_t^h}\,\dd t.
\end{align}

The weak phase $\Es{s}=\Hs{-1-s}$ is adapted to the exact difference equation. Since $s\in(0,1/2)$, for $f,g\in H^1$ and $h\in H^{-s}$,
\begin{align}\label{gap:eq:negative-trilinear}
 \norm{fgh}_{H^{-1-s}}\le C_s\norm f_{H^1}\norm g_{H^1}\norm h_{H^{-s}}.
\end{align}
Equivalently, by duality,
\begin{align}\label{gap:eq:positive-trilinear}
 \norm{fg\phi}_{H^s}\le C_s\norm f_{H^1}\norm g_{H^1}\norm\phi_{H^{1+s}}.
\end{align}
These estimates follow from the fractional Leibniz rule and Sobolev embedding; for instance one combines $H^1\hookrightarrow L^6$ with $H^1\hookrightarrow W^{s,6/(1+2s)}$.

The next proposition gives the stable--compact decomposition required in
\Cref{ass:QS-stable-compact}\textup{(i)}.  Strong measurability in
\Cref{ass:QS-stable-compact}\textup{(i)} follows from the Borel dependence
of the pathwise solution and the representation
\eqref{gap:eq:weak-stable-compact}.  To keep the main argument focused, the verification of
parts~\textup{(ii)}--\textup{(iii)} is deferred to
\Cref{gap:lem:weak-core-operators,gap:lem:wave-twojet-weak}.

\begin{proposition}\label{prop:wave-stable-compact-structure}
For every $0<s<1/2$, the semigroup $S_a(t)$ is exponentially stable on $\Es{s}$, and for every $T>0$, $\mathsf x,\mathsf y\in\Hh$, and $\omega\in E_T$, the operator $\mathcal K_T(\mathsf x,\mathsf y,\omega)$ in \eqref{gap:eq:weak-stable-compact} is compact on $\Es{s}$.
\end{proposition}

\begin{proof}
The exponential stability of $S_a(t)$ on
$\Es{s}=\Hs{-1-s}$ follows from \Cref{lem:linear-semigroup}.  It remains
to prove compactness of $\mathcal K_T$.  Fix
\begin{align*}
 s<s_1<\frac12,\qquad 0<\theta_0<\frac15.
\end{align*}
Writing the exact difference equation \eqref{gap:eq:exact-difference}
relative to $S_a(t)$ gives, for
$\mathcal J_th=(z_t^h,\partial_tz_t^h)$,
\begin{align*}
 \mathcal J_th
 =S_a(t)h+\int_0^tS_a(t-r)
 \binom0{(1-\mathscr W_r)z_r^h}\,\dd r.
\end{align*}
At the weak level, \eqref{gap:eq:negative-trilinear} yields
\begin{align*}
 \norm{(1-\mathscr W_r)z_r^h}_{H^{-1-s_1}}
 \le C[1+\Vv(X_r)+\Vv(\widetilde X_r)]
 \norm{\mathcal J_rh}_{\Es{s_1}}.
\end{align*}
Since $S_a(t)$ is bounded on $\Es{s_1}$, Gronwall's inequality gives
\begin{align*}
 \sup_{t\le T}\norm{\mathcal J_t}_{\cL(\Es{s_1})}\le C\exp\left\{
 C_T\int_0^T[1+\Vv(X_t)+\Vv(\widetilde X_t)]\,\dd t
 \right\}<\infty,
\end{align*}
and hence $\mathcal K_T:\Es{s_1}\to\Es{s_1}$. The same argument at the energy level uses
$H^1\hookrightarrow L^6$ and
\begin{align*}
 \norm{\mathscr W_rz_r^h}_{L^2}
 \le C\bigl(\norm{u_r}_{H^1}^2+\norm{\widetilde u_r}_{H^1}^2\bigr)
 \norm{z_r^h}_{H^1},
\end{align*}
and therefore gives
\begin{align*}
 \sup_{t\le T}\norm{\mathcal J_th}_{\Hh}
 \le C_{T,\mathsf x,\mathsf y,\omega}\norm h_{\Hh}.
\end{align*}

We now gain positive regularity in the defect.  The fractional Leibniz
rule, together with
$H^1\hookrightarrow W^{\theta_0,10/3}\cap W^{\theta_0,30/7}$ for
$\theta_0<1/5$, gives
\begin{align*}
 \norm{\mathscr W_tz_t^h}_{H^{\theta_0}}
 \le C_{\theta_0}\Bigl(
 \norm{u_t}_{L^{10}}^2+\norm{\widetilde u_t}_{L^{10}}^2
 +\norm{u_t}_{H^1}^2+\norm{\widetilde u_t}_{H^1}^2
 \Bigr)\norm{z_t^h}_{H^1}.
\end{align*}
For the fixed trajectories $X$ and $\widetilde X$, the finite time energy
bounds and the Strichartz estimates in \Cref{prop:strichartz} make the coefficient on the
right integrable on $[0,T]$.  Together with the preceding energy level
bound for $\mathcal J_th$, this gives
\begin{align*}
 \int_0^T\norm{(1-\mathscr W_t)z_t^h}_{H^{\theta_0}}\,\dd t
 \le C_{T,\mathsf x,\mathsf y,\omega}\norm h_{\Hh}.
\end{align*}
Hence, by \eqref{gap:eq:weak-stable-compact} and the boundedness of
$S_a(t)$ on $\Hs{\theta_0}$,
\begin{align*}
 \norm{\mathcal K_T(\mathsf x,\mathsf y,\omega)h}_{\Hs{\theta_0}}
 \le C_{T,\mathsf x,\mathsf y,\omega}\norm h_{\Hh},
\end{align*}
Interpolating this bound with
$\mathcal K_T:\Es{s_1}\to\Es{s_1}$ at
$\lambda=(1+s)/(1+s_1)$ yields
\begin{align*}
 \mathcal K_T:\Es{s}\longrightarrow
 H^{-s+\gamma}\times H^{-1-s+\gamma},
 \qquad
 \gamma=(1-\lambda)\theta_0
 =\theta_0\frac{s_1-s}{1+s_1}>0.
\end{align*}
The latter space is compactly embedded into $\Es{s}$ by Rellich's theorem,
and therefore $\mathcal K_T(\mathsf x,\mathsf y,\omega)$ is compact on
$\Es{s}$.
\end{proof}

\subsection{Dense range of the Malliavin derivative}

We next verify the qualitative H\"ormander condition required in
\Cref{ass:QS-Malliavin}.  The argument is based on backward exact zero
propagation, while the weaker coupling phase provides the additional
dual regularity needed to close the propagation scheme.

Under the phase identification \eqref{eq:phase-identification}, $\Es{s}$ is mapped isometrically onto $H^{-1-s}\times H^{-1-s}$, so an annihilating terminal functional is represented by $(p_T,q_T)\in H^{1+s}\times H^{1+s}$.  The corresponding backward adjoint equation is
\begin{align}\label{gap:eq:rough-adjoint}
 \dot p=(\Lambda-\Lambda^{-1})q+3\Lambda^{-1}(u^2q),\qquad \dot q=-\Lambda p+aq.
\end{align}

\subsubsection{Exact-zero backward propagation}
We first derive the time regularity needed in the backward propagation
argument.
\begin{lemma}\label{lem:weak-adjoint-time-regularity}
Fix $T>0$, an energy solution $X=(u,v)$ on $[0,T]$, and a solution $(p,q)$ of \eqref{gap:eq:rough-adjoint} with terminal value in $H^{1+s}\times H^{1+s}$.  Writing
\begin{align*}
 \bar v_t=v_t-\sum_{j=1}^mb_jW^j(t),
\end{align*}
one has
\begin{align}\label{gap:eq:energy-time-reg}
 u&\in C H^1\cap W^{1,\infty}L^2,& \bar v&\in C L^2\cap W^{1,\infty}H^{-1},
\end{align}
and
\begin{align*}
 (p,q)\in C(H^{1+s}\times H^{1+s})\cap W^{1,\infty}(H^s\times H^s).
\end{align*}
Consequently, with $h=|t-r|$,
\begin{align}\label{gap:eq:adjoint-rough-time}
 \norm{u_t-u_r}_{H^\rho}&\le Ch^{1-\rho},&&0\le\rho\le1,\notag\\
 \norm{\bar v_t-\bar v_r}_{H^{-\rho}}&\le Ch^\rho,&&0\le\rho\le1,\\
 \norm{p_t-p_r}_{H^\beta}+\norm{q_t-q_r}_{H^\beta}&\le Ch^{1+s-\beta},&&s\le\beta\le1+s.\notag
\end{align}
\end{lemma}

\begin{proof}
The shifted equation \eqref{eq:pathwise-shifted-wave} gives \eqref{gap:eq:energy-time-reg}.  The product estimate \eqref{gap:eq:positive-trilinear} yields
\begin{align*}
 \norm{u^2q}_{H^s}\le C_s\norm u_{H^1}^2\norm q_{H^{1+s}},
\end{align*}
so variation of constants and Gronwall applied to \eqref{gap:eq:rough-adjoint} give the asserted adjoint regularity.  The estimates in \eqref{gap:eq:adjoint-rough-time} then follow by interpolation between the corresponding uniform and Lipschitz bounds.  All constants are finite pathwise.
\end{proof}

For each $T>0$, let $\Omega_{\mathrm{sep},T}\subset E_T$ be the coefficient independent full measure event supplied by \Cref{lem:Brownian-affine-separation} on all rational subintervals of $[0,T]$. The next lemma is the exact zero propagation step that closes the saturation recursion.
\begin{lemma}\label{lem:weak-exact-zero-propagation}
Fix $0<s<1/2$.  On $\Omega_{\mathrm{sep},T}$, for every energy solution and every solution $(p,q)$ of \eqref{gap:eq:rough-adjoint}, the following implication holds for every $\phi\in C^\infty(\mathbb T^3)$:
\begin{align*}
 \ip{q_t}\phi=0\quad(0\le t\le T)
\end{align*}
implies
\begin{align}\label{gap:eq:weak-first-zero}
 \ip{p_t}{\Lambda\phi}-\ip{q_t}{a\phi}=0,
 \qquad0\le t\le T,
\end{align}
and
\begin{align}\label{gap:eq:rough-cascade-conclusion}
 \ip{q_t}{\phi b_jb_k}=0,
 \qquad0\le t\le T,\quad1\le j,k\le m.
\end{align}
\end{lemma}

\begin{proof}
Choose $\alpha$ with $\frac12<\alpha<\frac12+s$. 
Since $\langle q_t,\phi\rangle=0$, the second equation in \eqref{gap:eq:rough-adjoint} gives \eqref{gap:eq:weak-first-zero}.  Differentiating once more and using $(\Lambda-\Lambda^{-1})\Lambda=-\Delta$ gives
\begin{align*}
 0=\ip{q_t}{-\Delta\phi-a^2\phi+3u_t^2\phi}+\ip{p_t}{\Lambda(a\phi)}.
\end{align*}
Put $\Psi_t=-\Delta\phi-a^2\phi+3u_t^2\phi$.  A further differentiation, followed by $v=\bar v+\sum_jb_jW^j$, yields
\begin{align*}
 0={}&-\ip{p_t}{\Lambda\Psi_t}+\ip{q_t}{a\Psi_t}-\ip{q_t}{\Delta(a\phi)}+3\ip{q_t}{u_t^2a\phi}+6\ip{q_t}{u_t\phi\bar v_t}\\
 &+6\sum_{j=1}^m\ip{q_t}{u_t\phi b_j}W^j(t).
\end{align*}
The first line belongs to $C^\alpha([0,T])$.  The only borderline increments are those containing $u^2$ and $\bar v$.  By \Cref{lem:weak-adjoint-time-regularity},
\begin{align*}
 \abs{\ip{\Lambda(p_t-p_r)}{u_t^2\phi}}\le Ch^\alpha,\qquad
 \norm{(u_t^2-u_r^2)\phi}_{H^{-s}}\le Ch^{1/2+s},
\end{align*}
and
\begin{align*}
 \abs{\ip{q_t-q_r}{u_t\phi\bar v_t}}+\abs{\ip{q_r}{(u_t-u_r)\phi\bar v_t}}+\abs{\ip{q_ru_r\phi}{\bar v_t-\bar v_r}}\le Ch^\alpha.
\end{align*}
The remaining terms satisfy the same bound directly.  Moreover, each $t\mapsto\langle q_t,u_t\phi b_j\rangle$ is absolutely continuous.  The first application of \Cref{lem:Brownian-affine-separation} therefore gives
\begin{align}\label{gap:eq:weak-intermediate-zero}
 \ip{q_t}{u_t\phi b_j}=0,
 \qquad0\le t\le T,\quad1\le j\le m.
\end{align}
Differentiating \eqref{gap:eq:weak-intermediate-zero} gives
\begin{align*}
 0=\ip{-\Lambda p_t+aq_t}{u_t\phi b_j}+\ip{q_t}{\bar v_t\phi b_j}+\sum_{k=1}^m\ip{q_t}{\phi b_jb_k}W^k(t).
\end{align*}
The first two terms are Lipschitz in time.  For example, with $\chi=\phi b_j$,
\begin{align*}
 \abs{\ip{-\Lambda(p_t-p_r)}{u_t\chi}}+\abs{\ip{-\Lambda p_r}{(u_t-u_r)\chi}}\le Ch,
\end{align*}
and the remaining increments follow directly from \eqref{gap:eq:energy-time-reg}--\eqref{gap:eq:adjoint-rough-time}.  The coefficients $t\mapsto\langle q_t,\phi b_jb_k\rangle$ are absolutely continuous.  A second application of \Cref{lem:Brownian-affine-separation} gives \eqref{gap:eq:rough-cascade-conclusion}; the endpoints follow by continuity.
\end{proof}

\subsubsection{Dense endpoint range}

\begin{theorem}[Dense range of the Malliavin derivative]\label{thm:rough-weak-Hormander}
Assume \eqref{eq:saturation}.  For every $T>0$, every $\mathsf x\in\Hh$, every $0<s<1/2$, and every $\omega\in\Omega_{\mathrm{sep},T}$,
\begin{align}\label{gap:eq:rough-dense-range}
 \overline{\Ran\cA_T(\mathsf x,\omega)}^{\,\Es{s}}=\Es{s}.
\end{align}
In particular, the same exceptional Wiener set works for all energy initial states.
\end{theorem}

\begin{proof}
By the phase identification \eqref{eq:phase-identification}, it is enough to prove that $\mathfrak I\Ran\cA_T$ is dense in $H^{-1-s}\times H^{-1-s}$.  Let $\zeta=(p_T,q_T)\in H^{1+s}\times H^{1+s}$ annihilate this transformed range and let $(p,q)$ solve \eqref{gap:eq:rough-adjoint}.  Adjoint duality with \eqref{eq:Malliavin-endpoint} gives
\begin{align*}
 \int_0^T\ip{B\partial_th(t)}{q_t}\,\dd t=0,\qquad h\in H_T.
\end{align*}
Since $\partial_th$ ranges over $L^2(0,T;\R^m)$, $\langle q_t,b_j\rangle=0$ for almost every $t$ and every $j$, hence for every $t$ by continuity.  Applying \Cref{lem:weak-exact-zero-propagation} inductively along \eqref{eq:saturation-closure} yields
\begin{align*}
 \ip{q_t}\phi=0,\qquad \ip{p_t}{\Lambda\phi}-\ip{q_t}{a\phi}=0,
 \qquad \phi\in\cS_\infty,\quad0\le t\le T.
\end{align*}
At $t=T$, density of $\cS_\infty$ in $H^1$ gives $q_T=0$.  The second identity then gives $\langle p_T,\Lambda\phi\rangle=0$ for $\phi\in\cS_\infty$; since $\Lambda:H^1\to L^2$ is an isomorphism, $\Lambda\cS_\infty$ is dense in $L^2$, and hence $p_T=0$.  Thus the annihilator is trivial, proving \eqref{gap:eq:rough-dense-range}.
\end{proof}

\subsection{Approximate controllability and uniform irreducibility}
We first prove fixed time full support by approximate controllability through the scaling and uniform saturation
viewpoint of 
\cite{GlattHoltzHerzogMattingly2018}, rooted in the Agrachev--Sarychev method \cite{AgrachevSarychev2005,AgrachevSarychev2006}.  

Consider the controlled equation 
\begin{align*}
 \partial_t u=v,\qquad
 \partial_t v=\Delta u-av-u^3+B\partial_t h,\qquad h\in H_T,
\end{align*}
whose endpoint map is denoted by $\Phi_t^{h}$. For a smooth spatial profile $\phi$, denote the constant phase space fields
\begin{align*}
 G_\phi=(0,\phi),\qquad J_\phi=(\phi,-a\phi),
\end{align*}
and, for $c\in\R$, write
\begin{align}\label{eq:translation-gj}
 \mathsf G_\phi^c(\mathsf x)=\mathsf x+cG_\phi,\qquad
 \mathsf J_\phi^c(\mathsf x)=\mathsf x+cJ_\phi, \quad \mathsf x\in\mathcal H
\end{align}
for the corresponding phase space translations. The controls first generate the velocity translations $\mathsf G_\phi^c$,
while short-time conjugation produces $\mathsf J_\phi^c$ and the cubic flow
closes the profile class under the saturation recursion.  Since
\begin{align*}
 \mathsf G_\psi^1\circ\mathsf J_\phi^1(u,v)
 =(u+\phi,v+\psi-a\phi),
\end{align*}
the saturation condition eventually yields arbitrary phase space
translations.

Recall the operators $\mathbb A, \mathbb F, \mathbb B$ from \eqref{eq:AFB} and put
\begin{align*}
 M_0=\sup_{0\le r\le1}\norm{S_a(r)}_{\cL(\Hh)}.
\end{align*}
The following short-time estimates will be used repeatedly in the
conjugation limits underlying the saturation argument.

\begin{lemma}
Let $0<r\le1$. If $f\in L^\infty(0,r;L^2)$, then
\begin{align}\label{eq:general-short-convolution}
 \norm{\int_0^rS_a(r-s)G_{f(s)}\,\dd s}_{\Hh}
 \le M_0r\norm f_{L^\infty(0,r;L^2)}.
\end{align}
For every smooth $\phi$ there are $\epsilon_\phi(r)\downarrow0$ and
$C_\phi<\infty$ such that,
\begin{align}
 \norm{S_a(r)G_\phi-G_\phi-rJ_\phi}_{\Hh}
 &\le r\epsilon_\phi(r),\label{eq:G-first-order}\\
 \norm{\Pi_1S_a(r)G_\phi-r\phi}_{H^1}
 &\le r\epsilon_\phi(r),\label{eq:G-displacement-Ot}\\
 \norm{S_a(r)J_\phi-J_\phi}_{\Hh}
 +\norm{\Pi_1S_a(r)J_\phi-\phi}_{H^1}
 &\le C_\phi r,\label{eq:J-short-time}\\
 \norm{\Pi_1\int_0^rS_a(r-s)G_\phi\,\dd s}_{H^1}
 &\le C_\phi r^2.\label{eq:smooth-source-displacement}
\end{align}
Moreover, for every compact $K\subset\Hh$,
\begin{align}\label{eq:compact-semigroup-modulus}
 \omega_K(r):=\sup_{\mathsf x\in K}\sup_{0\le s\le r}
 \norm{(S_a(s)-\Id)\mathsf x}_{\Hh}\longrightarrow0.
\end{align}
\end{lemma}

\begin{proof}
Estimate \eqref{eq:general-short-convolution} follows directly from the
boundedness of $S_a$ on $[0,1]$ by \Cref{lem:linear-semigroup}.  Since $\phi$  is smooth, 
$G_\phi,J_\phi\in\Dom(\mathbb A)$ and
\begin{align*}
 \mathbb A G_\phi=J_\phi.
\end{align*}
Hence
\begin{align*}
 S_a(r)G_\phi-G_\phi-rJ_\phi
 =\int_0^r(S_a(s)-\Id)J_\phi\,\dd s,
\end{align*}
which gives \eqref{eq:G-first-order} by strong continuity.  Since
$\Pi_1G_\phi=0$ and $\Pi_1J_\phi=\phi$, estimate \eqref{eq:G-displacement-Ot} follows from
the same identity.  Similarly,
\begin{align*}
 S_a(r)J_\phi-J_\phi
 =\int_0^rS_a(s)\mathbb A J_\phi\,\dd s,
\end{align*}
and therefore
\begin{align*}
 \norm{S_a(r)J_\phi-J_\phi}_{\Hh}\le C_\phi r.
\end{align*}
Since $\Pi_1J_\phi=\phi$, this also gives \eqref{eq:J-short-time}.
Finally, \eqref{eq:G-displacement-Ot} yields
\begin{align*}
 \norm{\Pi_1\int_0^rS_a(r-s)G_\phi\,\dd s}_{H^1}
 \le C_\phi\int_0^r(r-s)\,\dd s
 \le C_\phi r^2,
\end{align*}
which proves \eqref{eq:smooth-source-displacement}.  The compact uniform
convergence in \eqref{eq:compact-semigroup-modulus} follows from strong
continuity of $S_a$ by a finite-net argument.
\end{proof}

The next lemma shows that the velocity translations generated by the forced profiles arise as small time limits of controlled endpoint maps.
\begin{lemma}\label{lem:forced-velocity-translations}
For every $\phi\in\cS_0$, $c\in\R$, and compact $K\subset\Hh$, there are
Cameron--Martin controls $h_\delta$ on $[0,\delta]$ such that
\begin{align}\label{eq:forced-velocity-scaling}
 \sup_{\mathsf x\in K}
 \norm{\Phi_\delta^{h_\delta}(\mathsf x)-\mathsf G_\phi^c(\mathsf x)}_{\Hh}
 \longrightarrow0
 \qquad\text{as }\delta\downarrow0.
\end{align}
\end{lemma}

\begin{proof}
It is enough to take $\phi=b_j$.  Choose
$\vartheta\in C^\infty([0,1])$ with $\vartheta(0)=0$,
$\vartheta(1)=1$, and set
\begin{align*}
 h_\delta^j(t)=c\vartheta(t/\delta),\qquad 0\le t\le\delta,
\end{align*}
with all other components equal to zero.  For $\mathsf x\in K$, write $X=(u,v)$ for the corresponding controlled solution, suppressing its dependence on $\mathsf x$ and $h_\delta$.  Then
\begin{align*}
 \Phi_\delta^{h_\delta}(\mathsf x)
 =S_a(\delta)\mathsf x+
 \int_0^\delta S_a(\delta-s)\mathbb F(X_s)\,\dd s
 +c\int_0^1S_a(\delta(1-r))G_{b_j}\vartheta'(r)\,\dd r.
\end{align*}
By strong continuity of $S_a$ and $\int_0^1\vartheta'=1$,
\begin{align*}
 c\int_0^1S_a(\delta(1-r))G_{b_j}\vartheta'(r)\,\dd r
 \longrightarrow cG_{b_j}
 \quad\text{in }\Hh.
\end{align*}
Moreover, the control convolution is uniformly bounded on
$[0,\delta]$:
\begin{align*}
 \sup_{s\le\delta}
 \norm{\frac c\delta\int_0^s
 S_a(s-r)G_{b_j}\vartheta'(r/\delta)\,\dd r}_{\Hh}
 \le C_{c,b_j,\vartheta}.
\end{align*}
Note that 
\begin{align*} 
\norm{\mathbb F(u,v)}_{\Hh} \le C\bigl(1+\norm{(u,v)}_{\Hh}^3\bigr). 
\end{align*}
The mild equation, the uniform bound on the control convolution, and
compactness of $K$ give
\begin{align*}
Y_\delta(t)
\le A+C\int_0^t\bigl(1+Y_\delta(r)^3\bigr)\,\dd r,
\qquad
Y_\delta(t):=
\sup_{\mathsf x\in K}\sup_{0\le s\le t}
\norm{X_s}_{\Hh},
\end{align*}
for some $A$ independent of $\delta$.  A standard bootstrap continuity
argument therefore yields, for all sufficiently small $\delta$,
\begin{align*}
\sup_{\mathsf x\in K}\sup_{s\le\delta}
\norm{X_s}_{\Hh}\le 2A.
\end{align*}
Hence
\begin{align*}
 \sup_{\mathsf x\in K}
 \norm{\int_0^\delta S_a(\delta-s)
 \mathbb F(X_s)\,\dd s}_{\Hh}
 \le C_{K,c,b_j}\delta.
\end{align*}
Finally, $S_a(\delta)\mathsf x\to \mathsf x$ uniformly on $K$, which proves
\eqref{eq:forced-velocity-scaling}.
\end{proof}

We next show that the displacement translations are generated from the velocity translations through a short-time scaling limit.
\begin{lemma}\label{lem:displacement-descendants}
For every smooth $\phi$, $t>0$, and compact $K\subset\Hh$, if
$\delta_R=t/R$, then
\begin{align*}
 \mathsf G_\phi^{-R}\circ\Phi_{\delta_R}^{0}\circ\mathsf G_\phi^{R}
 \longrightarrow\mathsf J_\phi^t
\end{align*}
uniformly on $K$ as $R\to\infty$.
\end{lemma}
\begin{proof}
Set
\begin{align*}
 \Psi_R=\mathsf G_\phi^{-R}\circ\Phi_{\delta_R}^{0}
 \circ\mathsf G_\phi^{R},
 \qquad
 \delta_R=\frac{t}{R}.
\end{align*}
The mild formula gives
\begin{align*}
 \Psi_R(\mathsf x)
 =S_a(\delta_R)\mathsf x
 +R\bigl(S_a(\delta_R)G_\phi-G_\phi\bigr)
 +\int_0^{\delta_R}S_a(\delta_R-s)
 \mathbb F(X_R(s))\,\dd s.
\end{align*}
For $\mathsf x\in K$, let $X_R=(u_R,v_R)$ be the solution issued from $\mathsf x+RG_\phi$, with the dependence on $\mathsf x$ suppressed, and set
\begin{align*}
 D_R(r)=\sup_{\mathsf x\in K}\sup_{0\le s\le r}\norm{u_R(s)}_{H^1}.
\end{align*}
By \eqref{eq:G-displacement-Ot}, for $s\le\delta_R=t/R$,
\begin{align*}
 R\norm{\Pi_1S_a(s)G_\phi}_{H^1}
 \le C_\phi Rs
 \le C_\phi t.
\end{align*}
Hence the mild equation and
\begin{align*}
 \norm{u-u^3}_{L^2}
 \le C\bigl(1+\norm u_{H^1}^3\bigr)
\end{align*}
give
\begin{align*}
 D_R(r)\le A+C\int_0^r\bigl(1+D_R(s)^3\bigr)\,\dd s,
 \qquad 0\le r\le\delta_R,
\end{align*}
with $A$ independent of $R$.  A standard bootstrap continuity argument
therefore yields, for all sufficiently large $R$,
\begin{align*}
 \sup_{\mathsf x\in K}\sup_{s\le\delta_R}
 \norm{u_R(s)}_{H^1}\le2A.
\end{align*}
Consequently,
\begin{align}\label{eq:G-conjugation-nonlinear-error}
 \sup_{\mathsf x\in K}
 \norm{\int_0^{\delta_R}S_a(\delta_R-s)
 \mathbb F(X_R(s))\,\dd s}_{\Hh}
 \le C_{K,t,\phi}\delta_R.
\end{align}

We now subtract the limiting translation.  Since
$R\delta_R=t$, the preceding mild identity gives
\begin{align*}
\begin{aligned}
 \Psi_R(\mathsf x)-\mathsf x-tJ_\phi
 ={}&S_a(\delta_R)\mathsf x-\mathsf x+R\bigl(S_a(\delta_R)G_\phi-G_\phi-\delta_RJ_\phi\bigr)\\
 &+\int_0^{\delta_R}S_a(\delta_R-s)
 \mathbb F(X_R(s))\,\dd s.
\end{aligned}
\end{align*}
Therefore, in view of \eqref{eq:translation-gj}, by \eqref{eq:G-first-order},
\eqref{eq:compact-semigroup-modulus} and
\eqref{eq:G-conjugation-nonlinear-error},
\begin{align*}
 \sup_{\mathsf x\in K}
 \norm{\Psi_R(\mathsf x)-\mathsf J_\phi^t(\mathsf x)}_{\Hh}
 \le
 \omega_K(\delta_R)
 +t\epsilon_\phi(\delta_R)
 +C_{K,t,\phi}R^{-1}
 \longrightarrow0.
\end{align*}
\end{proof}

The next scaling limit uses the displacement translations generated in the previous lemma to create new velocity directions through the cubic nonlinearity.
\begin{lemma}\label{lem:cubic-descendants}
For every smooth $\zeta$, $t>0$, and compact $K\subset\Hh$, if
$\delta_R=t/R^3$, then
\begin{align*}
 \mathsf J_\zeta^{-R}\circ\Phi_{\delta_R}^{0}\circ\mathsf J_\zeta^{R}
 \longrightarrow\mathsf G_{\zeta^3}^{-t}
\end{align*}
uniformly on $K$ as $R\to\infty$.
\end{lemma}

\begin{proof}
Set
\begin{align*}
 \Theta_R=\mathsf J_\zeta^{-R}\circ\Phi_{\delta_R}^{0}
 \circ\mathsf J_\zeta^{R},
 \qquad
 \delta_R=\frac{t}{R^3}.
\end{align*}
For $\mathsf x\in K$, let $X_R=(u_R,v_R)$ be the solution issued from $\mathsf J_\zeta^R\mathsf x$, suppressing the dependence on $\mathsf x$, and set
\begin{align*}
 q_R(s)=u_R(s)-R\zeta.
\end{align*}
The first-coordinate mild equation is
\begin{align*}
 q_R(s)
 ={}\Pi_1S_a(s)\mathsf x
 +R\bigl(\Pi_1S_a(s)J_\zeta-\zeta\bigr)+\Pi_1\int_0^sS_a(s-r)
 G_{u_R(r)-(u_R(r))^3}\,\dd r.
\end{align*}

Since $u_R=R\zeta+q_R$, we have 
\begin{align}\label{eq:cubic-source-expansion}
 u_R-(u_R)^3
 =-R^3\zeta^3+\rho_R,
\end{align}
where
\begin{align*}
 \rho_R
 =R\zeta+q_R
 -3R^2\zeta^2q_R
 -3R\zeta(q_R)^2-(q_R)^3.
\end{align*}
Fix $A>1+M_0\sup_{\mathsf x\in K}\norm{\mathsf x}_{\Hh}$.  As long as
\begin{align*}
 \sup_{\mathsf x\in K}\norm{q_R(s)}_{H^1}\le2A,
\end{align*}
the embedding $H^1\hookrightarrow L^6$ gives
\begin{align*}
 \norm{\rho_R(s)}_{L^2}
 \le C_{\zeta,A}(R^2+R+1)
 \le C_{\zeta,A}R^2.
\end{align*}
Substituting \eqref{eq:cubic-source-expansion} into the first-coordinate
mild equation and using
\eqref{eq:J-short-time},
\eqref{eq:smooth-source-displacement}, and
\eqref{eq:general-short-convolution}, we obtain, for
$s\le\delta_R=t/R^3$,
\begin{align*}
 \sup_{\mathsf x\in K}\norm{q_R(s)}_{H^1}
 &\le A-1
 +C_\zeta Rs
 +C_{\zeta}R^3s^2
 +C_{\zeta,A}R^2s\\
 &\le A-1
 +C_{\zeta,A,t}
 \bigl(R^{-2}+R^{-3}+R^{-1}\bigr).
\end{align*}
For all sufficiently large $R$ the right-hand side is strictly smaller
than $A$.  A bootstrap continuity argument therefore gives
\begin{align*}
 \sup_{\mathsf x\in K}\sup_{s\le\delta_R}
 \norm{q_R(s)}_{H^1}\le2A.
\end{align*}
Consequently, \eqref{eq:cubic-source-expansion} and the bound
\begin{align}\label{eq:cubic-source-remainder}
 \sup_{\mathsf x\in K,s\le\delta_R}
 \norm{\rho_R(s)}_{L^2}
 \le C_{K,t,\zeta}R^2
\end{align}
hold on the whole interval $[0,\delta_R]$.

We now return to the full mild equation.  Since
\begin{align*}
 \Theta_R(\mathsf x)
 =S_a(\delta_R)\mathsf x
 +R\bigl(S_a(\delta_R)J_\zeta-J_\zeta\bigr)
 +\int_0^{\delta_R}S_a(\delta_R-s)
 G_{u_R(s)-(u_R(s))^3}\,\dd s,
\end{align*}
equation \eqref{eq:cubic-source-expansion} gives
\begin{align*}
\begin{aligned}
 \Theta_R(\mathsf x)-\mathsf x+tG_{\zeta^3}
 ={}&S_a(\delta_R)\mathsf x-\mathsf x
 +R\bigl(S_a(\delta_R)J_\zeta-J_\zeta\bigr)\\
 &+tG_{\zeta^3}
 -R^3\int_0^{\delta_R}
 S_a(\delta_R-s)G_{\zeta^3}\,\dd s\\
 &+\int_0^{\delta_R}
 S_a(\delta_R-s)G_{\rho_R(s)}\,\dd s.
\end{aligned}
\end{align*}
Therefore, by \eqref{eq:J-short-time},
\eqref{eq:general-short-convolution},
\eqref{eq:cubic-source-remainder}, and
\eqref{eq:compact-semigroup-modulus},
\begin{align*}
 \sup_{\mathsf x\in K}
 \norm{\Theta_R(\mathsf x)-\mathsf G_{\zeta^3}^{-t}(\mathsf x)}_{\Hh}
 &\le\omega_K(\delta_R)
 +C_\zeta R\delta_R\\
 &\quad
 +t\sup_{0\le r\le\delta_R}
 \norm{(S_a(r)-\Id)G_{\zeta^3}}_{\Hh}
 +C_{K,t,\zeta}R^2\delta_R
 \longrightarrow0.
\end{align*}
\end{proof}

The preceding scaling limits close the same cubic saturation recursion as
in \Cref{def:saturation}.  The next proposition upgrades the resulting
profile directions to arbitrary phase space translations.  This is the
deterministic controllability mechanism used below to prove fixed time full
support.

\begin{proposition}
\label{prop:small-time-translation-saturation}
Assume \eqref{eq:saturation}. For every $(f,g)\in\Hh$, compact $K\subset\Hh$, $\varepsilon>0$, and
$\tau>0$, there are $0<t<\tau$ and $h\in H_t$ such that
\begin{align}\label{eq:small-time-arbitrary-translation}
 \sup_{\mathsf x\in K}
 \norm{\Phi_t^h(\mathsf x)-(\mathsf x+(f,g))}_{\Hh}<\varepsilon.
\end{align}
\end{proposition}

\begin{proof}
For $n\ge0$, put
\begin{align*}
 V_n=\spanop_{\R}\cS_n.
\end{align*}
We claim that, for every $n\ge0$, $\phi\in V_n$, and $c\in\R$, both
$\mathsf G_\phi^c$ and $\mathsf J_\phi^c$ belong to the arbitrarily
small time compact uniform closure of the controlled endpoint maps.

For $n=0$, \Cref{lem:forced-velocity-translations} gives
$\mathsf G_{b_j}^c$ for every $j$ and $c$.  Since constant translations
commute and $G_\phi$ depends linearly on $\phi$, finite composition gives
$\mathsf G_\phi^c$ for every $\phi\in V_0$.  Since $\mathsf G_\phi^c$ belongs to the arbitrarily small time
compact uniform closure for every $c\in\R$, the flow property allows the
approximating controlled trajectories to be concatenated.  By the standard compact uniform transitivity of the saturation procedure
\cite[Section~3]{GlattHoltzHerzogMattingly2018} and
\Cref{lem:displacement-descendants}, this yields
$\mathsf J_\phi^c$ for every $\phi\in V_0$ and $c\in\R$.

Assume the claim holds on $V_n$.  Fix $\phi\in\cS_n$ and $1\le j,k\le m$,
and set
\begin{align*}
 \zeta_1&=\phi+b_j+b_k,&
 \zeta_2&=\phi-b_j-b_k,\\
 \zeta_3&=-\phi+b_j-b_k,&
 \zeta_4&=-\phi-b_j+b_k.
\end{align*}
Since $\zeta_\ell\in V_n$, the induction hypothesis gives all translations
$\mathsf J_{\zeta_\ell}^c$.  By \Cref{lem:cubic-descendants}, it therefore
gives all translations $\mathsf G_{\zeta_\ell^3}^c$.  By the cubic polarization identity, 
\begin{align*}
 24\phi b_jb_k
 =\zeta_1^3+\zeta_2^3+\zeta_3^3+\zeta_4^3,
\end{align*}
and composing the corresponding constant velocity translations, we obtain
$\mathsf G_{\phi b_jb_k}^c$ for every $c\in\R$.  Hence
\Cref{lem:displacement-descendants} also gives
$\mathsf J_{\phi b_jb_k}^c$.

Since
\begin{align*}
 V_{n+1}
 =\spanop_{\R}\left(
 V_n\cup\{\phi b_jb_k:\phi\in\cS_n,\ 1\le j,k\le m\}
 \right),
\end{align*}
linearity and finite composition prove the claim on $V_{n+1}$.  Thus the
claim holds for every
\begin{align*}
 \phi\in\bigcup_{n\ge0}V_n=\cS_\infty.
\end{align*}

Now let $(f,g)\in\Hh$.  By \eqref{eq:saturation}, choose
$\phi_n,\psi_n\in\cS_\infty$ such that
\begin{align*}
 \phi_n\to f\quad\text{in }H^1,
 \qquad
 \psi_n\to g+af\quad\text{in }L^2.
\end{align*}
Since
\begin{align*}
 \mathsf G_{\psi_n}^1\circ\mathsf J_{\phi_n}^1(u,v)
 =(u+\phi_n,v+\psi_n-a\phi_n),
\end{align*}
these translations converge uniformly on $\Hh$ to
\begin{align*}
 (u,v)\longmapsto(u+f,v+g).
\end{align*}
Choose $n$ so large that the translation error is smaller than
$\varepsilon/2$, and then approximate
$\mathsf G_{\psi_n}^1\circ\mathsf J_{\phi_n}^1$ on $K$ within
$\varepsilon/2$ by a controlled endpoint of duration less than $\tau$.
This proves \eqref{eq:small-time-arbitrary-translation}.
\end{proof}

A direct consequence of the approximate controllability is the following topological irreducibility. 
\begin{theorem}\label{thm:full-support}
Under \eqref{eq:saturation}, for every $\mathsf x\in\Hh$, every $T>0$, and every
nonempty open $O\subset\Hh$,
\begin{align*}
 P_T(\mathsf x,O)>0.
\end{align*}
\end{theorem}

\begin{proof}
Choose $y\in O$ and $\varepsilon>0$ such that $B_{\Hh}(y,3\varepsilon)\subset O$ and set $\xi=y-\Phi_T^0(\mathsf x)$. 
By continuity of the zero-control trajectory, for all sufficiently small
$\delta<T/2$,
\begin{align*}
 K_\delta=\{\Phi_s^0(\mathsf x):T-\delta\le s\le T\}
\end{align*}
is compact and satisfies
\begin{align}\label{eq:terminal-arc-translation}
 \sup_{z\in K_\delta}\norm{z+\xi-y}_{\Hh}<\frac{\varepsilon}{2}.
\end{align}

Apply \Cref{prop:small-time-translation-saturation} to the translation
$z\mapsto z+\xi$, the compact set $K_\delta$, and the time 
$\delta$.  There are $0<t<\delta$ and $h\in H_t$ such that
\begin{align*}
 \sup_{z\in K_\delta}
 \norm{\Phi_t^h(z)-(z+\xi)}_{\Hh}<\frac{\varepsilon}{2}.
\end{align*}
Since $\Phi_{T-t}^0(\mathsf x)\in K_\delta$, concatenating the zero control on
$[0,T-t]$ with $h$ on $[T-t,T]$ gives a Cameron--Martin driver
$\widehat h\in H_T$ for which, by \eqref{eq:terminal-arc-translation},
\begin{align*}
 \norm{\Phi_T^{\widehat h}(\mathsf x)-y}_{\Hh}<\varepsilon.
\end{align*}

By the shifted formulation \eqref{eq:pathwise-shifted-wave}, the standard
wave energy estimate and $H^1(\mathbb T^3)\hookrightarrow L^6(\mathbb T^3)$ imply that the solution map is
continuous at $\widehat h$ in the uniform driver topology.  Hence there is
$r>0$ such that
\begin{align*}
 \norm{\omega-\widehat h}_{C([0,T];\R^m)}<r
 \quad\Longrightarrow\quad
 \norm{\Phi_T(\mathsf x,\omega)-y}_{\Hh}<2\varepsilon.
\end{align*}
Wiener measure has full support on $C_0([0,T];\R^m)$
\cite{Bogachev1998}, and therefore
\begin{align*}
 \mathbb P\left(
 \norm{W-\widehat h}_{C([0,T];\R^m)}<r
 \right)>0.
\end{align*}
Since $B_{\Hh}(y,2\varepsilon)\subset O$, this yields
\begin{align*}
 P_T(\mathsf x,O)>0.
\end{align*}
\end{proof}

The next proposition combines topological irreducibility with the
decaying--regular decomposition to obtain uniform accessibility on
bounded energy sets, thereby verifying \Cref{ass:QS-accessibility} for
the wave equation.

\begin{proposition}[Uniform irreducibility]
\label{gap:prop:uniform-accessibility}
For every $R<\infty$ and $\varepsilon>0$, there is
$T_{R,\varepsilon}<\infty$ such that
\begin{align}\label{gap:eq:uniform-small-ball-access}
 \inf_{\Vv(\mathsf x)\le R}
 P_T\bigl(\mathsf x,B_{\Hh}(0,\varepsilon)\bigr)>0
\end{align}
for every $T\ge T_{R,\varepsilon}$.
\end{proposition}

\begin{proof}
For $\mathsf x\in\mathbb V_R$, write
\begin{align*}
 X_t=S_a(t)\mathsf x+Y_t,
\end{align*}
suppressing the dependence of $X$ and $Y$ on $\mathsf x$.
By \eqref{eq:strong-remainder-energy-uniform} and the compact Sobolev
embedding, there is a compact set $K\subset\Hh$ such that
\begin{align}\label{gap:eq:access-compact-remainder}
 \inf_{\Vv(\mathsf x)\le R}\inf_{t\ge0}
 \Prob \{Y_t\in K\}\ge\frac12.
\end{align}
Fix any $\tau>0$ and put $B=B_{\Hh}(0,\varepsilon)$.  By
\Cref{thm:full-support}, $P_\tau(z,B)>0$ for every $z\in\Hh.$
Since $P_\tau$ is Feller and $B$ is open, the map
$z\mapsto P_\tau(z,B)$ is lower semicontinuous.  Compactness of $K$
therefore gives
\begin{align*}
 q:=\inf_{z\in K}P_\tau(z,B)>0.
\end{align*}
Hence there is an open neighborhood $U\supset K$ such that
\begin{align}\label{gap:eq:access-terminal-minorization}
 \inf_{z\in U}P_\tau(z,B)\ge\frac q2.
\end{align}
Let $\delta=\operatorname{dist}_{\Hh}(K,U^c)>0$.  By
\eqref{eq:semigroup-decay-scale} and boundedness of the energy sublevel,
there is $t_0<\infty$ such that
\begin{align*}
 \sup_{\Vv(\mathsf x)\le R}
 \norm{S_a(t)\mathsf x}_{\Hh}<\delta,
 \qquad t\ge t_0.
\end{align*}
Thus, on $\{Y_t\in K\}$, one has
$X_t\in U$, and \eqref{gap:eq:access-compact-remainder}
implies
\begin{align*}
 \inf_{\Vv(\mathsf x)\le R}P_t(\mathsf x,U)\ge\frac12,
 \qquad t\ge t_0.
\end{align*}
Using the Markov property and
\eqref{gap:eq:access-terminal-minorization},
\begin{align*}
 P_{t+\tau}(\mathsf x,B)
 =\int_{\Hh}P_\tau(z,B)\,P_t(\mathsf x,\dd z)
 \ge\frac q2P_t(\mathsf x,U)\ge\frac q4
\end{align*}
uniformly for $\Vv(\mathsf x)\le R$ and $t\ge t_0$.  Taking
$T_{R,\varepsilon}=t_0+\tau$ proves
\eqref{gap:eq:uniform-small-ball-access}.
\end{proof}

\subsection{Geometric characterization of saturation}\label{section:geometric-saturation}
We now give two complementary descriptions of the density condition
\eqref{eq:saturation}.  For general smooth forcing profiles, the saturation
closure admits a geometric characterization through the rank one profile
map.  For Fourier-mode forcing, it admits an exact lattice
characterization.

Let
\begin{align*}
 b=(b_1,\ldots,b_m):\mathbb T^3\longrightarrow\mathbb R^m,
 \qquad
 \Gamma_b(x)=b(x)b(x)^{\mathsf T}\in\operatorname{Sym}_m,
\end{align*}
and write
\begin{align*}
 \mathbb R[z_1,\ldots,z_m]_{\mathrm{odd}}
 =\{P\in\mathbb R[z_1,\ldots,z_m]:P(-z)=-P(z)\}.
\end{align*}

\begin{theorem}
\label{thm:odd-algebra-geometry}
The following statements hold.
\begin{enumerate}[label=\textup{(\roman*)}]
\item The saturation closure has the exact algebraic description
\begin{align}\label{eq:exact-odd-algebra}
 \cS_\infty
 =\{P\circ b:P\in\mathbb R[z_1,\ldots,z_m]_{\mathrm{odd}}\}.
\end{align}

\item One has
\begin{align}\label{eq:C0-geometric-criterion}
 \overline{\cS_\infty}^{\,C^0}=C(\mathbb T^3)
\end{align}
if and only if $b(x)\ne0$ for every $x$ and $\Gamma_b$ is injective.

\item For every integer $k\ge1$, one has
\begin{align}\label{eq:C1-geometric-criterion}
 \overline{\cS_\infty}^{\,C^k}=C^k(\mathbb T^3)
\end{align}
if and only if $\Gamma_b$ is a smooth embedding.  In particular,
\eqref{eq:saturation} holds whenever $\Gamma_b$ is a smooth embedding.
\end{enumerate}
\end{theorem}

\begin{proof}
By the recursion \eqref{eq:saturation-closure} and the definition
\eqref{eq:saturation-space}, $\operatorname{span}_{\mathbb R}\cS_n$ is the real linear span of the monomials in
$b_1,\ldots,b_m$ of odd total degree at most $2n+1$.  Taking the union over $n$ gives
\eqref{eq:exact-odd-algebra}.
Let $\cA_b$ be the unital algebra generated by the entries of $\Gamma_b$,
that is,
\begin{align*}
 \cA_b
 =\mathbb R\bigl[(\Gamma_b)_{jk}:1\le j,k\le m\bigr].
\end{align*}
Since $(\Gamma_b)_{jk}=b_jb_k$ and every odd-degree monomial is one
linear factor times a product of quadratic monomials, the algebraic
description \eqref{eq:exact-odd-algebra} may equivalently be written as
\begin{align}\label{eq:odd-module-decomposition}
 \cS_\infty=\sum_{i=1}^m b_i\cA_b.
\end{align}

Suppose first that $\Gamma_b$ is injective.  Then $\cA_b$ separates
points, and the real Stone--Weierstrass theorem gives
\begin{align}\label{eq:even-algebra-C0-density}
 \overline{\cA_b}^{\,C^0}=C(\mathbb T^3).
\end{align}
If, in addition, $b$ is nowhere zero, then for
$f\in C(\mathbb T^3)$ the functions $c_i=fb_i/|b|^2$ for $1\le i\le m$ 
are continuous and satisfy $f=\sum_i b_ic_i$.  Approximating each $c_i$
by elements of $\cA_b$ in \eqref{eq:even-algebra-C0-density} and using
\eqref{eq:odd-module-decomposition} proves the sufficiency in
\eqref{eq:C0-geometric-criterion}. For necessity, if $b(x_0)=0$, every member of $\cS_\infty$ vanishes at $x_0$.
If $\Gamma_b(x)=\Gamma_b(y)$ and neither vector is zero, equality of the
two rank one matrices gives $b(y)=\sigma b(x)$
for some $\sigma\in\{1,-1\}$.  By
\eqref{eq:exact-odd-algebra}, every $F\in\cS_\infty$ then satisfies
\begin{align*}
 F(y)=\sigma F(x).
\end{align*}
If $\sigma=1$, this prevents point separation; if $\sigma=-1$, it
prevents uniform approximation of the constant function one.  This proves
the necessity in \eqref{eq:C0-geometric-criterion}.

Suppose next that $\Gamma_b$ is a smooth embedding.  Its image is therefore a compact embedded submanifold
of $\operatorname{Sym}_m$.  For $f\in C^\infty(\mathbb T^3)$ the
coefficients $c_i=fb_i/|b|^2$ are smooth, and there exist a tubular neighborhood $U$ of
$\Gamma_b(\mathbb T^3)$ and a smooth retraction
$\pi:U\rightarrow\Gamma_b(\mathbb T^3)$ such that 
$c_i\circ\Gamma_b^{-1}\circ\pi$ is a smooth extension of
$c_i\circ\Gamma_b^{-1}$ to $U$.  Choosing $\chi\in C_c^\infty(U)$ with
$\chi=1$ near $\Gamma_b(\mathbb T^3)$ and extending by zero outside $U$,
we obtain a compactly supported smooth extension of
$c_i\circ\Gamma_b^{-1}$ to $\operatorname{Sym}_m$.  On a box containing
the image, there exist multivariate polynomials $B_n$ approximating this
extension together with all derivatives up to any prescribed finite
order.  Hence, after restriction to $\Gamma_b(\mathbb T^3)$ and pullback
by $\Gamma_b$,
\begin{align*}
 B_n\circ\Gamma_b\longrightarrow c_i
 \qquad\text{in }C^k(\mathbb T^3)
\end{align*}
for every finite $k$. Each $B_n\circ\Gamma_b$ is a polynomial in the entries
$(\Gamma_b)_{jk}=b_jb_k$, yielding 
$B_n\circ\Gamma_b\in\cA_b$.  Thus $c_i$ is approximable in every finite
$C^k$ norm by elements of $\cA_b$.  The identity $f=\sum_i b_ic_i$ and
\eqref{eq:odd-module-decomposition} therefore show that every smooth
function is approximable by $\cS_\infty$ in every finite $C^k$ norm, proving 
\eqref{eq:C1-geometric-criterion} for every integer $k\ge1$.

Conversely, suppose that \eqref{eq:C1-geometric-criterion} holds for some
integer $k\ge1$.  Since convergence in $C^k$ implies convergence in
$C^0$, the preceding part shows that $b$ is nowhere zero and $\Gamma_b$
is injective.  If $0\ne\xi\in T_x\mathbb T^3$ and
$\dd\Gamma_b(x)\xi=0$, put $w=\dd b(x)\xi.$
Then
\begin{align*}
 wb(x)^{\mathsf T}+b(x)w^{\mathsf T}=0.
\end{align*}
Multiplication by $b(x)$, followed by the orthogonal decomposition of $w$
parallel and perpendicular to $b(x)$, gives $w=0$.  Consequently,
\begin{align*}
 \dd(P\circ b)(x)\xi=0
\end{align*}
for every polynomial $P$.  By \eqref{eq:exact-odd-algebra}, the same is
true for every element of $\cS_\infty$.  Since $k\ge1$, convergence in
$C^k$ implies convergence in $C^1$, whereas there exists a smooth
function whose derivative at $x$ in the direction $\xi$ is nonzero, contradicting \eqref{eq:C1-geometric-criterion}.  Thus $\Gamma_b$ is
an injective immersion, and compactness of $\mathbb T^3$ makes it a smooth
embedding.

In particular, taking $k$ sufficiently large and using the continuous
embedding $C^k(\mathbb T^3)\hookrightarrow H^1(\mathbb T^3)$ gives
\eqref{eq:saturation}.
\end{proof}

\begin{proposition}
Four profiles are necessary and sufficient for the smooth-embedding
criterion of \Cref{thm:odd-algebra-geometry}.  More precisely, fix $ 0<r<R,\, L>R+r$ and set, with $x=(x_1,x_2,x_3)$,
\begin{align}
 b_1(x)&=L+(R+r\cos x_3)\cos x_1,\notag\\
 b_2(x)&=(R+r\cos x_3)\sin x_1,\notag\\
 b_3(x)&=(R+r\sin x_3)\cos x_2,\label{eq:four-profile-embedding}\\
 b_4(x)&=(R+r\sin x_3)\sin x_2.\notag
\end{align}
Then $\Gamma_b$ is a smooth embedding, and these four profiles satisfy
\eqref{eq:saturation}. 
\end{proposition}

\begin{proof}
If $\Gamma_b$ is an immersion, then $b$ is an immersion as $\ker\dd b(x)\subset\ker\dd\Gamma_b(x)$.  For
$m\le2$ this is excluded by dimension.  For $m=3$, it would make $b$ a
local diffeomorphism, so $b(\mathbb T^3)$ would be both open in
$\mathbb R^3$ and compact, which is impossible.  Thus $m\ge4$.

For \eqref{eq:four-profile-embedding}, one has $b_1>L-(R+r)>0$.  Moreover,
$\sqrt{(b_1-L)^2+b_2^2}=R+r\cos x_3$ and
$\sqrt{b_3^2+b_4^2}=R+r\sin x_3$, so these two quantities recover $x_3$,
while $(b_1-L,b_2)$ and $(b_3,b_4)$ recover $x_1$ and $x_2$.  Hence $b$
is a smooth embedding.  If $\Gamma_b(x)=\Gamma_b(y)$, then
$b(y)=\pm b(x)$, and positivity of the first component excludes the minus
sign.  Since $\Gamma_b$ is also an immersion, it is a smooth embedding.
\end{proof}

We now specialise to Fourier-mode forcing.  Let
$\cK=\{k_1,\ldots,k_N\}\subset\mathbb Z^3\setminus\{0\}$, $N\ge1$,
satisfy $\cK\cap(-\cK)=\varnothing$, and assume that
\begin{align}\label{eq:Fourier-mode-forcing}
 \cS_0
 =
 \{1:\varepsilon=1\}\cup
 \{\cos(k\cdot x),\sin(k\cdot x):k\in\cK\},
 \qquad \varepsilon\in\{0,1\},
\end{align}
where the first set is empty when $\varepsilon=0$.

\begin{theorem}
\label{thm:sharp-Fourier-criterion}
Set
\begin{align*}
 \cA=\{\pm k:k\in\cK\}\cup\{0:\varepsilon=1\},
 \qquad
 \Gamma_{\mathrm{pair}}
 =\spanop_{\mathbb Z}\{k+k':k,k'\in\cA\}.
\end{align*}
Then, for any $k^\circ\in\cA$,
\begin{align}\label{eq:exact-Fourier-coset}
 \cS_\infty^{\mathbb C}
 =
 \spanop_{\mathbb C}\{\mathrm e^{\mathrm i\ell\cdot x}:
 \ell\in k^\circ+\Gamma_{\mathrm{pair}}\}.
\end{align}
Consequently, the following are equivalent:
\begin{enumerate}[label=\textup{(\alph*)}]
\item the forcing is cubically saturating;
\item $\Gamma_{\mathrm{pair}}=\mathbb Z^3$;
\item there are no nonzero $\theta\in\mathbb T^3$ and
$c\in\{1,-1\}$ such that
\begin{align}\label{eq:common-sign-translation}
 \mathrm e^{\mathrm ik\cdot\theta}=c
 \qquad\text{for every }k\in\cK,
\end{align}
with $c=1$ when $\varepsilon=1$.
\end{enumerate}
\end{theorem}

\begin{proof}
By \eqref{eq:exact-odd-algebra}, $\cS_\infty$ is the linear span of all
odd products of elements of $\cS_0$, and hence depends only on the real
linear span of the initial forcing profiles.  After complexification,
\begin{align*}
 \spanop_{\mathbb C}\cS_0
 =
 \spanop_{\mathbb C}
 \{\mathrm e^{\mathrm ik\cdot x}:k\in\cA\},
\end{align*}
so we may work with the exponential generators
$\mathrm e^{\mathrm ik\cdot x}$, $k\in\cA$.  The frequencies produced by
odd products are precisely
\begin{align*}
 \mathscr C
 =
 \bigcup_{n\ge0}
 \{\ell_1+\cdots+\ell_{2n+1}:\ell_j\in\cA\}.
\end{align*}

Fix $k^\circ\in\cA$.  If
$\ell=\ell_1+\cdots+\ell_{2n+1}\in\mathscr C$, then
\begin{align*}
 \ell-k^\circ
 =
 (\ell_1-k^\circ)
 +(\ell_2+\ell_3)+\cdots
 +(\ell_{2n}+\ell_{2n+1})
 \in\Gamma_{\mathrm{pair}},
\end{align*}
because $-k^\circ\in\cA$.  Thus
$\mathscr C\subset k^\circ+\Gamma_{\mathrm{pair}}$.

Conversely, since $\cA=-\cA$, the set of pair sums is symmetric under
sign.  Hence every integer linear combination of pair sums can be written
as a finite sum of pair sums.  Adding $k^\circ$ therefore produces an odd
sum of elements of $\cA$, and so
$k^\circ+\Gamma_{\mathrm{pair}}\subset\mathscr C$.  Thus
\begin{align*}
 \mathscr C=k^\circ+\Gamma_{\mathrm{pair}},
\end{align*}
which proves \eqref{eq:exact-Fourier-coset}.

By the Fourier basis,
\begin{align*}
 \overline{\cS_\infty}^{\,H^s}=H^s(\mathbb T^3)
 \quad\Longleftrightarrow\quad
 k^\circ+\Gamma_{\mathrm{pair}}=\mathbb Z^3
 \quad\Longleftrightarrow\quad
 \Gamma_{\mathrm{pair}}=\mathbb Z^3
\end{align*}
for every finite $s$.  In particular, \textup{(a)} and \textup{(b)} are
equivalent.

Suppose that $\Gamma_{\mathrm{pair}}\ne\mathbb Z^3$.  Then there is a
nontrivial character of $\mathbb Z^3$ which is trivial on
$\Gamma_{\mathrm{pair}}$, hence a nonzero $\theta\in\mathbb T^3$ such
that $\mathrm e^{\mathrm i\gamma\cdot\theta}=1$ for every
$\gamma\in\Gamma_{\mathrm{pair}}$.  Since
$k+k'\in\Gamma_{\mathrm{pair}}$ for all $k,k'\in\cA$,
\begin{align*}
 \mathrm e^{\mathrm ik\cdot\theta}
 \mathrm e^{\mathrm ik'\cdot\theta}=1.
\end{align*}
Taking $k=k'$ shows that every
$\mathrm e^{\mathrm ik\cdot\theta}$ belongs to $\{1,-1\}$, while
comparison of two choices of $k$ shows that all these values are equal
to one common sign $c$.  If $0\in\cA$, then necessarily $c=1$.
Restricting to $k\in\cK$ gives \eqref{eq:common-sign-translation}.

Conversely, suppose that \eqref{eq:common-sign-translation} holds, with
$c=1$ when $\varepsilon=1$.  Since $\cA$ contains both $k$ and $-k$,
every $k\in\cA$ has character value $c$.  Hence every pair sum satisfies
\begin{align*}
 \mathrm e^{\mathrm i(k+k')\cdot\theta}=c^2=1.
\end{align*}
Thus $\Gamma_{\mathrm{pair}}$ is contained in the kernel of the
nontrivial character
$\ell\mapsto\mathrm e^{\mathrm i\ell\cdot\theta}$ and is therefore a
proper subgroup of $\mathbb Z^3$.  This proves the equivalence of
\textup{(b)} and \textup{(c)}.
\end{proof}

The next proposition gives useful arithmetic characterizations of
saturation.

\begin{proposition}
\label{cor:Fourier-arithmetic-criterion}
If $\varepsilon=1$, then
\begin{align}\label{eq:constant-mode-lattice-test}
 \Gamma_{\mathrm{pair}}
 =\spanop_{\mathbb Z}\cK,
\end{align}
and hence the forcing is saturating if and only if
$\spanop_{\mathbb Z}\cK=\mathbb Z^3$.

If $\varepsilon=0$, then
\begin{align}\label{eq:no-constant-pair-sum-formula}
 \Gamma_{\mathrm{pair}}
 =\spanop_{\mathbb Z}\{2k_1,k_j-k_1:2\le j\le N\},
\end{align}
and the forcing is saturating if and only if
\begin{align}\label{eq:no-constant-arithmetic-test}
 \spanop_{\mathbb Z}\cK=\mathbb Z^3,
 \qquad
 \sum_{j=1}^Nq_jk_j=0
 \quad\text{for some }q\in\mathbb Z^N
 \text{ with }\sum_{j=1}^Nq_j\text{ odd}.
\end{align}
Under the first condition in \eqref{eq:no-constant-arithmetic-test}, the
second is equivalently the absence of
$\eta\in(\mathbb Z/2\mathbb Z)^3$ satisfying
\begin{align}\label{eq:parity-hyperplane-obstruction}
 \eta\cdot k_j=1\pmod2,\qquad 1\le j\le N.
\end{align}
\end{proposition}

\begin{proof}
If $\varepsilon=1$, then $0\in\cA$, so every $k\in\cK$ is itself a pair
sum $k+0$.  Conversely, every pair sum belongs to
$\spanop_{\mathbb Z}\cK$, proving
\eqref{eq:constant-mode-lattice-test}.

Suppose that $\varepsilon=0$.  The elements $2k_1$ and $k_j-k_1$ are pair
sums.  Conversely, every pair sum is generated by them, since
\begin{align*}
 k_i-k_j&=(k_i-k_1)-(k_j-k_1),\\
 k_i+k_j&=2k_1+(k_i-k_1)+(k_j-k_1).
\end{align*}
This proves \eqref{eq:no-constant-pair-sum-formula}.

Let $L=\spanop_{\mathbb Z}\cK$ and let
$\rho:\mathbb Z^N\to L$ be the linear map which sends the $j$th coordinate
vector to $k_j$, so that
\begin{align*}
 \rho(q_1,\ldots,q_N)=q_1k_1+\cdots+q_Nk_N.
\end{align*}
Then $L=\rho(\mathbb Z^N)$ and
\begin{align*}
 \Gamma_{\mathrm{pair}}
 =\rho\{q\in\mathbb Z^N:\textstyle\sum_jq_j\text{ is even}\}.
\end{align*}
Thus $\Gamma_{\mathrm{pair}}=L$ if and only if
$\ker\rho$ contains a vector of odd coordinate sum.  Together with
$L=\mathbb Z^3$, this gives \eqref{eq:no-constant-arithmetic-test}.

Finally, assume $L=\mathbb Z^3$ and reduce $\rho$ modulo two.  The
resulting map
$\bar\rho:(\mathbb Z/2\mathbb Z)^N\to(\mathbb Z/2\mathbb Z)^3$ is
surjective.  The absence of an odd relation is equivalent to the parity
functional $q\mapsto\sum_jq_j$ vanishing on $\ker\bar\rho$, and hence to
its factoring through $\bar\rho$.  Thus there exists
$\eta\in(\mathbb Z/2\mathbb Z)^3$ such that
$\eta\cdot k_j=1$ for every $j$, which proves
\eqref{eq:parity-hyperplane-obstruction}.
\end{proof}

Next we deduce the minimal number of forced profiles and explicit
saturating examples.

\begin{corollary}
\label{cor:minimal-pure-mode-systems}
Within the Fourier-mode class \eqref{eq:Fourier-mode-forcing}, the smallest
saturating system containing the constant profile has seven real profiles:
the constant together with the sine--cosine pairs associated with three
wave vectors forming a unimodular basis of $\mathbb Z^3$.  In particular,
one may take
\begin{align*}
 1,\qquad \cos x_i,\qquad \sin x_i,\qquad 1\le i\le3.
\end{align*}
Without the constant profile, at least four wave vectors, hence eight real
profiles, are necessary.  The minimum is attained by
\begin{align*}
 \cK=\{e_1,e_2,e_3,e_1+e_2\}.
\end{align*}
The three coordinate frequency pairs alone fail by an index two parity
obstruction.
\end{corollary}

\begin{proof}
Suppose first that the constant profile is forced.  By
\eqref{eq:constant-mode-lattice-test}, saturation is equivalent to
\begin{align*}
 \spanop_{\mathbb Z}\cK=\mathbb Z^3.
\end{align*}
Hence at least three wave vectors are necessary.  If $N=3$, the three
vectors generate $\mathbb Z^3$ if and only if the corresponding
$3\times3$ integer matrix has determinant $\pm1$, that is, if and only if
they form a unimodular basis of $\mathbb Z^3$.  Such a basis is therefore
sufficient.  Since each nonzero wave vector contributes the two real
profiles $\cos(k\cdot x)$ and $\sin(k\cdot x)$, the minimum number of real
profiles is seven.

Suppose now that the constant profile is not forced.  By
\eqref{eq:no-constant-arithmetic-test}, saturation requires both
$\spanop_{\mathbb Z}\cK=\mathbb Z^3$ and the existence of an odd integer
relation.  Again at least three wave vectors are necessary.  If $N=3$
and $k_1,k_2,k_3$ generate $\mathbb Z^3$, then they form a unimodular
basis and are therefore linearly independent over $\mathbb Z$.  Hence
\begin{align*}
 \sum_{j=1}^3q_jk_j=0
\end{align*}
implies $q_1=q_2=q_3=0$, so no odd relation exists.  Thus three wave
vectors can never be saturating when the constant profile is absent.
Four wave vectors suffice, since for
\begin{align*}
 \cK=\{e_1,e_2,e_3,e_1+e_2\},
\end{align*}
the first three vectors generate $\mathbb Z^3$ and the obstruction
\eqref{eq:parity-hyperplane-obstruction} is absent.
\end{proof}

The next proposition identifies the exact invariant Fourier phase space
generated by the forcing and the corresponding condition on the damping.
Fix $k^\circ\in\cA$ and recall from the proof of
\Cref{thm:sharp-Fourier-criterion} that
\begin{align}\label{eq:generated-Fourier-coset}
 \mathscr C=k^\circ+\Gamma_{\mathrm{pair}}.
\end{align}
Since $k-k^\circ=k+(-k^\circ)\in\Gamma_{\mathrm{pair}}$ for every
$k\in\cA$, the set $\mathscr C$ is independent of the choice of
$k^\circ$.  Moreover,
\begin{align}\label{eq:Fourier-coset-closure}
 \cA\subset\mathscr C,\qquad
 -\mathscr C=\mathscr C,\qquad
 \mathscr C+\mathscr C+\mathscr C=\mathscr C.
\end{align}
For $s\ge0$, define the real Fourier subspace
\begin{align*}
 H_{\mathscr C}^s
 =\{f\in H^s(\mathbb T^3;\mathbb R):
       \widehat f(\ell)=0\text{ for }\ell\notin\mathscr C\},
\end{align*}
and put
\begin{align*}
 \mathcal H_{\mathscr C}
 =H_{\mathscr C}^1\times H_{\mathscr C}^0,
 \qquad
 \mathcal H_{\mathscr C}^{s}
 =H_{\mathscr C}^{1+s}\times H_{\mathscr C}^{s}.
\end{align*}

\begin{proposition}
\label{prop:invariant-Fourier-sector}
Assume the Fourier-mode forcing hypothesis
\eqref{eq:Fourier-mode-forcing}.  Then, for every finite $s\ge0$,
\begin{align}\label{eq:relative-Fourier-density}
 \overline{\cS_\infty}^{\,H^s}=H_{\mathscr C}^s.
\end{align}
Suppose in addition that \eqref{eq:damping-assumption} holds and
$\supp\widehat a\subset\Gamma_{\mathrm{pair}}$.  Then
$\mathcal H_{\mathscr C}$ is invariant under the dynamics of
\eqref{eq:main-spde}.  For every $0<s<1/2$, the restricted Markov semigroup
satisfies the weighted Wasserstein spectral gap conclusion of
\Cref{thm:weighted-gap}, with $\Hh$ and $\Es{s}$ replaced by
$\mathcal H_{\mathscr C}$ and $\mathcal E_{s,\mathscr C}$,
respectively.  It also satisfies all conclusions of \Cref{thm:main} in
the relative energy phase space; in particular, it has a unique invariant
probability measure $\mu_{\mathscr C}$ with
\begin{align*}
 \supp_{\mathcal H_{\mathscr C}}\mu_{\mathscr C}
 =\mathcal H_{\mathscr C}.
\end{align*}
Furthermore,
\begin{align}\label{eq:full-space-iff-pair-lattice}
 \mathcal H_{\mathscr C}=\Hh
 \quad\Longleftrightarrow\quad
 \Gamma_{\mathrm{pair}}=\mathbb Z^3.
\end{align}
\end{proposition}

\begin{proof}
Equation \eqref{eq:exact-Fourier-coset} and the Fourier basis give
\eqref{eq:relative-Fourier-density}.  Since
$-\mathscr C=\mathscr C$, these complex Fourier modes define the real
spaces $H_{\mathscr C}^s$.  Moreover,
\begin{align}\label{eq:damping-sector-invariance}
 \mathsf M_aH_{\mathscr C}^s\subset H_{\mathscr C}^s
 \quad\Longleftrightarrow\quad
 \supp\widehat a\subset\Gamma_{\mathrm{pair}},
\end{align}
since $\Gamma_{\mathrm{pair}}$ is precisely the translation stabilizer of
the coset $\mathscr C$ and multiplication by a Fourier mode shifts Fourier
support by its frequency.

The Laplacian preserves Fourier support in $\mathscr C$, and the forced
profiles belong to $H_{\mathscr C}^s$ by
$\cA\subset\mathscr C$.  Moreover,
\eqref{eq:Fourier-coset-closure} gives
\begin{align*}
 u\in H_{\mathscr C}^1
 \quad\Longrightarrow\quad
 u^3\in H_{\mathscr C}^0.
\end{align*}
Indeed, this is immediate for trigonometric polynomials and follows in
general by approximation in $H^1$ and the embedding
$H^1(\mathbb T^3)\hookrightarrow L^6(\mathbb T^3)$.  Together with
\eqref{eq:damping-sector-invariance}, this proves invariance of
$\mathcal H_{\mathscr C}$.

All ingredients in the proof of \Cref{thm:weighted-gap} restrict to
$\mathcal H_{\mathscr C}$.  In particular, the damped linear semigroup and
the exact difference propagator preserve $\mathcal E_{s,\mathscr C}$,
\eqref{eq:relative-Fourier-density} gives the relative dense Malliavin
range and relative controllability, and the Lyapunov, compact core, and
stationary regularization estimates are unchanged.  Hence the restricted
semigroup has the same weighted Wasserstein spectral gap.  The proof of
\Cref{thm:main} then restricts verbatim and gives unique ergodicity,
relative full support, the moment and Sobolev regularity conclusions, and
exponential mixing in the relative energy topology.

Finally, by \eqref{eq:generated-Fourier-coset},
$\mathscr C=\mathbb Z^3$ if and only if
$\Gamma_{\mathrm{pair}}=\mathbb Z^3$.  This is equivalent to
$\mathcal H_{\mathscr C}=\Hh$, proving
\eqref{eq:full-space-iff-pair-lattice}.
\end{proof}

In particular, one has the following spectral gap and mixing result in dimensions one
and two.  Let $d\in\{1,2\}$ and consider on $\mathbb T^d$ the stochastic
cubic wave equation
\begin{align}\label{eq:lower-dimensional-wave}
 \dd u=v\,\dd t,\qquad
 \dd v=(\Delta u-a(x)v-u^3)\,\dd t+B\,\dd W_t,
\end{align}
where
\begin{align*}
 a\in C^\infty(\mathbb T^d),\qquad
 0<a_0\le a(x)\le a_1<\infty.
\end{align*}
As in the three-dimensional setting of
\Cref{thm:sharp-Fourier-criterion}, let
$\cK=\{k_1,\ldots,k_N\}\subset\mathbb Z^d\setminus\{0\}$ satisfy
$\cK\cap(-\cK)=\varnothing$, and suppose that
\begin{align*}
 \cS_0
 =
 \{1:\varepsilon=1\}\cup
 \{\cos(k\cdot x),\sin(k\cdot x):k\in\cK\},
 \qquad \varepsilon\in\{0,1\}.
\end{align*}
Set
\begin{align*}
 \cA_d=\{\pm k:k\in\cK\}\cup\{0:\varepsilon=1\},
 \qquad
 \Gamma_{\mathrm{pair}}^{(d)}
 =\spanop_{\mathbb Z}\{k+k':k,k'\in\cA_d\}.
\end{align*}

\begin{corollary}
\label{cor:lower-dimensional-unique-ergodicity}
If $\Gamma_{\mathrm{pair}}^{(d)}=\mathbb Z^d$, then the Markov semigroup
generated by \eqref{eq:lower-dimensional-wave} satisfies the conclusions
of \Cref{thm:weighted-gap,thm:main} on
$H^1(\mathbb T^d)\times L^2(\mathbb T^d)$.  In particular, it has a unique
invariant probability measure $\mu_d$ with
\begin{align*}
 \supp\mu_d
 =H^1(\mathbb T^d)\times L^2(\mathbb T^d).
\end{align*}
The following finite rank forcings satisfy $\Gamma_{\mathrm{pair}}^{(d)}=\mathbb Z^d$:
\begin{align*}
 d=1:\qquad&
 \{1,\ \cos x,\ \sin x\},
 \qquad\text{or}\qquad
 \{\cos x,\ \sin x,\ \cos(2x),\ \sin(2x)\},\\
 d=2:\qquad&
 \{1,\ \cos x_1,\ \sin x_1,\ \cos x_2,\ \sin x_2\},
\end{align*}
while in dimension two, without the constant profile, one may take
\begin{align*}
 \{\cos x_1,\ \sin x_1,\ 
 \cos x_2,\ \sin x_2,\ 
 \cos(x_1+x_2),\ \sin(x_1+x_2)\}.
\end{align*}
\end{corollary}

\begin{proof}
Embed $\mathbb Z^d$ into $\mathbb Z^3$ by
\begin{align*}
 \iota_d(k_1,\ldots,k_d)
 =(k_1,\ldots,k_d,0,\ldots,0),
\end{align*}
and regard functions on $\mathbb T^d$ as functions on $\mathbb T^3$
independent of the remaining variables.  Then
\begin{align*}
 \Gamma_{\mathrm{pair}}
 =\iota_d\bigl(\Gamma_{\mathrm{pair}}^{(d)}\bigr)
 =\iota_d(\mathbb Z^d),
\end{align*}
so the Fourier coset $\mathscr C$ of
\Cref{prop:invariant-Fourier-sector} is precisely
$\iota_d(\mathbb Z^d)$.  Hence
\begin{align*}
 H_{\mathscr C}^s\simeq H^s(\mathbb T^d)
\end{align*}
for every $s\in\mathbb R$.  Since the damping depends only on the first $d$
variables,
\begin{align*}
 \supp\widehat a
 \subset\iota_d(\mathbb Z^d)
 =\Gamma_{\mathrm{pair}}.
\end{align*}
Thus \eqref{eq:damping-sector-invariance} holds.  Applying
\Cref{prop:invariant-Fourier-sector} and identifying
$\mathcal H_{\mathscr C}$ with
$H^1(\mathbb T^d)\times L^2(\mathbb T^d)$ gives all the asserted conclusions.

The displayed examples follow directly from the arithmetic criterion.
With the constant profile, the displayed frequencies generate
$\mathbb Z^d$.  Without the constant profile, $\{1,2\}$ in dimension one
and $\{e_1,e_2,e_1+e_2\}$ in dimension two generate the full lattice and
admit an odd integer relation, so saturation follows from
\Cref{cor:Fourier-arithmetic-criterion}.
\end{proof}

\subsection{Proof of the main results}

\begin{proof}[Proof of \Cref{thm:weighted-gap}]
Fix $0<s<1/2$.  We apply \Cref{thm:stable-compact-spectral-gap} to the semigroup $(P_t)_{t\ge0}$ on $H=\Hh$, with coupling space $\mathbb H=\Es{s}$ and the endpoint family $\Phi_T:\Hh\times E_T\to\Hh$ from \eqref{eq:wiener-cm}.

Choose $\chi\ge\chi_0$ in \Cref{gap:prop:high} and fix any $\alpha_*\in(0,\alpha_0(\chi)]$.  We verify the five assumptions for all sufficiently large blocks.

First, by \eqref{gap:eq:uniform-terminal-exp}, after decreasing $\lambda_0>0$ if necessary,
\begin{align*}
 P_t\e^{\sigma\Vv}(\mathsf x)
 \le C\exp\{\sigma\e^{-c_1t}\Vv(\mathsf x)\},
 \qquad 0<\sigma\le\lambda_0.
\end{align*}
Thus \eqref{eq:QS-Lyapunov} holds with $\vartheta(t)=\e^{-c_1t}=o(t^{-1})$.  By \eqref{eq:periodic-V-equivalence}, every energy sublevel is bounded in $\Hh$, while
\begin{align*}
 H^1(\mathbb T^3)\Subset H^{-s}(\mathbb T^3),
 \qquad
 L^2(\mathbb T^3)\Subset H^{-1-s}(\mathbb T^3).
\end{align*}
Together with the weak lower semicontinuity of the quadratic terms and the compact embedding $H^1\Subset L^4$, this shows that the sublevels of $\Vv$ are compact in $\Es{s}$.  Hence \Cref{ass:QS-Lyapunov} holds.

For an arbitrary fixed block, the exact identity \eqref{gap:eq:weak-stable-compact} and \Cref{prop:wave-stable-compact-structure} verify \Cref{ass:QS-stable-compact}\textup{(i)}.  Taking
\begin{align*}
 \rho(T)=\norm{S_a(T)}_{\cL(\Es{s})}
\end{align*}
and using \eqref{eq:semigroup-decay-scale} gives $\rho(T)\to0$.  The operator continuity required in part~\textup{(ii)} is \Cref{gap:lem:weak-core-operators}, while the common exponent $p_0$ and the finite factor Cameron--Martin Taylor estimate in part~\textup{(iii)} are supplied by \Cref{gap:lem:wave-twojet-weak}.  Hence \Cref{ass:QS-stable-compact} holds.  The dense-range assumption \Cref{ass:QS-Malliavin} is exactly \Cref{thm:rough-weak-Hormander}.

For the synchronous estimate, \eqref{gap:eq:high-sync} at $\alpha=\alpha_*$ is precisely \eqref{eq:QS-weighted-synchronous}.  By \eqref{gap:eq:high-constants}, its constants satisfy $\log C_T\le C\alpha_*(1+T)$ and $c_{\alpha_*}>0$, so \Cref{ass:QS-synchronous} holds.  Finally, \Cref{ass:QS-accessibility} follows from \Cref{gap:prop:uniform-accessibility} with $\mathsf x_*=0$.

The abstract criterion therefore yields a sampling time $\tau>0$,
constants $r_0,\beta,\lambda>0$, and $\varrho_0\in(0,1)$ such that
\begin{align*}
 \Wass_d(\mu_1P_{n\tau},\mu_2P_{n\tau})
 \le\varrho_0^n\Wass_d(\mu_1,\mu_2),
 \qquad n\ge0,
\end{align*}
where $d_0$ and $d$ are given by
\eqref{gap:eq:basic-cost} and \eqref{gap:eq:weighted-cost}.  By
\Cref{gap:cor:finite-time-weighted-bound},
\begin{align*}
 \Wass_d(\mu_1P_r,\mu_2P_r)
 \le C_\tau\Wass_d(\mu_1,\mu_2),
 \qquad 0\le r\le\tau.
\end{align*}
Writing $t=n\tau+r$ with $0\le r<\tau$ and using the semigroup property,
we obtain
\begin{align*}
 \Wass_d(\mu_1P_t,\mu_2P_t)
 &\le C_\tau
 \Wass_d(\mu_1P_{n\tau},\mu_2P_{n\tau})\\
 &\le C_\tau\varrho_0^n\Wass_d(\mu_1,\mu_2)
 \le C\e^{-\varrho t}\Wass_d(\mu_1,\mu_2).
\end{align*}
This proves the continuous-time weighted Wasserstein spectral gap.
\end{proof}
\begin{proof}[Proof of \Cref{thm:main}]
Fix $0<s<1/2$ and let $d$ and the corresponding constants be supplied by
\Cref{thm:weighted-gap}.  By \Cref{prop:existence} there is an invariant
probability measure $\mu$, and \Cref{prop:lyapunov} gives every invariant
probability a finite $\e^{\lambda\Vv}$ moment.  Hence, if $\nu$ is another
invariant probability measure, the continuous-time spectral gap gives
\begin{align*}
 \Wass_d(\mu,\nu)
 \le C\e^{-\varrho t}\Wass_d(\mu,\nu),\qquad t\ge0.
\end{align*}
Since $\Wass_d(\mu,\nu)<\infty$, choosing $t$ so that $C\e^{-\varrho t}<1$
gives $\mu=\nu$.  The fixed time support theorem \Cref{thm:full-support}
gives $\supp\mu=\Hh$, the exponential moment
\eqref{eq:main-exp-moment} follows from \Cref{prop:lyapunov}, and
\eqref{eq:main-strong-support} is exactly
\Cref{prop:stationary-regularization}.
Since $d_0\le d$, the spectral gap and the independent coupling imply
\begin{align}\label{eq:d0mixing}
 \Wass_{d_0}(P_t(\mathsf x,\cdot),\mu)\le C\e^{-\varrho t}\Wass_d(\delta_{\mathsf x},\mu)
 \le C\e^{-\varrho t}\bigl(1+\e^{\lambda\Vv(\mathsf x)/2}\bigr),
 \qquad t\ge0,
\end{align}
where the last estimate uses $d_0\le1$ and
\eqref{eq:main-exp-moment}.

It remains to upgrade the convergence to the energy topology.  Fix
$0<s_0<2/5$ and set
\begin{align*}
 \theta=\frac{s_0}{1+s+s_0}\in(0,1).
\end{align*}
Interpolation between $\Es{s}=\Hs{-1-s}$ and $\Hs{s_0}$ gives
\begin{align*}
 \norm z_{\Hh}
 \le C\norm z_{\Es{s}}^\theta
       \norm z_{\Hs{s_0}}^{1-\theta}.
\end{align*}
For $\varepsilon>0$, realise an $\varepsilon$-optimal coupling $(X,Z)$ of
$P_t(\mathsf x,\cdot)$ and $\mu$ so that $X$ is the actual solution
endpoint issued from $\mathsf x$, and write
\begin{align*}
 X=S_a(t)\mathsf x+Y_t.
\end{align*}
Then \eqref{eq:strong-remainder-energy-uniform} gives polynomial moments
of $Y_t$ in $\Hs{s_0}$, while
\Cref{prop:stationary-regularization} gives moments of every order for
$Z$ in $\Hs{s_0}$.  Put
\begin{align*}
 M=1+\norm{Y_t}_{\Hs{s_0}}+\norm Z_{\Hs{s_0}}.
\end{align*}
By \eqref{eq:semigroup-decay-scale} and interpolation,
\begin{align}\label{eq:interpolation}
\begin{split}
\norm{X-Z}_{\Hh}
 &\le C\e^{-\varpi_0t}\norm{\mathsf x}_{\Hh}
 +C\left(
 \norm{X-Z}_{\Es{s}}
 +C\e^{-\varpi_0t}\norm{\mathsf x}_{\Hh}
 \right)^\theta M^{1-\theta}\\
 &\le A_t + C(D+A_t)^{\theta} M^{1-\theta}  
\end{split}
\end{align}
by setting
\begin{align*}
A_t=C\e^{-\varpi_0t}\norm{\mathsf x}_{\Hh},
\qquad
D=\frac{\norm{X-Z}_{\Es{s}}}{r_0}.
\end{align*} 
Moreover, since $(X,Z)$ is $\varepsilon$-optimal for
$\Wass_{d_0}(P_t(\mathsf x,\cdot),\mu)$, in view of \eqref{gap:eq:basic-cost}, one has 
\begin{align*}
 \E\left[
 1\wedge
 \left(\frac{\norm{X-Z}_{\Es{s}}}{r_0}\right)^\alpha
 \right]
 \le
 \Wass_{d_0}(P_t(\mathsf x,\cdot),\mu)+\varepsilon.
\end{align*}
Choose $0<\eta<\min\{1,\alpha/\theta\}$ and then $q>1$ sufficiently close
to one that $q\theta\eta<\alpha$. Writing $q'=q/(q-1)$, the preceding
interpolation estimate \eqref{eq:interpolation}, H\"older's inequality, Jensen's inequality and the moment bounds give some $N$ such that 
\begin{align*}
\begin{split}
 \E\bigl[1\wedge\norm{X-Z}_{\Hh}^{\eta}\bigr]
 &\le C A_t^\eta+C A_t^{\theta\eta}\E M^{(1-\theta)\eta}
 +C\E\bigl[(1\wedge D^{\theta\eta})M^{(1-\theta)\eta}\bigr]\\
 &\le C(1+\Vv(\mathsf x))^N\e^{-\varpi_0\theta\eta t}
 +C\bigl(\E[1\wedge D^{q\theta\eta}]\bigr)^{1/q}
      \bigl(\E M^{q'(1-\theta)\eta}\bigr)^{1/q'}\\
 &\le C(1+\Vv(\mathsf x))^N\left[
 \e^{-\varpi_0\theta\eta t}
 +\bigl(\E[1\wedge D^\alpha]\bigr)^{\theta\eta/\alpha}
 \right]\\
 &\le C(1+\Vv(\mathsf x))^N\left[
 \e^{-\varpi_0\theta\eta t}
 +\bigl(\Wass_{d_0}(P_t(\mathsf x,\cdot),\mu)+\varepsilon\bigr)^{\theta\eta/\alpha}
 \right].
\end{split}
\end{align*}
Set $\kappa_0=\theta\eta/\alpha<1$.  Since
\begin{align*}
 (1+V)^N\bigl(1+\e^{\lambda V/2}\bigr)^{\kappa_0}
 \le C\bigl(1+\e^{\lambda V/2}\bigr),
 \qquad V\ge1,
\end{align*}
the weak-topology estimate \eqref{eq:d0mixing} implies, after letting
$\varepsilon\downarrow0$ and decreasing the exponential rate if necessary,
\begin{align*}
 \Wass_{d_{\Hh,\eta}}(P_t(\mathsf x,\cdot),\mu)
 \le C\e^{-\varrho t}
 \bigl(1+\e^{\lambda\Vv(\mathsf x)/2}\bigr).
\end{align*}
This is \eqref{gap:eq:energy-topology-gap}.
\end{proof}

\begin{proof}[Proof of \Cref{thm:sector-and-lower-dimensional}]
The Fourier-sector assertions are exactly \Cref{prop:invariant-Fourier-sector}, while the lower-dimensional conclusion is \Cref{cor:lower-dimensional-unique-ergodicity}.
\end{proof}

\appendix
\section{Estimates for the stochastic wave equation}\label{app:analytic}
This appendix collects the analytic estimates that support the main argument: the modified energy and Lyapunov bounds, endpoint differentiability, and the weak phase stable--compact estimates used in the spectral gap criterion.

\subsection{Lyapunov structure and high energy dissipation}
Standard energy space theory for stochastic nonlinear damped wave equations with polynomial nonlinearities yields, for every $\mathsf x\in\Hh$, a unique global adapted mild solution $X\in C([0,\infty);\Hh)$ almost surely, with continuous dependence on the initial condition on finite time intervals, and hence defines a Feller Markov semigroup.  We refer to \cite{BarbuDaPrato2002,Kim2004,BrzezniakOndrejatSeidler2016} for the standard construction.

The modified energy introduced in \eqref{eq:V-energy} also controls the constant Fourier mode.  The following coercivity statement explains the additive constant in the Foster drift.

\begin{lemma}
\label{lem:periodic-energy-coercivity}
For $\eta>0$ sufficiently small,
\begin{align}\label{eq:periodic-V-equivalence}
 c\bigl(1+\norm u_{H^1}^2+\norm v_{L^2}^2+\norm u_{L^4}^4\bigr)
 \le\Vv(u,v)\le C\bigl(1+\norm u_{H^1}^2+\norm v_{L^2}^2+\norm u_{L^4}^4\bigr).
\end{align}
Moreover, with
\begin{align*}
\mathsf V(u,v)=\int_{\mathbb T^3}(a-\eta)|v|^2+\eta\norm{\nabla u}_{L^2}^2+\eta\norm u_{L^4}^4,
\end{align*}
there are $\kappa,C>0$ such that
\begin{align}\label{eq:periodic-dissipation-coercivity}
 \mathsf V(u,v)\ge\kappa\Vv(u,v)-C.
\end{align}
\end{lemma}

\begin{proof}
Choose $0<\eta\le\min\{1,a_0/4\}$.  The complete estimate for the cross term is
\begin{align}\label{eq:periodic-cross-young}
 \eta\abs{\ip uv}\le\frac14\norm v_{L^2}^2+\eta^2\norm u_{L^2}^2\le\frac14\norm v_{L^2}^2+\frac{\eta a_0}{4}\norm u_{L^2}^2.
\end{align}
Consequently,
\begin{align*}
\Vv(u,v)\ge1+\frac12\norm{\nabla u}_{L^2}^2+\frac14\norm v_{L^2}^2+\frac14\norm u_{L^4}^4+\frac{\eta a_0}{4}\norm u_{L^2}^2.
\end{align*}
The reverse estimate follows from the first inequality in \eqref{eq:periodic-cross-young} and $\frac\eta2\int au^2\le\frac{\eta a_1}{2}\norm u_{L^2}^2$.  This proves \eqref{eq:periodic-V-equivalence}.  Moreover,
\begin{align*}
\mathsf V(u,v)\ge\frac{3a_0}{4}\norm v_{L^2}^2+\eta\norm{\nabla u}_{L^2}^2+\eta\norm u_{L^4}^4.
\end{align*}
On the finite-volume torus,
\begin{align*}
\norm u_{L^2}^2\le\varepsilon\norm u_{L^4}^4+C_\varepsilon.
\end{align*}
Combining this inequality with the upper bound in \eqref{eq:periodic-V-equivalence} gives $\Vv(u,v)\le C_1+C_2\mathsf V(u,v)$, which is \eqref{eq:periodic-dissipation-coercivity}.
\end{proof}

\begin{proposition}
For every $\mathsf x\in\Hh$, the global solution $X=(u,v)$ satisfies for every $t\ge0$,
\begin{align}\label{eq:energy-identity}
 \Vv(X_t)+\int_0^t\mathsf V(X_s)\,\dd s=\Vv(\mathsf x)+b_0t+M_t,
\end{align}
where
\begin{align*}
b_0=\frac12\norm B_{\cL_2(\R^m;L^2)}^2,\qquad M_t=\int_0^t\ip{v_s+\eta u_s}{B\,\dd W_s}_{L^2}.
\end{align*}
There are $\kappa,\beta,C_*>0$ such that
\begin{align}\label{eq:coercive-qv}
 \mathsf V(\mathsf x)\ge\kappa\Vv(\mathsf x)-C_*,
 \qquad \dd\langle M\rangle_t\le\beta\Vv(X_t)\,\dd t.
\end{align}
Moreover, for every $p\ge1$ and $T<\infty$,
\begin{align}\label{eq:finite-horizon-energy-moment}
 \E\sup_{0\le t\le T}(1+\Vv(X_t))^p\le C_{p,T}(1+\Vv(\mathsf x))^p.
\end{align}
\end{proposition}

\begin{proof}
Recall the energy $\cE$ from \eqref{eq:E-energy}. By It\^o's formula, 
\begin{align*}
 \dd\cE&=-\int_{\mathbb T^3}a|v|^2\,\dd t+b_0\,\dd t+\ip v{B\,\dd W},\\
 \dd\ip uv&=\left(\norm v_{L^2}^2-\norm{\nabla u}_{L^2}^2-\norm u_{L^4}^4-\int_{\mathbb T^3}auv\right)\dd t+\ip u{B\,\dd W},\\
 \frac12\dd\int_{\mathbb T^3}au^2&=\int_{\mathbb T^3}auv\,\dd t.
\end{align*}
Multiplying the last two identities by $\eta$ and adding them to the first gives \eqref{eq:energy-identity}.  The first estimate in \eqref{eq:coercive-qv} is \Cref{lem:periodic-energy-coercivity}, while
\begin{align*}
\dd\langle M\rangle_t=\norm{B^*(v_t+\eta u_t)}_{\R^m}^2\,\dd t\le C_B\bigl(\norm{v_t}_{L^2}^2+\norm{u_t}_{L^2}^2\bigr)\,\dd t\le\beta\Vv(X_t)\,\dd t.
\end{align*}

For the moment estimate, applying It\^o's formula to $Y_t^p$ where $Y_t=1+\Vv(X_t)$ we have 
\begin{align}\label{eq:Ytp}
\dd Y_t^p
\le\left[-p\kappa Y_t^p+pC_0Y_t^{p-1}
+\frac{p(p-1)}{2}\beta Y_t^{p-2}\Vv(X_t)\right]\dd t
+pY_t^{p-1}\,\dd M_t,
\end{align}
where $C_0=\kappa+C_*+b_0$. The required \eqref{eq:finite-horizon-energy-moment} then follows from  \eqref{eq:coercive-qv}, Burkholder--Davis--Gundy's inequality, Young's inequality and Gronwall.
\end{proof}

\begin{proposition}[Lyapunov estimates]\label{prop:lyapunov}
For every $p\ge1$ there are $c_p>0$ and $C_p<\infty$ such that
\begin{align}\label{eq:poly-Foster}
 P_t(1+\Vv)^p(\mathsf x)\le \e^{-c_pt}(1+\Vv(\mathsf x))^p+C_p.
\end{align}
There is $\lambda_0>0$ such that
\begin{align}\label{eq:exp-Foster}
 P_t\e^{\lambda_0\Vv}(\mathsf x)\le \e^{-c_0t}\e^{\lambda_0\Vv(\mathsf x)}+C_0.
\end{align}
Every invariant probability $\nu$ satisfies the corresponding polynomial and exponential moment bounds as well.
\end{proposition}

\begin{proof}
Since $\Vv(X_t)\le Y_t$, \eqref{eq:Ytp} and Young's inequality give
\begin{align*}
\dd Y_t^p\le\left[-p\kappa Y_t^p+C_pY_t^{p-1}\right]\dd t+\dd N_t^{(p)}\le\left[-c_pY_t^p+C_p\right]\dd t+\dd N_t^{(p)}
\end{align*}
for a local martingale $N^{(p)}$. 
Taking expectations and applying Gronwall's inequality yields
\begin{align*}
P_t(1+\Vv)^p(\mathsf x)\le\e^{-c_pt}(1+\Vv(\mathsf x))^p+C_p,
\end{align*}
which proves \eqref{eq:poly-Foster}.

For $\lambda>0$, another application of It\^o's formula together with \eqref{eq:energy-identity} and \eqref{eq:coercive-qv} gives
\begin{align*}
\dd\e^{\lambda\Vv(X_t)}
\le\e^{\lambda\Vv(X_t)}
\left[-\lambda\kappa\Vv(X_t)+\lambda(C_*+b_0)
+\frac{\lambda^2\beta}{2}\Vv(X_t)\right]\dd t+\dd N_t^{(\lambda)}.
\end{align*}
Choose $\lambda_0>0$ so that $\lambda_0\beta\le\kappa$.  Then
\begin{align*}
\dd\e^{\lambda_0\Vv(X_t)}
\le\left[-c_0\e^{\lambda_0\Vv(X_t)}+C_0\right]\dd t+\dd N_t^{(\lambda_0)}.
\end{align*}
Taking expectations and applying Gronwall's inequality proves \eqref{eq:exp-Foster}.

Finally, let $\nu$ be invariant and let $L$ denote either $(1+\Vv)^p$ or $\e^{\lambda_0\Vv}$.  Fix $t>0$ and put $\rho=\e^{-ct}<1$, so that
\begin{align*}
 P_tL\le\rho L+C.
\end{align*}
For $R>0$, set $L_R=L\wedge R$.  By invariance and $L_R\le L$,
\begin{align*}
 \int_{\Hh}L_R\,\dd\nu
 =\int_{\Hh}P_tL_R\,\dd\nu
 \le\int_{\{L\le R\}}(\rho L+C)\,\dd\nu+R\nu\{L>R\}.
\end{align*}
Since
\begin{align*}
 \int_{\Hh}L_R\,\dd\nu=\int_{\{L\le R\}}L\,\dd\nu+R\nu\{L>R\},
\end{align*}
we obtain
\begin{align*}
 (1-\rho)\int_{\{L\le R\}}L\,\dd\nu\le C.
\end{align*}
Letting $R\to\infty$ and using monotone convergence gives
\begin{align*}
 \int_{\Hh}L\,\dd\nu\le\frac{C}{1-\rho}.
\end{align*}
This proves the asserted polynomial and exponential moment bounds.
\end{proof}

The preceding energy identity also gives the exponential occupation
estimate below. In particular, it verifies \Cref{ass:QS-Lyapunov} for the wave equation.  By
\eqref{eq:energy-identity} and \eqref{eq:coercive-qv}, after changing the
constants,
\begin{align}\label{eq:energy-semimartingale}
 \dd\Vv(X_t)
 &\le[-\kappa\Vv(X_t)+C_*]\,\dd t+\dd M_t,
 &
 \dd\langle M\rangle_t
 &\le\beta_*\Vv(X_t)\,\dd t.
\end{align}

\begin{lemma}\label{gap:lem:occupation}
There is $r_*>0$ such that for every $r,\sigma\in[0,r_*]$ and $T>0$,
\begin{align}\label{gap:eq:FK}
 \E \exp\left\{
 \sigma\Vv(X_T)+r\int_0^T\Vv(X_t)\,\dd t
 \right\}
 \le C_T\exp\{q_T(\sigma,r)\Vv(\mathsf x)\},
\end{align}
where $q_T(\sigma,r)$ is the value at time $T$ of the solution to the Riccati equation
\begin{align}\label{gap:eq:riccati}
 q'=r-\kappa q+\tfrac12\beta_*q^2,
 \qquad q(0)=\sigma.
\end{align}
In particular, there are $c_1,C>0$ independent of $T$ such that
\begin{align}\label{gap:eq:uniform-terminal-exp}
 P_T\e^{\sigma\Vv}(\mathsf x)
 \le C\exp\{\sigma\e^{-c_1T}\Vv(\mathsf x)\}
\end{align}
for all $\sigma\in [0,r_*]$.
\end{lemma}
\begin{proof}
Fixing any $0<r_*<\min\{\kappa/\beta_*,\kappa^2/(2\beta_*)\}$, then $[0,\kappa/\beta_*]$ is invariant under the flow of \eqref{gap:eq:riccati}, implying for $r,\sigma\in[0,r_*]$, 
\begin{align}\label{eq:qt-stable}
 0\le q_t(\sigma,r)\le \frac{\kappa}{\beta_*},\qquad t\ge0.
\end{align}
Fix $T>0$ and set $ p_t=q_{T-t}(\sigma,r)$.  Applying It\^o's
formula to
\begin{align*}
 Z_t=\exp\left\{
  p_t\Vv(X_t)+r\int_0^t\Vv(X_\tau)\,\dd\tau
 \right\}
\end{align*}
and using \eqref{eq:energy-semimartingale}, we obtain
\begin{align*}
 \dd Z_t
 &\le Z_t\left[
 \left( p_t'-\kappa p_t+r
 +\frac12\beta_* p_t^2\right)\Vv(X_t)
 +C_* p_t\right]\dd t
 + p_tZ_t\,\dd M_t\\
 &=C_* p_t Z_t \dd t
 + p_tZ_t\,\dd M_t.
\end{align*}
Therefore
\begin{align*}
 Z_T
 \le Z_0\exp\left\{C_*\int_0^T p_t\,\dd t\right\}
 \mathcal E_T, \text{ where } \mathcal E_T = \exp\left\{\int_0^T p_t\,\dd M_t-\frac12\int_0^T p_t^2\,\dd \langle M\rangle_t\right\}. 
\end{align*}
Since $\mathcal E$ is a nonnegative
local martingale, $\E\mathcal E_T\le1$.  Thus by \eqref{eq:qt-stable},
\begin{align}\label{eq:ZT}
 \E Z_T
 \le
 \exp\left\{C_*\int_0^Tq_t(\sigma,r)\,\dd t\right\}
 \exp\{q_T(\sigma,r)\Vv(\mathsf x)\}
 \le C_T\exp\{q_T(\sigma,r)\Vv(\mathsf x)\},
\end{align}
which proves \eqref{gap:eq:FK}. For $r=0$, from \eqref{gap:eq:riccati} and \eqref{eq:qt-stable} we deduce 
\begin{align*}
 q_t'(\sigma,0)\le-\frac\kappa2q_t(\sigma,0).
\end{align*}
Consequently, taking $r=0$ in \eqref{eq:ZT}  gives
\begin{align*}
 P_T\e^{\sigma\Vv}(\mathsf x)
 \le
 \exp\left\{\frac{2C_*}{\kappa}\sigma\right\}
 \exp\left\{\sigma\e^{-\kappa T/2}\Vv(\mathsf x)\right\},
\end{align*}
and \eqref{gap:eq:uniform-terminal-exp} follows with $c_1=\kappa/2$ and $C=\exp\left\{2C_*r_*/\kappa\right\}$.
\end{proof}

Recall that, for $\chi>0$, the weak phase premetric is
\begin{align}\label{eq:weak-premetric}
 \mathsf D_\chi(\mathsf x,\mathsf y)
 =\exp\{\chi[\Vv(\mathsf x)+\Vv(\mathsf y)]\}
 \norm{\mathsf x-\mathsf y}_{\Es{s}},
 \qquad \mathsf x,\mathsf y\in\Hh.
\end{align}
Let $X$ and $Y$ be the solutions to the wave equation issued from $\mathsf x$ and
$\mathsf y$ under the same noise.

\begin{proposition}\label{gap:prop:high}
There is $\chi_0<\infty$ such that for every fixed
$\chi\ge\chi_0$, there is $\alpha_0(\chi)\in(0,1]$ with the following
property.  Whenever $0<\alpha\le\alpha_0(\chi)$, for $T\geq T_0:=4\log2/\kappa$ there are $C_T\ge1$ and $c_\alpha>0$ such that
\begin{align}\label{gap:eq:high-sync}
 \E\mathsf D_\chi(X_T,Y_T)^\alpha
 \le C_T\e^{-c_\alpha[\Vv(\mathsf x)+\Vv(\mathsf y)]}
 \mathsf D_\chi(\mathsf x,\mathsf y)^\alpha,
\end{align}
where the constants may be chosen
so that
\begin{align}\label{gap:eq:high-constants}
 \log C_T\le C_\chi\alpha(1+T),
 \qquad
 c_\alpha = \frac{C_0\alpha}{4\kappa}. 
\end{align}
Consequently, with
\begin{align*}
 R_T=1\vee c_\alpha^{-1}\log(4C_T),
\end{align*}
the factor on the right-hand side of \eqref{gap:eq:high-sync} is at most
$1/4$ whenever
$\Vv(\mathsf x)+\Vv(\mathsf y)\ge R_T$ with 
\begin{align*}
 R_T\le C_\chi(1+T), \quad T\ge T_0\vee\alpha^{-1}. 
\end{align*}
\end{proposition}

\begin{proof}
Write $U=X-Y=(w,\dot w)$ and $\mathscr W=u^2+u\widetilde u+\widetilde u^2$.  Then
\begin{align*}
 \ddot w-\Delta w+a\dot w+\mathscr Ww=0,
\end{align*}
or, relative to the massive damped-wave semigroup $S_a(t)$,
\begin{align*}
 U_t=S_a(t)U_0+\int_0^tS_a(t-\tau)\binom0{(1-\mathscr W_\tau)w_\tau}\,\dd\tau.
\end{align*}
By \eqref{eq:semigroup-decay-scale} and \eqref{gap:eq:negative-trilinear},
\begin{align*}
 \norm{U_t}_{\Es{s}}
 &\le C\norm{U_0}_{\Es{s}}+C\int_0^t\left[1+\norm{u_\tau}_{H^1}^2+\norm{\widetilde u_\tau}_{H^1}^2\right]\norm{U_\tau}_{\Es{s}}\,\dd\tau\\
 &\le C\norm{U_0}_{\Es{s}}+C\int_0^t[\Vv(X_\tau)+\Vv(Y_\tau)]\norm{U_\tau}_{\Es{s}}\,\dd\tau.
\end{align*}
Hence Gronwall's inequality gives
\begin{align}\label{gap:eq:high-weak-difference}
 \norm{U_T}_{\Es{s}}\le C_0\exp\left\{C_0\int_0^T[\Vv(X_t)+\Vv(Y_t)]\,\dd t\right\}\norm{\mathsf x-\mathsf y}_{\Es{s}}
\end{align}
for some structural constant $C_0\ge1$.  Set $\sigma_\alpha=\alpha\chi$ and $r_\alpha=C_0\alpha$.  Raising \eqref{gap:eq:high-weak-difference} to the power $\alpha$, using \eqref{eq:weak-premetric}, and applying Cauchy--Schwarz together with \eqref{eq:ZT} to the two solutions gives
\begin{align}\label{gap:eq:high-riccati-reduction}
 \E\mathsf D_\chi(X_T,Y_T)^\alpha
 \le C_T\exp\left\{\left[-\sigma_\alpha+\frac12q_T(2\sigma_\alpha,2r_\alpha)\right][\Vv(\mathsf x)+\Vv(\mathsf y)]\right\}\mathsf D_\chi(\mathsf x,\mathsf y)^\alpha,
\end{align}
where
\begin{align*}
 C_T=C_0^\alpha\exp\left\{C_*\int_0^Tq_t(2\sigma_\alpha,2r_\alpha)\,\dd t\right\}.
\end{align*}

It remains to control the Riccati term.  Let $q_-(\alpha)<q_+(\alpha)$ be the two equilibria of $q'=2r_\alpha-\kappa q+\frac12\beta_*q^2$.  Since $r_\alpha=C_0\alpha$,
\begin{align*}
 q_-(\alpha)=\frac{2C_0}{\kappa}\alpha+O(\alpha^2),\qquad q_+(\alpha)=\frac{2\kappa}{\beta_*}+O(\alpha).
\end{align*}
Set $\chi_0=2C_0/\kappa$ and $\delta=C_0/(4\kappa)$.  After decreasing $r_*$ in \Cref{gap:lem:occupation} once and for all, we may assume that $r_*\le\kappa/(2\beta_*)$ and, whenever $2C_0\alpha\le r_*$,
\begin{align*}
 q_-(\alpha)\le2\alpha(\chi_0-2\delta),\qquad q_+(\alpha)\ge\frac{\kappa}{\beta_*}.
\end{align*}
For fixed $\chi\ge\chi_0$, take
\begin{align*}
 \alpha_0(\chi)=\min\left\{\frac{r_*}{2\chi},\frac{r_*}{2C_0}\right\}.
\end{align*}
Then, for $0<\alpha\le\alpha_0(\chi)$,
\begin{align*}
 q_-(\alpha)\le2\alpha(\chi-2\delta)<q_0=2\alpha\chi<q_+(\alpha).
\end{align*}
Thus, writing $q_t=q_t(2\sigma_\alpha,2r_\alpha)$,
\begin{align*}
 q_t'=\frac{\beta_*}{2}[q_t-q_-(\alpha)][q_t-q_+(\alpha)]
\end{align*}
shows that $q_t$ decreases to $q_-(\alpha)$.  In particular, $0\le q_t\le2\alpha\chi$, and hence
\begin{align*}
 C_T\le C_0^\alpha\e^{2C_*\alpha\chi T},\qquad \log C_T\le C_\chi\alpha(1+T).
\end{align*}
Moreover, $q_+(\alpha)-q_t\ge\kappa/(2\beta_*)$, so
\begin{align*}
 \frac{\dd}{\dd t}[q_t-q_-(\alpha)]\le-\frac{\kappa}{4}[q_t-q_-(\alpha)].
\end{align*}
Therefore, for $T\ge T_0:=4\log2/\kappa$,
\begin{align*}
 q_T
 &\le q_-(\alpha)+\frac12[2\alpha\chi-q_-(\alpha)]
 \le2\alpha\chi-2\delta\alpha.
\end{align*}
Consequently, as $\sigma_{\alpha}=\alpha\chi$,
\begin{align*}
 -\sigma_\alpha+\frac12q_T(2\sigma_\alpha,2r_\alpha)\le-\delta\alpha.
\end{align*}
Substituting this into \eqref{gap:eq:high-riccati-reduction} proves \eqref{gap:eq:high-sync} with $c_\alpha=\delta\alpha$.

Finally, replacing $C_T$ by $C_0^\alpha\e^{2C_*\alpha\chi T}$ if necessary, take
\begin{align*}
 R_T=1\vee\frac{\log(4C_T)}{\delta\alpha}.
\end{align*}
Then the factor in \eqref{gap:eq:high-sync} is at most $1/4$ whenever $\Vv(\mathsf x)+\Vv(\mathsf y)\ge R_T$.  Moreover,
\begin{align*}
 R_T\le1+\frac{\log4}{\delta\alpha}+\frac{\log C_0}{\delta}+\frac{2C_*\chi}{\delta}T,
\end{align*}
so $R_T\le C_\chi(1+T)$ whenever $T\ge T_0\vee\alpha^{-1}$.  This proves \eqref{gap:eq:high-constants}.
\end{proof}

Fix $\chi\ge\chi_0$ and $0<\alpha\le\alpha_0(\chi)$ from \Cref{gap:prop:high} and $\lambda_0$ from \Cref{prop:lyapunov}. For any $r_0,\beta>0$ and  $0<\lambda\le\lambda_0$, let $d_0$
and $d$ be given by \eqref{gap:eq:basic-cost} and
\eqref{gap:eq:weighted-cost}.

\begin{corollary}
\label{gap:cor:finite-time-weighted-bound}
For every
$\tau<\infty$, there is $C_\tau<\infty$ such that,
for solutions $X,Y$ issued from $\mathsf x,\mathsf y$, one has 
\begin{align}\label{gap:eq:finite-time-weighted-bound}
 \E d(X_t,Y_t)\le C_\tau d(\mathsf x,\mathsf y),
 \qquad 0\le t\le\tau.
\end{align}
Consequently,
\begin{align}\label{gap:eq:finite-time-Wasserstein-bound}
 \Wass_d(\mu_1P_r,\mu_2P_r)
 \le C_\tau\Wass_d(\mu_1,\mu_2),
 \qquad 0\le r\le\tau,
\end{align}
for all probability measures with finite $\e^{\lambda\Vv}$ moment.
\end{corollary}

\begin{proof}
Note that \eqref{gap:eq:high-riccati-reduction} and the bound
$q_t(2\alpha\chi,2C_0\alpha)\le2\alpha\chi$ established in the above proof give 
\begin{align*}
 \sup_{0\le t\le\tau}
 \E\mathsf D_\chi(X_t,Y_t)^\alpha
 \le C_\tau\mathsf D_\chi(\mathsf x,\mathsf y)^\alpha.
\end{align*}
As $d_0=1\wedge\left(r_0^{-1}\mathsf D_{\chi}\right)^\alpha$,
it follows that
\begin{align*}
 \sup_{0\le t\le\tau}\E d_0(X_t,Y_t)
 \le C_\tau d_0(\mathsf x,\mathsf y).
\end{align*}
Moreover, \Cref{prop:lyapunov} gives
\begin{align*}
 \sup_{0\le t\le\tau}P_t\e^{\lambda\Vv}(\mathsf x)
 \le C_\tau\e^{\lambda\Vv(\mathsf x)}.
\end{align*}
Therefore, Cauchy--Schwarz yields
\begin{align*}
 \E d(X_t,Y_t)
 \le
 \bigl(\E d_0(X_t,Y_t)\bigr)^{1/2}
 \left(
 1+\beta\E\e^{\lambda\Vv(X_t)}
 +\beta\E\e^{\lambda\Vv(Y_t)}
 \right)^{1/2}\le C_\tau d(\mathsf x,\mathsf y),
 \qquad 0\le t\le\tau,
\end{align*}
which proves \eqref{gap:eq:finite-time-weighted-bound}.  Integrating against an arbitrary coupling of $\mu_1,\mu_2$ and
taking the infimum proves \eqref{gap:eq:finite-time-Wasserstein-bound}.
\end{proof}

\begin{remark}\label{rem:highcontraction}
\Cref{gap:prop:high} verifies \Cref{ass:QS-synchronous} for the wave equation.
More generally, the same high energy contraction follows from an
exponential occupation estimate of the form \eqref{gap:eq:FK} together
with a synchronous difference bound of the form
\eqref{gap:eq:high-weak-difference}, provided that, for small $\alpha$,
the associated Riccati flow satisfies
\begin{align*}
 \frac12 q_T(2\alpha\chi,2C_0\alpha)<\alpha\chi
\end{align*}
for all sufficiently large $T$.
\end{remark}

\subsection{Technical estimates for the stable--compact structure}
We collect here the estimates needed for parts~\textup{(ii)} and~\textup{(iii)} of \Cref{ass:QS-stable-compact}.  We first establish the variational energy and fractional Cameron--Martin two-jet bounds used below.  We then prove operator norm continuity of the compact defect and the Malliavin derivative on compact weak-energy cores, followed by the moment and finite factor Taylor estimates required in part~\textup{(iii)}.  Let $X=(u,v)$ be a solution on $[0,T]$ and recall that $\Lambda=\sqrt{I-\Delta}$.

\begin{lemma}\label{lem:variational-energies}
For $f\in L^2(0,T;L^2(\mathbb T^3))$, let $R=(r,\partial_t r)$ solve
\begin{align}\label{eq:linear-variation-general}
 \partial_t^2 r+\Lambda^2r+a\partial_t r+(3u^2-1)r=f,\qquad R(0)=0.
\end{align}
Then
\begin{align}\label{eq:linear-variation-energy-bound}
 \sup_{t\le T}\norm{R(t)}_{\Hh}^2
 \le C\exp\left\{
 C\int_0^T(1+\Vv(X_s))\,\dd s
 \right\}\norm f_{L^2_tL^2_x}^2,
\end{align}
where $C$ is independent of $T$.
\end{lemma}

\begin{proof}
Writing \eqref{eq:linear-variation-general} relative to the massive
damped-wave semigroup gives
\begin{align*}
 R_t=\int_0^tS_a(t-s)
 \binom0{f_s+(1-3u_s^2)r_s}\,\dd s.
\end{align*}
By $H^1(\mathbb T^3)\hookrightarrow L^6(\mathbb T^3)$ and
\eqref{eq:periodic-V-equivalence},
\begin{align*}
 \norm{(1-3u_s^2)r_s}_{L^2}
 \le C\bigl(1+\norm{u_s}_{H^1}^2\bigr)\norm{r_s}_{H^1}
 \le C(1+\Vv(X_s))\norm{R_s}_{\Hh}.
\end{align*}
Since $S_a(t)$ is uniformly bounded on $\Hh$ for $t\ge0$,
\begin{align*}
 \norm{R_t}_{\Hh}
 \le C\int_0^t\norm{f_s}_{L^2}\,\dd s
 +C\int_0^t(1+\Vv(X_s))\norm{R_s}_{\Hh}\,\dd s.
\end{align*}
Gronwall's inequality therefore gives
\begin{align*}
 \sup_{t\le T}\norm{R_t}_{\Hh}
 \le C\exp\left\{
 C\int_0^T(1+\Vv(X_s))\,\dd s
 \right\}
 \int_0^T\norm{f_s}_{L^2}\,\dd s.
\end{align*}
By Cauchy--Schwarz,
\begin{align*}
 \left(\int_0^T\norm{f_s}_{L^2}\,\dd s\right)^2
 \le T\norm f_{L^2_tL^2_x}^2.
\end{align*}
Since $\Vv\ge1$, the factor $T$ can be absorbed into the exponential
after increasing $C$.  This proves
\eqref{eq:linear-variation-energy-bound}.
\end{proof}

The following lemma proves the Malliavin differentiability with moment bounds. 

\begin{lemma}
\label{gap:lem:fractional-twojet}
There is  $p_0>0$ with the following property.
For every $T<\infty$ there exist measurable functions
\begin{align*}
 \mathcal C_T:\Hh\times E_T\to[1,\infty),
 \qquad
 r_T:\Hh\times E_T\to(0,1],
\end{align*}
such that, for every $(\mathsf x,\omega)\in\Hh\times E_T$, the map
\begin{align*}
 h\longmapsto\Phi_T(\mathsf x,\omega+h)
\end{align*}
is twice continuously Fr\'echet differentiable on $\norm h_{H_T}<r_T(\mathsf x,\omega)$ and
\begin{align*}
 \sup_{\norm h_{H_T}<r_T(\mathsf x,\omega)}
 \left(\norm{\mathcal D\Phi_T(\mathsf x,\omega+h)}
 +\norm{\mathcal D^2\Phi_T(\mathsf x,\omega+h)}\right)
 \le\mathcal C_T(\mathsf x,\omega).
\end{align*}
Moreover, for every $R<\infty$,
\begin{align}\label{gap:eq:fractional-twojet-moment}
 \sup_{\Vv(\mathsf x)\le R}
 \E\left[\mathcal C_T(\mathsf x)^{p_0}+r_T(\mathsf x)^{-p_0}\right]
 \le C_{T,R}<\infty.
\end{align}
\end{lemma}

\begin{proof}
Put
\begin{align*}
 \Xi_T=T+\int_0^T\Vv(X_t)\,\dd t.
\end{align*}
Let $Y=(\widetilde{u},\widetilde{v})$ be the solution associated with $(\mathsf x,\omega+h)$ and set
$Z=Y-X$.  As long as $\sup_{t\le T}\norm{Z_t}_{\Hh}\le1,$ 
the difference equation and the same  estimate as in
\Cref{lem:variational-energies} give
\begin{align*}
 \sup_{t\le T}\norm{Z_t}_{\Hh}
 \le C\e^{C\Xi_T}\norm h_{H_T}.
\end{align*}
Thus, choosing $c>0$ sufficiently small and setting
\begin{align}\label{gap:eq:fractional-radius-choice}
 r_T(\mathsf x,\omega)=1\wedge c\e^{-C\Xi_T},
\end{align}
one has for $\norm h_{H_T}<r_T(\mathsf x,\omega)$,
\begin{align*}
 \sup_{t\le T}\norm{Z_t}_{\Hh}<1,
\end{align*}
which combined with \eqref{eq:periodic-V-equivalence}  gives 
\begin{align}\label{gap:eq:perturbed-occupation-control}
 \int_0^T\bigl(1+\Vv(Y_t)\bigr)\,\dd t
 \le C\left(T+\int_0^T\Vv(X_t)\,\dd t\right)
 =C\Xi_T.
\end{align}

Let $R=(r,\partial_t r)$ be the first variation of $Y$ in a
Cameron--Martin direction $k\in H_T$.  It has zero initial condition and
satisfies
\begin{align*}
 \partial_t^2r+\Lambda^2r+a\partial_tr+(3\widetilde{u}^2-1)r
 =B\partial_tk,
\end{align*}
where $\widetilde{u}$ denotes the displacement component of $Y$.  Applying
\Cref{lem:variational-energies} along $Y$ and using
\eqref{gap:eq:perturbed-occupation-control} gives
\begin{align}\label{gap:eq:fractional-first-variation}
 \sup_{t\le T}\norm{R_t}_{\Hh}
 \le C\e^{C\Xi_T}\norm k_{H_T}.
\end{align}

For two first variations $R_i=(r_i,\partial_t r_i)$, the mixed second
variation $Q=(q,\partial_tq)$ has zero initial condition and solves
\begin{align*}
 \partial_t^2q+\Lambda^2q+a\partial_tq+(3\widetilde{u}^2-1)q
 =-6\widetilde{u}r_1r_2.
\end{align*}
Since $H^1\hookrightarrow L^6$,
\begin{align*}
 \norm{\widetilde{u}r_1r_2}_{L^2}
 \le C\Vv(Y)^{1/2}
 \norm{R_1}_{\Hh}\norm{R_2}_{\Hh}.
\end{align*}
Hence \eqref{gap:eq:perturbed-occupation-control} and
\eqref{gap:eq:fractional-first-variation} give
\begin{align*}
 \norm{\widetilde{u}r_1r_2}_{L^2(0,T;L^2)}
 \le C\e^{C\Xi_T}\norm{k_1}_{H_T}\norm{k_2}_{H_T},
\end{align*}
after increasing $C$.  A second application of
\Cref{lem:variational-energies} therefore yields
\begin{align}\label{gap:eq:fractional-second-variation}
 \sup_{t\le T}\norm{Q_t}_{\Hh}
 \le C\e^{C\Xi_T}\norm{k_1}_{H_T}\norm{k_2}_{H_T}.
\end{align}

The preceding estimates are uniform for all
$\norm h_{H_T}<r_T(\mathsf x,\omega)$.  The usual difference-quotient argument identifies the preceding
variations with the first and second Fr\'echet derivatives in the Cameron--Martin variable.
Standard continuous dependence for the perturbed equation and its
variational equations shows that these derivatives depend continuously
on $h$ in the ball \eqref{gap:eq:fractional-radius-choice}.  Hence
\begin{align*}
 h\longmapsto\Phi_T(\mathsf x,\omega+h)
\end{align*}
is twice continuously Fr\'echet differentiable there, and
\eqref{gap:eq:fractional-first-variation}--\eqref{gap:eq:fractional-second-variation}
hold uniformly on this ball. We may therefore take
\begin{align*}
 \mathcal C_T(\mathsf x,\omega)=C\e^{C\Xi_T}
\end{align*}
by increasing $C$ if necessary. Then by \eqref{gap:eq:fractional-radius-choice}, 
\begin{align}\label{gap:eq:twojet-exponential-envelope}
 \mathcal C_T(\mathsf x,\omega)
 +r_T(\mathsf x,\omega)^{-1}
 \le C\e^{C\Xi_T}.
\end{align}
Both quantities are measurable by the Borel dependence of the solution
on $(\mathsf x,\omega)$.  All statements up to this point are
deterministic; only the moment estimate below uses Wiener measure.

Let $r_*>0$ be the exponent in \Cref{gap:lem:occupation} and choose
$p_0>0$ so small that
\begin{align*}
 Cp_0<r_*.
\end{align*}
Applying \Cref{gap:lem:occupation} with terminal exponent zero and
occupation exponent $Cp_0$ gives, for every $T,R<\infty$,
\begin{align*}
 \sup_{\Vv(\mathsf x)\le R}
 \E\exp\left\{
 Cp_0\int_0^T\Vv(X_t)\,\dd t
 \right\}<\infty.
\end{align*}
Since
\begin{align*}
 \e^{Cp_0\Xi_T}
 =\e^{Cp_0T}
 \exp\left\{
 Cp_0\int_0^T\Vv(X_t)\,\dd t
 \right\},
\end{align*}
\eqref{gap:eq:twojet-exponential-envelope} proves
\eqref{gap:eq:fractional-twojet-moment}.  The constant $C$ above and the
exponent $r_*$ are independent of $T$ and $R$, so the same $p_0$ works
for every finite $T$ and $R$.
\end{proof}

The next lemma gives deterministic core bounds and weak continuity with
respect to the Wiener driver.  Set
\begin{align}\label{eq:VR}
 \mathbb V_R=\{\mathsf x\in\Hh:\Vv(\mathsf x)\le R\}.
\end{align}

\begin{lemma}

Fix $0<s<1/2$ and $T,R<\infty$.  For every bounded $G\subset E_T$,
\begin{align}\label{gap:eq:pathwise-core-bounds}
 \sup_{(\mathsf x,\omega)\in\mathbb V_R\times G}
 \left[
 \sup_{t\le T}\Vv(X_t(\mathsf x,\omega))
 +\norm{u}_{L^5(0,T;L^{10})}
 \right]<\infty.
\end{align}
Moreover, if
$(\mathsf x_n,\omega_n),(\mathsf x,\omega)\in\mathbb V_R\times G$
satisfy
\begin{align*}
 \mathsf x_n\to\mathsf x\quad\text{in }\Es{s},
 \qquad
 \omega_n\to\omega\quad\text{in }E_T,
\end{align*}
then
\begin{align}\label{gap:eq:pathwise-weak-continuity}
 X(\mathsf x_n,\omega_n)
 \longrightarrow X(\mathsf x,\omega)
 \qquad\text{in }C([0,T];\Es{s}).
\end{align}
Consequently, for every $\delta>0$, the displacement components satisfy
\begin{align}\label{gap:eq:pathwise-interpolated-continuity}
 u_n\longrightarrow u
 \qquad\text{in }C([0,T];H^{1-\delta}).
\end{align}
\end{lemma}

\begin{proof}
Write $\bar X=(u,\bar v)$, where $\bar v=v-B\omega$, and set 
$\mathbb G=\mathbb A\mathbb B$.  The shifted equation \eqref{eq:pathwise-shifted-wave} has the deterministic
mild form
\begin{align}\label{gap:eq:pathwise-shifted-mild}
 \bar X_t
 =S_a(t)\mathsf x
 +\int_0^tS_a(t-r)\mathbb F(\bar X_r)\,\dd r
 +\int_0^tS_a(t-r)\mathbb G\omega(r)\,\dd r,
 \qquad
 \mathbb G\xi=\binom{B\xi}{-aB\xi}.
\end{align}
Set $f_\omega=B\omega$ and
\begin{align*}
 \bar{\cE}(u,\bar v)
 =1+\frac12\norm{\nabla u}_{L^2}^2
 +\frac12\norm{\bar v}_{L^2}^2
 +\frac14\norm{u}_{L^4}^4.
\end{align*}
The energy identity for the shifted equation gives
\begin{align*}
 \frac{\dd}{\dd t}\bar{\cE}(u,\bar v)
 =-\ip{a\bar v}{\bar v}_{L^2}
 +\ip{\nabla u}{\nabla f_\omega}_{L^2}
 -\ip{a\bar v}{f_\omega}_{L^2}
 +\ip{u^3}{f_\omega}_{L^2}\le C_G\bar{\cE}(u,\bar v),
\end{align*}
where $C_G<\infty$ since $G$ is bounded in $E_T$ and $B$ has smooth
finite dimensional range.  Gronwall's inequality, the relation
$v=\bar v+f_\omega$, and \eqref{eq:periodic-V-equivalence} therefore give
\begin{align*}
 \sup_{(\mathsf x,\omega)\in\mathbb V_R\times G}
 \sup_{t\le T}\Vv(X_t(\mathsf x,\omega))<\infty.
\end{align*}

For $\omega\in E_T$, define the pathwise linear convolution by
\begin{align*}
 Z_t^\omega
 =\mathbb B\omega(t)
 +\int_0^tS_a(t-r)\mathbb G\omega(r)\,\dd r.
\end{align*}
For the canonical Brownian motion, stochastic integration by parts and
$\mathbb G=\mathbb A\mathbb B$ show that
\begin{align*}
 Z_t^W=\int_0^tS_a(t-r)\mathbb B\,\dd W_r
 \qquad\text{almost surely}.
\end{align*}
For any fixed $\sigma\in(1/5,1)$, smoothness and finite dimensionality of
the range of $B$ give
\begin{align*}
 \sup_{\omega\in G}
 \norm{Z^\omega}_{C([0,T];\Hs{\sigma})}<\infty.
\end{align*}
Applying \eqref{eq:unit-block-Strichartz} to $X-Z^\omega$ on the finitely
many unit subintervals of $[0,T]$, and using
\begin{align*}
 \norm{u-u^3}_{L^2}
 \le \norm u_{L^2}+\norm u_{L^6}^3
 \le C\left(\norm u_{H^1}+\norm u_{H^1}^3\right)
 \le C\Vv(X)^{3/2}.
\end{align*}
proves the remaining bound in \eqref{gap:eq:pathwise-core-bounds}.

Let $\bar X_n$ and $\bar X$ correspond to
$(\mathsf x_n,\omega_n)$ and $(\mathsf x,\omega)$.  Since
\begin{align*}
 u_n^3-u^3=(u_n^2+u_nu+u^2)(u_n-u),
\end{align*}
\eqref{gap:eq:negative-trilinear} and
\eqref{gap:eq:pathwise-core-bounds} imply
\begin{align*}
 \norm{\mathbb F(\bar X_n(t))-\mathbb F(\bar X(t))}_{\Es{s}}
 \le C_{T,R,G}
 \norm{\bar X_n(t)-\bar X(t)}_{\Es{s}}.
\end{align*}
Subtracting \eqref{gap:eq:pathwise-shifted-mild} for the two solutions and
applying Gronwall's inequality gives
\begin{align*}
 \sup_{t\le T}\norm{\bar X_n(t)-\bar X(t)}_{\Es{s}}
 \le C_{T,R,G}\left(
 \norm{\mathsf x_n-\mathsf x}_{\Es{s}}
 +\norm{\omega_n-\omega}_{E_T}\right).
\end{align*}
Since $v=\bar v+B\omega$, this proves
\eqref{gap:eq:pathwise-weak-continuity}.  If $0<\delta<1+s$, interpolation
and the uniform energy bound give
\begin{align*}
 \sup_{t\le T}\norm{u_n(t)-u(t)}_{H^{1-\delta}}
 \le C_{T,R,G}
 \left(
 \sup_{t\le T}\norm{u_n(t)-u(t)}_{H^{-s}}
 \right)^{\delta/(1+s)}
 \longrightarrow0.
\end{align*}
For $\delta\ge1+s$, the same conclusion follows directly from
$H^{-s}\hookrightarrow H^{1-\delta}$.  This proves
\eqref{gap:eq:pathwise-interpolated-continuity}.
\end{proof}

We now verify part~\textup{(ii)} of
\Cref{ass:QS-stable-compact} and then complete part~\textup{(iii)}
using the preceding two-jet estimates.

\begin{lemma}
\label{gap:lem:weak-core-operators}
Fix $0<s<1/2$, $T,R<\infty$, and a compact $G\subset E_T$.  Let the
sublevel $\mathbb V_R$ defined in \eqref{eq:VR} carry the $\Es{s}$-topology.
Then
\begin{align*}
 (\mathsf x,\mathsf y,\omega)&\longmapsto
 \mathcal K_T(\mathsf x,\mathsf y,\omega)\in\cL(\Es{s}),\\
 (\mathsf y,\omega)&\longmapsto
 \mathcal D\Phi_T(\mathsf y,\omega)\in\cL(H_T,\Es{s})\nonumber
\end{align*}
are operator norm continuous on $\mathbb V_R^2\times G$ and
$\mathbb V_R\times G$, respectively.  Both maps are strongly measurable
on their full parameter spaces.
\end{lemma}

\begin{proof}
Let $X=(u,v)$ and $\widetilde X=(\widetilde u,\widetilde v)$ be the two
base solutions.  The weak propagator estimate used in the proof of
\Cref{prop:wave-stable-compact-structure} gives
\begin{align}\label{gap:eq:weak-variation-bound}
 \sup_{t\le T}
 \norm{\mathcal J_t(\mathsf x,\mathsf y,\omega)}_{\cL(\Es{s})}
 \le C_T\exp\left\{
 C\int_0^T[1+\Vv(X_t)+\Vv(\widetilde X_t)]\,\dd t
 \right\}.
\end{align}
By \eqref{gap:eq:pathwise-core-bounds}, this bound is uniform on
$\mathbb V_R^2\times G$.

Suppose
\begin{align*}
 (\mathsf x_n,\mathsf y_n,\omega_n)
 \longrightarrow(\mathsf x,\mathsf y,\omega)
 \qquad\text{in }\mathbb V_R^2\times G.
\end{align*}
By
\eqref{gap:eq:pathwise-weak-continuity}--\eqref{gap:eq:pathwise-interpolated-continuity},
for every $\delta>0$,
\begin{align}\label{gap:eq:base-interpolated-convergence}
 u_n\to u,\qquad
 \widetilde u_n\to\widetilde u
 \quad\text{in }C([0,T];H^{1-\delta}).
\end{align}
Fix $s<s_1<1/2$ and $0<\theta_0<1/5$, and choose
\begin{align*}
 0<\varepsilon<
 \min\left\{\frac12-s_1,
 \frac{\theta_0(s_1-s)}{1+s}\right\}.
\end{align*}
For sufficiently small
$0<\delta<\min\{\varepsilon,1/2-s_1\}$, the strict Sobolev multiplication
theorem gives
\begin{align}\label{gap:eq:lossy-coefficient}
 \norm{(f_n-f)g\,r}_{H^{-1-s_1-\varepsilon}}
 \le C\norm{f_n-f}_{H^{1-\delta}}
       \norm g_{H^1}\norm r_{H^{-s_1}},
\end{align}
whenever $f_n,f,g$ are bounded in $H^1$.  Indeed, this follows by first
multiplying $f_n-f$ and $g$, and then testing the resulting product
against an element of $H^{1+s_1+\varepsilon}$.

Set
\begin{align*}
 D_n
 =\mathcal K_T(\mathsf x_n,\mathsf y_n,\omega_n)
 -\mathcal K_T(\mathsf x,\mathsf y,\omega).
\end{align*}
The difference of the corresponding coefficients
$\mathscr W_n-\mathscr W$ is a finite sum of terms of the form
$(f_n-f)g$.  Subtracting the two exact difference propagator equations,
using \eqref{gap:eq:lossy-coefficient},
\eqref{gap:eq:base-interpolated-convergence}, and
\eqref{gap:eq:weak-variation-bound} at level $s_1+\varepsilon$, gives
\begin{align*}
 \norm{D_n}_{\cL(\Es{s_1},\Es{s_1+\varepsilon})}
 \longrightarrow0.
\end{align*}
On the other hand, the energy-endpoint regularity estimate from the proof
of \Cref{prop:wave-stable-compact-structure}, its uniformity on
$\mathbb V_R^2\times G$, and the triangle inequality give
\begin{align*}
 \sup_{n\ge1}
 \norm{D_n}_{\cL(\Es{-1},\Es{-1-\theta_0})}<\infty.
\end{align*}
Interpolating these two bounds with
\begin{align*}
 \lambda=\frac{1+s}{1+s_1},
 \qquad
 \gamma=(1-\lambda)\theta_0
 =\theta_0\frac{s_1-s}{1+s_1},
\end{align*}
we obtain
\begin{align*}
 \norm{D_n}_{\cL(\Es{s},\Es{s-\gamma+\lambda\varepsilon})}
 \longrightarrow0.
\end{align*}
The choice of $\varepsilon$ gives
$\gamma-\lambda\varepsilon>0$, and hence
\begin{align*}
 \Es{s-\gamma+\lambda\varepsilon}\hookrightarrow\Es{s}.
\end{align*}
This proves operator norm continuity of $\mathcal K_T$.

It remains to consider $\mathcal D\Phi_T$.  Suppose
$(\mathsf y_n,\omega_n)\to(\mathsf y,\omega)$ in
$\mathbb V_R\times G$.  For $\norm h_{H_T}\le1$, write
\begin{align*}
 R_n(t)&=(r_n(t),\partial_t r_n(t))
 =\mathcal D\Phi_t(\mathsf y_n,\omega_n)h,\\
 R(t)&=(r(t),\partial_t r(t))
 =\mathcal D\Phi_t(\mathsf y,\omega)h.
\end{align*}
By \eqref{gap:eq:pathwise-core-bounds} and
\Cref{lem:variational-energies}, the families $R_n$ and $R$ are uniformly
bounded in $C([0,T];\Hh)$.  Their difference solves the linearized equation
along the limiting base solution with source
\begin{align*}
 -3\bigl[(\widetilde u_n)^2-\widetilde u^2\bigr]r_n
 =-3(\widetilde u_n-\widetilde u)
   (\widetilde u_n+\widetilde u)r_n.
\end{align*}
Choose $0<\delta<\delta'<1$.  The strict three-factor Sobolev product
estimate gives
\begin{align*}
 &\norm{
 (\widetilde u_n-\widetilde u)
 (\widetilde u_n+\widetilde u)r_n
 }_{H^{-\delta'}}\\
 &\hspace{20mm}\le
 C\norm{\widetilde u_n-\widetilde u}_{H^{1-\delta}}
 \norm{\widetilde u_n+\widetilde u}_{H^1}
 \norm{r_n}_{H^1}.
\end{align*}
It follows from
\eqref{gap:eq:pathwise-interpolated-continuity} and the preceding uniform
bounds that
\begin{align*}
 \sup_{\norm h_{H_T}\le1}
 \norm{
 (\widetilde u_n-\widetilde u)
 (\widetilde u_n+\widetilde u)r_n
 }_{L^1(0,T;H^{-\delta'})}
 \longrightarrow0.
\end{align*}
The Duhamel--Gronwall estimate at level $\Hs{-\delta'}$ therefore gives
\begin{align*}
 \sup_{\norm h_{H_T}\le1}
 \norm{
 [\mathcal D\Phi_T(\mathsf y_n,\omega_n)
 -\mathcal D\Phi_T(\mathsf y,\omega)]h
 }_{\Hs{-\delta'}}
 \longrightarrow0.
\end{align*}
Since $\delta'<1<1+s$,
\begin{align*}
 \Hs{-\delta'}\hookrightarrow\Hs{-1-s}=\Es{s},
\end{align*}
which proves operator norm continuity of $\mathcal D\Phi_T$.

Strong measurability of the operator-valued maps follows from the Borel dependence of the pathwise solution and the corresponding variational equations.
\end{proof}

The next lemma verifies the regular-dependence requirements in
\Cref{ass:QS-stable-compact}\textup{(iii)}. Recall that
$\iota_T:H_T\hookrightarrow E_T$ is the canonical Cameron--Martin
embedding, which we suppress in shifts by identifying $H_T$ with its
image in $E_T$.

\begin{lemma}\label{gap:lem:wave-twojet-weak}
There is $p_0>0$ such that, for every $T,R<\infty$, the measurable
functions $\mathcal C_T,r_T$ in \Cref{gap:lem:fractional-twojet} satisfy
\begin{align}\label{gap:eq:wave-random-Taylor}
 \norm{\Phi_T(\mathsf x,\omega+h)-\Phi_T(\mathsf x,\omega)
 -\mathcal D\Phi_T(\mathsf x,\omega)h}_{\Es{s}}
 \le \mathcal C_T(\mathsf x,\omega)\norm h_{H_T}^2
\end{align}
for $\mathsf x\in\mathbb V_R$, $\gamma_T$-almost every $\omega$, and
$\norm h_{H_T}\le r_T(\mathsf x,\omega)$, and
\begin{align}\label{gap:eq:wave-QS3-moment}
 \sup_{(\mathsf x,\mathsf y)\in\mathbb V_R^2}
 \E\left[
 \norm{\mathcal J_T(\mathsf x,\mathsf y)}_{\cL(\Es{s})}^{p_0}
 +\norm{\mathcal D\Phi_T(\mathsf x)}_{\cL(H_T,\Es{s})}^{p_0}
 +\mathcal C_T(\mathsf x)^{p_0}
 +r_T(\mathsf x)^{-p_0}
 \right]<\infty.
\end{align}
Moreover, for every finite dimensional
$F\subset\iota_T^*(E_T^*)$ and every compact $G\subset E_T$, there are
$C=C(T,R,G,F)<\infty$ and $r=r(T,R,G,F)>0$ such that
\begin{align}\label{gap:eq:wave-driver-Taylor}
 \norm{\Phi_T(\mathsf x,\omega+h)-\Phi_T(\mathsf x,\omega)
 -\mathcal D\Phi_T(\mathsf x,\omega)h}_{\Es{s}}
 \le C\norm h_{H_T}^2
\end{align}
for $(\mathsf x,\omega)\in\mathbb V_R\times G$, $h\in F$, and
$\norm h_{H_T}\le r$.
\end{lemma}

\begin{proof}
By \Cref{gap:lem:fractional-twojet}, Taylor's formula and the continuous
embedding $\Hh\hookrightarrow\Es{s}$ give
\eqref{gap:eq:wave-random-Taylor}; the boundary case follows by
continuity.  The same lemma also gives
\begin{align*}
 \norm{\mathcal D\Phi_T(\mathsf x,\omega)}_{\cL(H_T,\Es{s})}
 \le C\mathcal C_T(\mathsf x,\omega)
\end{align*}
and the required moments of $\mathcal C_T$ and $r_T^{-1}$.  The weak
propagator estimate \eqref{gap:eq:weak-variation-bound}, followed by
\Cref{gap:lem:occupation} with a sufficiently small exponent, gives
\begin{align*}
 \sup_{(\mathsf x,\mathsf y)\in\mathbb V_R^2}
 \E\norm{\mathcal J_T(\mathsf x,\mathsf y)}_{\cL(\Es{s})}^{p_0}<\infty.
\end{align*}
Decreasing $p_0$ once if necessary proves
\eqref{gap:eq:wave-QS3-moment}.

It remains to prove \eqref{gap:eq:wave-driver-Taylor}.  By
\eqref{gap:eq:pathwise-core-bounds},
\begin{align*}
 L:=
 \sup_{(\mathsf x,\omega)\in\mathbb V_R\times G}
 \left(T+\int_0^T\Vv(X_t(\mathsf x,\omega))\,\dd t\right)<\infty.
\end{align*}
The deterministic construction in the proof of
\Cref{gap:lem:fractional-twojet}, in particular
\eqref{gap:eq:fractional-radius-choice} and
\eqref{gap:eq:twojet-exponential-envelope}, therefore gives
$C=C(T,R,G)<\infty$ and $r=r(T,R,G)>0$ such that
\begin{align*}
 \sup_{(\mathsf x,\omega)\in\mathbb V_R\times G}
 \sup_{\norm h_{H_T}<2r}
 \norm{\mathcal D^2\Phi_T(\mathsf x,\omega+h)}
 _{\cL^2(H_T\times H_T,\Es{s})}\le C.
\end{align*}
Taylor's formula then gives
\eqref{gap:eq:wave-driver-Taylor} for every $h\in H_T$ with
$\norm h_{H_T}\le r$, and hence in particular for $h\in F$.
\end{proof}

\section{Useful probabilistic lemmas}\label{app:probabilistic}

The following elementary pathwise separation lemma is based on Brownian quadratic variation.  The exceptional set is fixed before the coefficients are chosen, so the coefficients may depend anticipatively on the entire Brownian path.

\begin{lemma}[Brownian coefficient separation]
\label{lem:Brownian-affine-separation}
There is a Borel set
$\Omega_{\mathrm{sep}}\subset C_0([0,\infty);\R^m)$ of full Wiener measure
such that the following holds.  Fix $W\in\Omega_{\mathrm{sep}}$ and a
nondegenerate interval $I=[a,b]$ with rational endpoints.  Let
$c_1,\ldots,c_m$ be absolutely continuous on $I$ and let
$c_0\in C^\alpha(I)$ for some $\alpha>1/2$.  If
\begin{align}\label{eq:affine-zero-identity}
 c_0(t)+\sum_{j=1}^m c_j(t)W^j(t)=0,
 \qquad t\in I,
\end{align}
then $c_j=0$ for every $0\le j\le m$.
All coefficients may depend anticipatively on the entire Brownian path.
\end{lemma}

\begin{proof}
Fix a nondegenerate rational interval $I=[a,b]$, set
$t_i^n=a+i(b-a)2^{-n}$, and write
$\Delta_i^nW=W_{t_{i+1}^n}-W_{t_i^n}$.  For each coordinate $W^j$,
\begin{align*}
 M_{n,r}^j
 =\sum_{i<r}\left((\Delta_i^nW^j)^2-(b-a)2^{-n}\right),
 \qquad0\le r\le2^n,
\end{align*}
is a discrete martingale and
\begin{align*}
 \E\abs{M_{n,2^n}^j}^2=2(b-a)^2\,2^{-n}.
\end{align*}
Doob's inequality and Borel--Cantelli therefore give
\begin{align*}
 \max_{r\le2^n}\abs{M_{n,r}^j}\longrightarrow0
 \qquad\text{almost surely}.
\end{align*}
The distribution functions of the signed measures
\begin{align*}
 \sum_i\left((\Delta_i^nW^j)^2-(b-a)2^{-n}\right)
 \delta_{t_i^n}
\end{align*}
therefore converge uniformly to zero.  Summation by parts, first against
piecewise linear functions and then by uniform approximation, yields
\begin{align*}
 \sum_i(\Delta_i^nW^j)^2\delta_{t_i^n}
 \rightharpoonup\dd t.
\end{align*}
Applying the same argument to $(W^j+W^k)/\sqrt2$ and
$(W^j-W^k)/\sqrt2$ and polarising gives
\begin{align}\label{eq:weighted-qv-weak-limit}
 \sum_i\Delta_i^nW^j\Delta_i^nW^k\delta_{t_i^n}
 \rightharpoonup\delta_{jk}\,\dd t,
\end{align}
where $\delta_{jk}=1$ for $j=k$ and vanishes otherwise.  Moreover,
\begin{align}\label{eq:qv-pathwise-bounds}
 \sup_n\sum_i\abs{\Delta_i^nW^k}^2<\infty,
 \qquad
 \max_i\abs{\Delta_i^nW}\longrightarrow0.
\end{align}
Taking the countable intersection over rational intervals and coordinate
pairs produces a Borel full measure set $\Omega_{\mathrm{sep}}$.  Once
$W\in\Omega_{\mathrm{sep}}$ is fixed,
\eqref{eq:weighted-qv-weak-limit} is ordinary weak convergence of
deterministic finite signed measures and therefore holds against every
continuous test function simultaneously, including anticipatively chosen
ones.

Fix now a nondegenerate rational subinterval $J=[r,s]\subset I$ and let
\begin{align*}
 t_i^{n,J}=r+i(s-r)2^{-n},
 \qquad0\le i\le2^n,
\end{align*}
be its dyadic partition, with mesh
$\abs{\pi_n(J)}=(s-r)2^{-n}$.  Write
$\Delta_i^{n,J}f=f(t_{i+1}^{n,J})-f(t_i^{n,J})$.  Taking discrete
covariation of \eqref{eq:affine-zero-identity} with $W^k$ along this
partition, use
\begin{align*}
 \Delta_i^{n,J}(c_jW^j)
 =c_j(t_i^{n,J})\Delta_i^{n,J}W^j
 +W^j(t_{i+1}^{n,J})\Delta_i^{n,J}c_j.
\end{align*}
By \eqref{eq:weighted-qv-weak-limit},
\begin{align*}
 \sum_i c_j(t_i^{n,J})
 \Delta_i^{n,J}W^j\Delta_i^{n,J}W^k
 \longrightarrow
 \delta_{jk}\int_r^s c_j(t)\,\dd t.
\end{align*}
Since each $c_j$ is absolutely continuous and therefore has finite
variation,
\begin{align*}
 &\abs{\sum_iW^j(t_{i+1}^{n,J})
 \Delta_i^{n,J}c_j\,\Delta_i^{n,J}W^k}\\
 &\qquad\le
 \norm{W^j}_{C(J)}\operatorname{Var}_J(c_j)
 \max_i\abs{\Delta_i^{n,J}W^k}
 \longrightarrow0.
\end{align*}
Also, since $\alpha>1/2$, Cauchy--Schwarz and
\eqref{eq:qv-pathwise-bounds} give
\begin{align*}
 \abs{\sum_i\Delta_i^{n,J}c_0\,\Delta_i^{n,J}W^k}
 &\le
 \left(\sum_i\abs{\Delta_i^{n,J}c_0}^2\right)^{1/2}
 \left(\sum_i\abs{\Delta_i^{n,J}W^k}^2\right)^{1/2}\\
 &\le
 C_J\abs{\pi_n(J)}^{\alpha-1/2}
 \left(\sum_i\abs{\Delta_i^{n,J}W^k}^2\right)^{1/2}
 \longrightarrow0.
\end{align*}
It follows that
\begin{align*}
 \int_r^s c_k(t)\,\dd t=0
\end{align*}
for every rational $r<s$ in $I$.  Since $c_k$ is continuous, approximation
of arbitrary subintervals by rational ones shows that its integral
vanishes on every subinterval of $I$, and hence $c_k\equiv0$.  This holds
for every $1\le k\le m$, and \eqref{eq:affine-zero-identity} then gives
$c_0\equiv0$.
\end{proof}

Next we prove a finite Gaussian transformation estimate. 
Let $(E,H_W,\gamma)$ be an abstract Wiener space with Cameron-Martin embedding $\iota: H_W\to E$. 
\begin{lemma}
\label{lem:Gaussian-entropy}
Let $u:E\to H_W$ be a smooth cylindrical map whose range is contained in a
finite dimensional subspace $F\subset\iota^*(E^*)$.  Suppose
\begin{align*}
 \sup_\omega\norm{\mathcal D u(\omega)}_{\rm op}\le\frac12,
 \qquad
 \sup_\omega\left(
 \norm{u(\omega)}_{H_W}
 +\norm{\mathcal D u(\omega)}_{\rm HS}\right)\le a.
\end{align*}
Then $T(\omega)=\omega+\iota u(\omega)$ acts nontrivially only on a
finite dimensional Gaussian factor, where it is a global $C^1$
diffeomorphism.  In particular, $T_\#\gamma$ and $\gamma$ are equivalent
and the relative entropy 
\begin{align}\label{eq:Gaussian-entropy-bound}
 \Ent(T_\#\gamma\mid\gamma)\le Ca^2.
\end{align}
If $L_T=\frac{\dd T_\#\gamma}{\dd\gamma}$, then, for every
$1\le p<\infty$,
\begin{align}\label{eq:Gaussian-density-moments}
 \norm{L_T}_{L^p(\gamma)}
 \le C_{p,\dim F,a}<\infty.
\end{align}
In particular, if $a\le1$, the right-hand side depends only on
$p$ and $\dim F$.
\end{lemma}

\begin{proof}
Since $u$ is cylindrical, it depends on finitely many functionals
$\ell_1,\ldots,\ell_N\in E^*$.  Since
$F\subset\iota^*(E^*)$ is finite dimensional, there are
$\eta_1,\ldots,\eta_m\in E^*$, where $m=\dim F$, such that
\begin{align*}
 F=\spanop\{\iota^*\eta_1,\ldots,\iota^*\eta_m\}.
\end{align*}
Collecting the two finite families into
$\xi_1,\ldots,\xi_M\in E^*$ and setting
\begin{align*}
 H_0=\spanop\{\iota^*\xi_1,\ldots,\iota^*\xi_M\},
\end{align*}
we have $F\subset H_0$, and both the dependence and the range of $u$ are
contained in this finite Gaussian factor.

Choose an orthonormal basis of $H_0$ consisting of linear combinations
of the vectors $\iota^*\xi_j$.  The associated Gaussian coordinates
identify this factor with $(\R^d,\gamma_d)$, where $d=\dim H_0$.
Under this identification, $F$ is a fixed $m$-dimensional subspace of
$\R^d$ and $u$ is represented by a $C^1$ map
\begin{align*}
 U:\R^d\to F
\end{align*}
satisfying
\begin{align}\label{eq:Gaussian-finite-factor-bounds}
 \sup_z\norm{DU(z)}_{\rm op}\le\frac12,
 \qquad
 \sup_z\left(\abs{U(z)}+\norm{DU(z)}_{\rm HS}\right)\le a.
\end{align}
Thus it is enough to work on this finite dimensional factor.

Set $S(z)=z+U(z)$.  Since $U$ is $1/2$-Lipschitz, for every
$y\in\R^d$ the map
\begin{align*}
 z\longmapsto y-U(z)
\end{align*}
is a contraction.  Hence $S$ is bijective, and since
$\Id+DU(z)$ is invertible everywhere, $S$ is a global $C^1$
diffeomorphism.

Writing $y=S(z)$, the Gaussian change-of-variables formula gives
\begin{align}\label{eq:Gaussian-density-formula}
 L_T(y)
 =\frac{\exp\{\ip z{U(z)}+\frac12\abs{U(z)}^2\}}
 {\det(\Id+DU(z))}.
\end{align}
The determinant is positive, since
$\Id+tDU(z)$ is invertible for every $0\le t\le1$.
Therefore
\begin{align*}
 \Ent(T_\#\gamma\mid\gamma)
 &=\int_{\R^d}\left[
 \ip z{U(z)}+\frac12\abs{U(z)}^2
 -\log\det(\Id+DU(z))\right]\gamma_d(\dd z).
\end{align*}
Gaussian integration by parts gives
\begin{align*}
 \int_{\R^d}\ip z{U(z)}\,\gamma_d(\dd z)
 =\int_{\R^d}\Tr DU(z)\,\gamma_d(\dd z),
\end{align*}
while, for $\norm A_{\rm op}\le1/2$,
\begin{align*}
 \abs{\Tr A-\log\det(\Id+A)}
 \le C\norm A_{\rm HS}^2.
\end{align*}
Together with \eqref{eq:Gaussian-finite-factor-bounds}, this yields
\begin{align*}
 \Ent(T_\#\gamma\mid\gamma)
 \le C\int_{\R^d}
 \left(\abs{U(z)}^2+\norm{DU(z)}_{\rm HS}^2\right)
 \gamma_d(\dd z)
 \le Ca^2,
\end{align*}
which proves \eqref{eq:Gaussian-entropy-bound}.

It remains to prove the density moments.  Since
$\Ran DU(z)\subset F$, one has $\rank DU(z)\le m$.  Hence
\begin{align*}
 \det(\Id+DU(z))^{-1}\le2^m.
\end{align*}
Moreover, since $U(z)\in F$,
\begin{align*}
 \ip z{U(z)}=\ip{P_Fz}{U(z)},
\end{align*}
where $P_F$ denotes orthogonal projection onto $F$.  Thus
\eqref{eq:Gaussian-density-formula} gives
\begin{align*}
 L_T(Sz)
 \le2^m\exp\left\{a\abs{P_Fz}+\frac12a^2\right\}.
\end{align*}
For $p>1$, using $T_\#\gamma=L_T\gamma$,
\begin{align*}
 \int_E L_T^p\,\dd\gamma
 &=\int_E L_T(T\omega)^{p-1}\,\gamma(\dd\omega)\\
 &\le
 2^{m(p-1)}\e^{(p-1)a^2/2}
 \int_{\R^m}\e^{(p-1)a\abs z}\,\gamma_m(\dd z)
 <\infty.
\end{align*}
The resulting bound depends only on $p,m,a$.  The case $p=1$ is immediate,
which proves \eqref{eq:Gaussian-density-moments}.
\end{proof}

\section*{Statements and declarations}

\textbf{Competing interests}.
The authors declare that they have no competing interests.

\medskip

\textbf{Use of generative artificial intelligence}.
OpenAI ChatGPT  was used to assist with proof auditing and manuscript preparation. The authors take full responsibility for the mathematical content.


\begin{thebibliography}{99}


\bibitem{AgrachevSarychev2005}
A.~A. Agrachev and A.~V. Sarychev,
Navier--Stokes equations: controllability by means of low modes forcing,
\emph{J. Math. Fluid Mech.} \textbf{7} (2005), 108--152.

\bibitem{AgrachevSarychev2006}
A.~A. Agrachev and A.~V. Sarychev,
Controllability of $2$D Euler and Navier--Stokes equations by degenerate
forcing,
\emph{Comm. Math. Phys.} \textbf{265} (2006), 673--697.

\bibitem{BarbuDaPrato2002}
V. Barbu and G. Da Prato,
The stochastic nonlinear damped wave equation,
\emph{Appl. Math. Optim.} \textbf{46} (2002), 125--141.

\bibitem{Bogachev1998}
V.~I. Bogachev,
\emph{Gaussian Measures},
Mathematical Surveys and Monographs, vol.~62, American Mathematical
Society, Providence, RI, 1998.

\bibitem{BrzezniakOndrejatSeidler2016}
Z. Brze\'zniak, M. Ondrej\'at, and J. Seidler,
Invariant measures for stochastic nonlinear beam and wave equations,
\emph{J. Differential Equations} \textbf{260} (2016), 4157--4179.

\bibitem{CacciafestaDanesiMeng2024}
F. Cacciafesta, E. Danesi, and L. Meng,
Strichartz estimates for the half wave/Klein--Gordon and Dirac equations on
compact manifolds without boundary,
\emph{Math. Ann.} \textbf{389} (2024), 3009--3042.


\bibitem{DengHani2023}
Y. Deng and Z. Hani,
Full derivation of the wave kinetic equation,
\emph{Invent. Math.} \textbf{233} (2023), 543--724.

\bibitem{DengHani2026}
Y. Deng and Z. Hani,
Derivation of the wave kinetic equation: full range of scaling laws,
\emph{Mem. Amer. Math. Soc.} \textbf{320} (2026), no.~1631.

\bibitem{EckmannHairer2001}
J.-P. Eckmann and M. Hairer,
Uniqueness of the invariant measure for a stochastic PDE driven by
degenerate noise,
\emph{Comm. Math. Phys.} \textbf{219} (2001), 523--565.

\bibitem{EckmannRuelle1985}
J.-P. Eckmann and D. Ruelle,
Ergodic theory of chaos and strange attractors,
\emph{Rev. Modern Phys.} \textbf{57} (1985), 617--656.

\bibitem{ForlanoTolomeo2024}
L. Forlano and L. Tolomeo,
On the unique ergodicity for a class of $2$ dimensional stochastic wave equations,
\emph{Trans. Amer. Math. Soc.} \textbf{377} (2024), 345--394.

\bibitem{Frisch1995}
U. Frisch,
\emph{Turbulence: The Legacy of A. N. Kolmogorov},
Cambridge University Press, Cambridge, 1995.


\bibitem{GerasimovicsHairer2019}
A. Gerasimovi\v{c}s and M. Hairer,
H\"ormander's theorem for semilinear SPDEs,
\emph{Electron. J. Probab.} \textbf{24} (2019), Paper No.~132, 56 pp.

\bibitem{GlattHoltzHerzogMattingly2018}
N.~E. Glatt-Holtz, D.~P. Herzog, and J.~C. Mattingly,
Scaling and saturation in infinite-dimensional control problems with
applications to stochastic partial differential equations,
\emph{Ann. PDE} \textbf{4} (2018), no.~2, Paper No.~16, 103 pp.

\bibitem{GrandeHani2026}
R. Grande and Z. Hani,
Rigorous derivation of damped--driven wave turbulence theory,
\emph{Arch. Ration. Mech. Anal.} \textbf{250} (2026), Paper No.~27.

\bibitem{HairerMattingly2006}
M. Hairer and J.~C. Mattingly,
Ergodicity of the $2$D Navier--Stokes equations with degenerate stochastic
forcing,
\emph{Ann. of Math. (2)} \textbf{164} (2006), 993--1032.

\bibitem{HairerMattingly2008}
M. Hairer and J.~C. Mattingly,
Spectral gaps in Wasserstein distances and the $2$D stochastic Navier--Stokes equations,
\emph{Ann. Probab.} \textbf{36} (2008), 2050--2091.

\bibitem{HairerMattingly2011}
M. Hairer and J.~C. Mattingly,
A theory of hypoellipticity and unique ergodicity for semilinear stochastic
PDEs,
\emph{Electron. J. Probab.} \textbf{16} (2011), 658--738.

\bibitem{HairerMattinglyScheutzow2011}
M. Hairer, J.~C. Mattingly, and M. Scheutzow,
Asymptotic coupling and a general form of Harris' theorem with applications
to stochastic delay equations,
\emph{Probab. Theory Related Fields} \textbf{149} (2011), 223--259.

\bibitem{KatoPonce1988}
T. Kato and G. Ponce,
Commutator estimates and the Euler and Navier--Stokes equations,
\emph{Comm. Pure Appl. Math.} \textbf{41} (1988), 891--907.

\bibitem{Kim2004}
J.~U. Kim,
Periodic and invariant measures for stochastic wave equations,
\emph{Electron. J. Differential Equations} \textbf{2004} (2004),
No.~05, 1--30.

\bibitem{Kolmogorov1941a}
A.~N. Kolmogorov,
The local structure of turbulence in incompressible viscous fluid for very large Reynolds numbers,
\emph{Dokl. Akad. Nauk SSSR} \textbf{30} (1941), 299--303;
English transl., \emph{Proc. Roy. Soc. London Ser. A} \textbf{434} (1991), 9--13.

\bibitem{Kolmogorov1941b}
A.~N. Kolmogorov,
Dissipation of energy in locally isotropic turbulence,
\emph{Dokl. Akad. Nauk SSSR} \textbf{32} (1941), 16--18;
English transl., \emph{Proc. Roy. Soc. London Ser. A} \textbf{434} (1991), 15--17.

\bibitem{KuksinNersesyanShirikyan2020a}
S. Kuksin, V. Nersesyan, and A. Shirikyan,
Exponential mixing for a class of dissipative PDEs with bounded degenerate
noise,
\emph{Geom. Funct. Anal.} \textbf{30} (2020), 126--187.

\bibitem{KuksinNersesyanShirikyan2020b}
S. Kuksin, V. Nersesyan, and A. Shirikyan,
Mixing via controllability for randomly forced nonlinear dissipative PDEs,
\emph{J. \'Ec. polytech. Math.} \textbf{7} (2020), 871--896.

\bibitem{Lebeau1996}
G. Lebeau,
\'Equation des ondes amorties,
in \emph{Algebraic and Geometric Methods in Mathematical Physics}
(Kaciveli, 1993), Math. Phys. Stud., vol.~19,
Kluwer Academic Publishers, Dordrecht, 1996, 73--109.

\bibitem{LianLiuLu2026}
Z. Lian, R. Liu, and K. Lu,
Unique ergodicity for the projective process of the $2$D Navier--Stokes
equation with nondegenerate noise,
preprint (2026), arXiv:2608.18075.

\bibitem{LiuWeiXiangZhangZhao2024}
Z. Liu, D. Wei, S. Xiang, Z. Zhang, and J.-C. Zhao,
Exponential mixing for random nonlinear wave equations: weak dissipation
and localized control,
preprint (2024), arXiv:2407.15058.

\bibitem{Martirosyan2014}
D. Martirosyan,
Exponential mixing for the white-forced damped nonlinear wave equation,
\emph{Evol. Equ. Control Theory} \textbf{3} (2014), 645--670.

\bibitem{NewellRumpf2011}
A.~C. Newell and B. Rumpf,
Wave turbulence,
\emph{Annu. Rev. Fluid Mech.} \textbf{43} (2011), 59--78.


\bibitem{Pazy1983}
A. Pazy,
\emph{Semigroups of Linear Operators and Applications to Partial Differential Equations},
Applied Mathematical Sciences, vol.~44,
Springer-Verlag, New York, 1983.


\bibitem{Ruelle1978}
D. Ruelle,
What are the measures describing turbulence?,
\emph{Progr. Theoret. Phys. Suppl.} \textbf{64} (1978), 339--345.


\bibitem{Sinai1989}
Ya.~G. Sinai,
Kolmogorov's work on ergodic theory,
\emph{Ann. Probab.} \textbf{17} (1989), no.~3, 833--839.

\bibitem{Tolomeo2020}
L. Tolomeo,
Unique ergodicity for a class of stochastic hyperbolic equations with
additive space--time white noise,
\emph{Comm. Math. Phys.} \textbf{377} (2020), 1311--1347.

\bibitem{ZakharovLvovFalkovich1992}
V.~E. Zakharov, V.~S. L'vov, and G. Falkovich,
\emph{Kolmogorov Spectra of Turbulence I: Wave Turbulence},
Springer Series in Nonlinear Dynamics, Springer-Verlag, Berlin, 1992.

\end{thebibliography}
\end{document}